\documentclass[times,final]{elsarticle}

\usepackage{jcomp}

\usepackage{framed,multirow}

\usepackage{amssymb,amsmath,amsthm,bm}
\usepackage{latexsym}
\usepackage{mathtools}

\usepackage{cleveref}
\usepackage{subcaption}
\usepackage{graphicx}

\usepackage{booktabs}
\usepackage{makecell}
\usepackage{enumitem} 

\usepackage{url}
\usepackage{xcolor}

\definecolor{newcolor}{rgb}{.8,.349,.1}

\makeatletter

\newcommand{\Rmnum}[1]{\uppercase\expandafter{\romannumeral #1}}
\makeatother

\newtheorem{theorem}{Theorem}[section]
\newtheorem{lemma}[theorem]{Lemma}

\newtheorem{corollary}{Corollary}[section]

\theoremstyle{definition}

\newtheorem{expl}{Example}[section]

\theoremstyle{remark}
\newtheorem{remark}{Remark}[]

\newcommand{\dd}{\mathrm{d}}
\newcommand{\dx}{\Delta x}
\newcommand{\dy}{\Delta y}

\newcommand{\dt}{\Delta t}

\newcommand{\halfone}{\frac{1}{2}}

\newcommand{\fourone}{\frac{1}{4}}
\newcommand{\mfourone}{\frac{m}{4}}
\newcommand{\sfourone}{\frac{\sigma}{4}}

\def\vec#1{{\bm #1}}

\def\jump#1{\llbracket #1 \rrbracket }
\def\average#1{\lbrace\!\lbrace #1  \rbrace\!\rbrace }

\allowdisplaybreaks

\journal{Journal of Computational Physics}
\begin{document}

\verso{Shengrong Ding, Kailiang Wu}

\begin{frontmatter}

\title{{\bf GQL-Based Bound-Preserving and Locally Divergence-Free Central Discontinuous Galerkin Schemes for Relativistic Magnetohydrodynamics}\tnoteref{tnote1}}
	

\tnotetext[tnote1]{This work is partially supported by Shenzhen Science and Technology Program (Grant No.~RCJC20221008092757098) and 
	National Natural Science Foundation of China (Grant No.~12171227 and No.~92370108).}

\author[1]{Shengrong {Ding}}
\ead{dingsr@sustech.edu.cn}
\author[2,1,3]{Kailiang {Wu}\corref{cor1}}
\cortext[cor1]{Corresponding author.}
\ead{wukl@sustech.edu.cn}

\address[1]{SUSTech International Center for Mathematics, Southern University of Science and Technology, Shenzhen 518055, China}
\address[2]{Department of Mathematics, Southern University of Science and Technology, Shenzhen 518055, China}
\address[3]{National Center for Applied Mathematics Shenzhen (NCAMS), Shenzhen 518055, China}


\begin{abstract}

This paper develops novel and robust central discontinuous Galerkin (CDG) schemes of arbitrarily high-order accuracy for special relativistic magnetohydrodynamics (RMHD) with a general equation of state (EOS). These schemes are provably bound-preserving (BP), meaning they consistently preserve the upper bound for subluminal fluid velocity and the positivity of density and pressure, while also (locally) maintaining the divergence-free (DF) constraint for the magnetic field. For 1D RMHD, the standard CDG method is exactly DF, and its BP property is proven under a condition achievable by the BP limiter. For 2D RMHD, we design provably BP and locally DF CDG schemes based on the suitable discretization of a modified RMHD system, which is the relativistic analogue of Godunov's symmetrizable form of the non-relativistic MHD system [S.~K.~Godunov, {\em Numerical Methods for Mechanics of Continuum Medium}, 1 (1972) 26--34]. A key novelty in our schemes is the meticulous discretization of additional source terms in the modified RMHD equations, so as to precisely counteract the influence of divergence errors on the BP property across overlapping meshes. Notably, we provide rigorous proofs of the BP property for our CDG schemes and first establish the theoretical connection between BP and discrete DF properties on overlapping meshes for RMHD. Owing to the absence of explicit expressions for primitive variables in terms of conserved variables, the constraints of physical bounds are strongly nonlinear, making the BP proofs highly nontrivial. We overcome these challenges through technical estimates within the geometric quasilinearization (GQL) framework [K. Wu \& C.-W. Shu, {\em SIAM Review}, 65 (2023) 1031--1073], which equivalently converts the nonlinear constraints into linear ones. Furthermore, we introduce a new 2D cell average decomposition on overlapping meshes, which relaxes the theoretical BP CFL constraint and reduces the number of internal nodes, thereby enhancing the efficiency of the 2D BP CDG method. Finally, we implement the proposed CDG schemes for extensive RMHD problems with various EOSs, demonstrating their robustness and effectiveness in challenging scenarios like ultra-relativistic blasts and jets in strongly magnetized environments.

\end{abstract}



\end{frontmatter}



\section{Introduction}

Relativistic magnetohydrodynamics (RMHD) integrates the principles of magnetohydrodynamics (MHD) with Einstein's theory of relativity and is a pivotal discipline in both astrophysics and plasma physics. This field primarily focuses on the dynamics of electrically conducting fluids, such as plasma, in the presence of magnetic fields, especially under conditions of exceedingly high velocities or strong gravitational forces. 
RMHD has become instrumental in studying a wide array of astrophysical phenomena across stellar and galactic scales. 

In the laboratory frame, the $d$-dimensional special RMHD equations can be formulated as
\begin{equation}\label{eq:RMHD}
	\bm{U}_t + \nabla \cdot \bm{F}(\bm{U})={\bf 0},
\end{equation} 
where $t$ represents the time, $d \in \{1, 2, 3\}$ specifies the dimensionality, the conservative vector $\bm{U} = \left(D, \bm{m}, \bm{B}, E \right)^{\top} \in \mathbb{R}^8$ comprises the mass density $D$, the momentum density vector $\bm{m}$, the magnetic field $\bm{B}$, and the energy density $E$. 
The divergence of the fluxes, $\nabla \cdot \bm{F}(\bm{U})=\sum_{i=1}^{d} \frac{\partial \bm{F}_i(\bm{U})}{\partial x_i}$, where each $\bm{F}_i(\bm{U})$ represents the flux in the $x_i$-direction, defined as
\begin{equation*}
	\bm{F}_i(\bm{U})=\Big( D v_i, v_i \bm{m} - B_i (W^{-2}\bm{B}+(\bm{v}\cdot \bm{B})\bm{v}) + p_{tot} \bm{e}_i, v_i \bm{B} - B_i \bm{v}, m_i \Big)^{\top}. 
\end{equation*}
Here,  $\bm{v}=(v_1,v_2,v_3)$ denotes the fluid velocity, $W = 1/\sqrt{1-|\bm{v}|^2}$ is the Lorentz factor, 
the total pressure $p_{tot} = p + p_m$ includes both thermal pressure $p$ and magnetic pressure $p_m = \frac{1}{2} ((1-|\bm{v}|^2) |\bm{B}|^2 + (\bm{v} \cdot \bm{B})^2)$, and the vector $\bm{e}_i$ represents the $i$th row of the identity matrix of size $3 \times 3$. 
The conservative vector $\bm{U} = \left(D, \bm{m}, \bm{B}, E \right)^{\top}$ can be explicitly expressed by the primitive variables $\bm W := (\rho,\bm v, \bm B, p)^\top$ via 
\begin{equation}\label{eq:W2U}
	\begin{cases}
		D = \rho W,\\
		m_i = \rho h W^2 v_i + |\bm{B}|^2 v_i - (\bm{v} \cdot \bm{B})B_i, \quad i \in \{1, 2, 3\},\\
		E = \rho h W^2 - p_{tot} + |\bm{B}|^2,
	\end{cases}
\end{equation}
where $\rho$ is the rest-mass density. The specific enthalpy $h = 1 + e + p/\rho$, where units are chosen such that the speed of light $c$ is equal to one (i.e., $c=1$), and $e$ is the specific internal energy.  
The system \eqref{eq:RMHD} is closed with an equation of state (EOS), which relates thermodynamic variables. 
The general form of the EOS is given by
\begin{equation}\label{eq:EOS}
	h=h(p,\rho) = 1+e(p,\rho)+p/{\rho}, 
\end{equation}
which must satisfy the inequality 
\begin{equation}\label{EOS:heq1}
	h(p,\rho) \geq \sqrt{1+p^2/{\rho}^2} + p/\rho,
\end{equation}
as implied by the relativistic kinetic theory \cite{WuTang2017ApJS}.
For the causal EOS, namely that the local speed of sound is slower than the speed of light $c$, the following relation holds: 
\begin{equation}\label{EOS:heq2}
	h\left(\frac{1}{\rho}-\frac{\partial h(p,\rho)}{\partial p}\right) < \frac{\partial h(p,\rho)}{\partial \rho} <0,
\end{equation}
if the fluid's thermal expansion coefficient is positive; see \cite{WuTang2017ApJS}.  
Throughout this paper, we consider a general EOS satisfying \eqref{EOS:heq1} and \eqref{EOS:heq2}, which includes a wide range of commonly used EOSs such as the ideal EOS in \eqref{EOS:IDEOS}.

Owing to relativistic effects, the flux functions $\bm{F}_i(\bm{U})$ and the primitive quantities $\bm W := (\rho,\bm v, \bm B, p)^\top$ cannot be explicitly expressed in terms of the conserved variables $\bm{U} = \left(D, \bm{m}, \bm{B}, E \right)^{\top}$. This unique characteristic, absent in non-relativistic MHD, renders the RMHD system highly nonlinear and complex, making its analytical analysis challenging. 
Consequently, numerical simulations have emerged as a vital approach for investigating the complex dynamics within RMHD. Nonetheless, the nonlinear hyperbolic nature of the RMHD equations implies that their solutions might develop discontinuities, like shocks, even under smooth initial conditions. Furthermore, the solutions of the RMHD system can manifest as large-scale structures or exhibit very strong discontinuities, particularly in the strongly magnetized or ultra-relativistic scenarios. 
These factors lead to the challenges of devising robust high-order accurate numerical schemes for RMHD.

The magnetic field $\bm{B}=(B_1,B_2,B_3)$ is subject to a divergence-free (DF) constraint:
\begin{equation}\label{eq:DF}
	\nabla \cdot \bm{B} := \sum_{i=1}^{d} \frac{\partial B_i}{\partial x_i} = 0,
\end{equation}
indicating the absence of magnetic monopoles in the system. If the initial magnetic field complies with this DF condition, the exact solutions of the RMHD equations always maintain zero divergence.
In the design of numerical RMHD schemes, careful consideration of the DF condition \eqref{eq:DF} is imperative, posing additional challenges beyond those typically encountered in solving nonlinear hyperbolic systems.
For both non-relativistic ideal MHD and RMHD, adherence to the DF condition \eqref{eq:DF} is critical for robust computations; significant divergence errors in the magnetic field can lead to nonphysical outcomes or numerical instabilities \cite{Brackbill1980, Evans1988, BalsaraSpicer1999, Toth2000, Li2005, Wu2017a, WuTangM3AS}. 
In the one-dimensional (1D) case, the DF condition \eqref{eq:DF} simplifies to ensuring that $B_1$ remains constant, which is straightforward. However, in multidimensional contexts, preserving this condition numerically becomes significantly more complex. 
To address this, various techniques have been developed either to reduce divergence errors or to enforce discrete analogs of the DF condition \eqref{eq:DF}. These techniques were primarily developed for the non-relativistic ideal MHD system; see, for example, \cite{Brackbill1980,Evans1988, powell1997approximate, BalsaraSpicer1999, Toth2000, Dedner2002, Gardiner2005, Li2005, Li2011, Li2012, Yakovlev2013, Fu2018} and the references therein.  
Among them, the eight-wave method by Powell and his collaborators \cite{Powell1995,powell1997approximate} aims to manage divergence errors through suitable discretization of a modified MHD system. In the non-relativistic case, this modified MHD system was first proposed by Godunov \cite{Godunov1972} in 1972 for entropy symmetrization.
To extend the eight-wave method and entropy-stable schemes to relativistic scenarios, the symmetrizable modification of the RMHD system was presented in \cite{WuShu2019SISC}, expressed as
 \begin{equation}\label{eq:ModRMHD}
 	\bm{U}_t + \nabla \cdot \bm{F}(\bm{U})=-(\nabla \cdot \bm{B}) \bm{S}(\bm{U}) ,
 \end{equation}
 where $\bm{S}(\bm{U})=\left(0, (1-|\bm{v}|^2)\bm{B}+ (\bm{v} \cdot \bm{B}) \bm{v}, \bm{v}, \bm{v}\cdot \bm{B}\right)^\top$. 
The symmetrization source terms $-(\nabla \cdot \bm{B}) \bm{S}(\bm{U})$ on the right-hand side of \eqref{eq:ModRMHD} are proportional to the magnetic divergence. Therefore, under the DF condition \eqref{eq:DF}, the original form \eqref{eq:RMHD} and the modified one \eqref{eq:ModRMHD} coincide at the continuous level.
However, these source terms not only render the modified system \eqref{eq:ModRMHD} symmetrizable, but also play a pivotal role in preserving the physical bound constraints \eqref{eq:G}, which will be elaborated upon subsequently.


Besides the difficulties in addressing the DF condition \eqref{eq:DF}, significant numerical challenges also arise from the need to preserve physical bounds. These include maintaining the positivity of density and thermal pressure (or internal energy), as well as capping the fluid velocity at the speed of light (note $c=1$). These constraints define the admissible state set:
\begin{equation}\label{eq:G}
	\mathcal{G} = \left\{
		\bm{U}=(D, \bm{m}, \bm{B}, E)^\top:~ D>0, ~ p(\bm{U}) >0, ~ |\bm{v}(\bm{U})|<1 \right\},
\end{equation}
where the functions $p(\bm{U})$ and $\bm{v}(\bm{U})$ are strongly nonlinear and have no explicit expressions (see \Cref{sec:2.1}). 
The development of bound-preserving (BP) numerical schemes is crucial not only for their physical relevance but also for numerical stability. Violating these bounds can disrupt the hyperbolic nature of the system, leading to severe numerical instability and simulation breakdowns. Indeed, schemes that lack the BP property often fail when simulating RMHD problems involving large Lorentz factors, strong shocks, high Mach numbers, low density, low plasma beta, or low pressure. Over recent decades, there has been notable progress in designing high-order BP numerical methods for hyperbolic conservation laws; see, for example,  \cite{zhang2010,zhang2010b,Xu2014,Xiong2016,christlieb2015positivity,ZHANG2017301,Wu2017a,WuShu2018,WuShu2019} and the surveys in \cite{zhang2011b,Shu2018}. These methods are primarily built on two types of limiters: the simple scaling limiter and the flux-correction limiter. Motivated by these, various BP numerical methods have been developed for special relativistic hydrodynamics (RHD) without magnetic fields, including BP finite difference \cite{WuTang2015}, central and non-central discontinuous Galerkin (DG) \cite{QinShu2016,WuTang2017ApJS}, and finite volume schemes \cite{LingDuanTang2019,ChenWu2022}. Frameworks for designing high-order accurate BP finite volume, DG, and finite difference methods for general RHD equations have been established in \cite{Wu2017}. However, extending these BP methods to RMHD is challenging due to the impact of magnetic divergence errors on the BP property.

Researchers have discovered intrinsic connections between the BP property and the DF condition from both discrete and continuous viewpoints \cite{WuTangM3AS,Wu2017a,WuShu2020NumMath}. Specifically, for the conservative RMHD system \eqref{eq:RMHD} with either an ideal or a general EOS satisfying \eqref{EOS:heq1} and \eqref{EOS:heq2}, it was first shown in \cite{WuTangM3AS,WuTangZAMP} that the BP properties of finite volume and (non-central) DG schemes are closely linked to a discrete DF condition. Furthermore, even minor violations of this discrete DF condition can cause the loss of the BP property \cite{WuTangM3AS}. Unfortunately, the globally coupled nature of the discrete DF condition makes it incompatible with standard local scaling BP limiters in multidimensional cases. Similarly, at the continuous level, researchers in \cite{WuShu2020NumMath} have found that the BP property of the exact solutions to \eqref{eq:RMHD} is also strongly tied to the (continuous) DF constraint \eqref{eq:DF}. Slight deviations from \eqref{eq:DF} can result in failure to preserve the constraints \eqref{eq:G}, implying that the conservative RMHD system \eqref{eq:RMHD} may be ill-posed outside the DF regime. These findings align with those in the non-relativistic case \cite{Wu2017a,WuShu2018,WuShu2019}, highlighting the complexity of understanding and constructing BP schemes for multidimensional RMHD. Fortunately, the symmetrizable modified RMHD system \eqref{eq:ModRMHD} successfully overcomes these challenges, as its exact solutions \eqref{eq:ModRMHD} remain BP even without satisfying the DF constraint \eqref{eq:DF}. This discovery has opened avenues for developing multidimensional BP schemes for RMHD through appropriate discretization of \eqref{eq:ModRMHD}. Inspired by this finding, high-order BP (non-central) DG schemes have been developed in \cite{WuShu2020NumMath} for the modified RMHD equations \eqref{eq:ModRMHD}. For further related developments in non-relativistic ideal MHD, readers can refer to \cite{DingWu2023,WuJiangShu2022}.

This paper focuses on developing  provably BP, locally DF, high-order accurate, central DG (CDG) schemes for RMHD. This work is a sequel to previous BP research \cite{WuTangM3AS,WuShu2020NumMath,WuTangZAMP} on finite volume and standard non-central DG methods. The CDG methods \cite{LiuYJCDG2007}, designed by integrating central schemes \cite{KURGANOV2000,LIUYJ2005} with DG methods \cite{COCKBURN1998}, retain the key advantages of the latter, such as compact stencil and efficient parallel implementation. 
Unlike regular (non-central) DG, the CDG methods evolve two sets of numerical solutions on overlapping cells (primal and dual meshes), thereby avoiding the need to solve complex (approximate) Riemann problems at cell interfaces. 
Despite the increased computational demands in each time step, CDG methods allow for larger time step-sizes \cite{LiuShuTadmorZhang2008,Reyna2015} and have shown superior performance in certain numerical simulations compared to non-central DG methods \cite{LiuShuTadmorZhang2008}. Owing to these advantages, CDG methods have been successfully applied to diverse systems, including Euler 
and special RHD systems  \cite{LiuYJCDG2007,JIANG2022111297,ZhaoTangCDG2017,WuTang2017ApJS}. Exactly (globally) DF CDG schemes were systematically developed in \cite{Li2011,Li2012} for non-relativistic ideal MHD and extended to the conservative RMHD system \eqref{eq:RMHD} in \cite{ZhaoTang2017}. Recognizing the importance of the BP property for robust simulations, high-order BP CDG methods have been developed for various equations, including the compressible Euler system \cite{LiPPCDG2016}, shallow water equations \cite{LiPPCDGSW2017,JIANG2022111297}, non-relativistic MHD systems \cite{WuJiangShu2022}, and the special RHD equations \cite{WuTang2017ApJS}. However, developing provably BP CDG schemes for RMHD faces different challenges, primarily due to the high nonlinearity of constraints \eqref{eq:G} and the impact of magnetic divergence errors on the BP property over overlapping meshes. 

The main efforts and novelties of this paper are summarized as follows:
\begin{itemize}
	\item For 1D conservative RMHD equations \eqref{eq:RMHD}, the standard CDG method maintains the 1D DF constraint, and its BP property is proven under a condition achievable by the BP limiter.
	\item In multidimensional cases, we propose provably BP, locally DF CDG schemes based on suitable discretization of the modified RMHD equations \eqref{eq:ModRMHD}.
	Owing to the locally DF feature, magnetic divergence errors exist only at the cell interfaces.
	The critical and innovative aspect of our CDG schemes lies in carefully discretizing the additional source terms in \eqref{eq:ModRMHD} to precisely counteract the influence of divergence errors on the BP property over overlapping meshes.
	This technical task is informed by our comprehensive theoretical analysis.
	Our source term discretization in the CDG framework, drawing information from the corresponding dual mesh, differs notably from non-central BP DG schemes \cite{WuShu2020NumMath}. 
	\item We provide rigorous proofs of the BP property for our 2D CDG schemes. The proofs are nontrivial and technical, due to the strong nonlinearity of the constraints \eqref{eq:G}. 
		To overcome these challenges, we adopt the recently developed geometric quasilinearization (GQL) approach \cite{WuTangM3AS,WuShu2021GQL}, which equivalently converts the nonlinear constraints into linear ones via a higher-dimensional perspective.  
Utilizing GQL, we establish theoretical links between BP and DF properties.
We prove that the BP property of CDG schemes for 2D conservative RMHD system \eqref{eq:RMHD}
without source terms is closely tied to a discrete DF condition on overlapping meshes, distinct from the one in \cite{WuTangM3AS,WuTangZAMP} for non-central DG methods. 
This discrete DF condition, globally coupled across all the cells, is not maintained by the local scaling BP limiter. 
However, thanks to properly discretized symmetrization source terms, the BP property of our CDG schemes for the modified RMHD equations \eqref{eq:ModRMHD} is influenced only by a local discrete DF condition, which is ensured by our locally DF CDG discretization.  
This further underscores the importance of symmetrization source terms for bound preservation.
The discrete DF constraint intrinsically couples the states in CDG schemes, adding significant complexity to the BP analysis. Thus, some standard BP frameworks, often based on transforming multidimensional schemes into a convex combination of formally 1D BP schemes \cite{zhang2010,zhang2010b,LiPPCDG2016,WuTang2017ApJS}, are not applicable in our multidimensional RMHD cases. 
\item We extend a novel 2D cell average decomposition (CAD), recently proposed for non-central DG schemes \cite{CuiDingWu2023JCP,CuiDingWu2022SINUM}, to overlapping meshes. 
As a cornerstone of BP analysis \cite{zhang2010,CuiDingWu2023JCP,CuiDingWu2022SINUM}, 
CAD decomposes the cell average of the numerical solution into a convex combination of values at some nodal points. 
It determines the theoretical CFL condition of the resulting BP schemes and also affects the computational cost of the associated BP limiter. 
Compared to the classic CAD, our extended CAD offers a notably milder BP CFL constraint and requires fewer internal nodes, enhancing the efficiency of 2D CDG methods.
\item The proposed CDG schemes are Riemann-solver-free, uniformly high-order accurate, and of high resolution. 
We implement these schemes to simulate a range of RMHD problems with various EOSs beyond the ideal EOS, tackling challenging scenarios like ultra-relativistic blast and jet problems in strongly magnetized environments.
\end{itemize}

This paper is organized as follows. 
\Cref{sec:2} introduces the challenges arising from the nonlinear implicit relations from conserved to primitive variables and the GQL approach for addressing these challenges.  
\Cref{1DBPCDG} presents the high-order CDG method for 1D conservative RMHD equations \eqref{eq:RMHD}, along with a rigorous proof of the BP property and the implementation of the BP limiter. \Cref{2DBPCDG} details the provably BP and locally DF high-order CDG method for the 2D RMHD system. In \cref{Numericalexperiments}, we provide several numerical experiments for RMHD with various EOSs. 
The concluding remarks are presented in \cref{Conclusions}.

\section{Preliminaries}\label{sec:2}

\subsection{Nonlinear Implicit Relations from Conserved to Primitive Variables}\label{sec:2.1}

The conserved variables $\bm{U} = (D, \bm{m}, \bm{B}, E)^\top$ can be explicitly derived from the  primitive variables $\bm{W} = (\rho, \bm{v}, \bm{B}, p)^\top$ via \eqref{eq:W2U}. However, the inverse determination of $\bm{W}$ from $\bm{U}$ is difficult, as there are no explicit expressions for $\bm{W}$ in terms of $\bm{U}$, even for the ideal EOS. As a result, the two functions $p(\bm{U})$ and $\bm{v}(\bm{U})$ in \eqref{eq:G} are implicit and highly nonlinear, posing significant challenges in developing BP schemes.

On the other hand, the flux functions ${\bm{F}}_i({\bm{U}})$ also cannot be explicitly expressed by $\bm{U}$. In practice, it is necessary to first recover $\bm{W}$ from $\bm{U}$ and then use both $\bm{U}$ and $\bm{W}$ to calculate ${\bm{F}}_i({\bm{U}})$. 
Recovering $\bm{W}$ from a given $\bm{U}$ is a highly complex task, yet it is involved in all conservative RMHD schemes. This process typically requires solving strongly nonlinear algebraic equations  \cite{noble2006primitive,mignone2006hllc}. Rewriting the EOS \eqref{eq:EOS} as $p = p(\rho, h)$ and defining $\xi := \rho h W^2$, we can follow the approach of \cite{mignone2006hllc,noble2006primitive} to derive a nonlinear equation for the unknown $\xi \in \mathbb{R}^+$:
\begin{equation}\label{eq:RMHD:fU(xi)}
	{\mathcal F}_{ \bm U}(\xi ) := \xi - p\left( \frac{D}{W(\xi)}, \frac{\xi}{D W(\xi)} \right) - \frac{1}{2}\left( {\frac{{{{\left| \vec B \right|}^2}}}{{{{W}^2(\xi)}}} + \frac{{{{(\vec m \cdot \vec B)}^2}}}{{{\xi ^2}}}} \right) + {\left| \vec B \right|^2}  - E = 0, 
\end{equation}
where 
\begin{align}\label{eq:Wxi}
	W(\xi) = \left( {\xi^{-2}} {{(\xi + {|\vec B|^2})}^{-2}}
	f_{\Omega}(\xi)\right)^{ - {1}/{2} } \quad \mbox{with} \quad 
	f_{\Omega}(\xi):={\xi^2}{{(\xi + {|\vec B|^2})}^2} - \left[ {\xi^2}{|\vec m|^2} + (2\xi+{|\vec B|^2}) {{(\vec m \cdot \vec B)}^2} \right].
\end{align}
By using a root-finding algorithm, such as the bisection or Newton-Raphson method, to solve equation \eqref{eq:RMHD:fU(xi)} within the interval $\Omega_f :=  \{\xi \in \mathbb{R}^+: f_{\Omega}(\xi) > 0\}$, one can obtain $\xi$. Subsequently, the primitive variables $\rho(\bm{U})$, $p(\bm{U})$, and $\bm{v}(\bm{U})$ can be computed by 
\begin{equation}\label{eq:RMHD:getv}
	\bm v(\bm U) =\frac{ {\bm m + \xi ^{ - 1}(\vec m \cdot \bm B) \bm B} }{\xi+ {|\bm B|^2}},\quad 
	\rho (\bm U) 	= \frac{D}{{W(\xi)}},\quad h(\bm U) = \frac{\xi}{D W(\xi)} , \quad 
	p(\bm U) = p\left( \rho (\bm U), h(\bm U) \right).
\end{equation}
For either an ideal or a general EOS satisfying \eqref{EOS:heq1} and \eqref{EOS:heq2}, it has been proven in \cite{WuTangM3AS,WuTangZAMP} that the function ${\mathcal{F}}_{\bm{U}}(\xi)$ is strictly increasing and thus has a unique root in $\Omega_f$, provided that $\bm{U} \in {\mathcal{G}}$. Indeed, ensuring the BP property for $\bm{U}$ is essential for robustly recovering the primitive variables $\bm{W}$.

\subsection{Geometric Quasilinearization (GQL)}\label{GQL} 

Since the two implicit functions $p(\bm{U})$ and $\bm{v}(\bm{U})$ in \eqref{eq:G} are highly nonlinear, determining whether a given $\bm{U}$ belongs to $\mathcal{G}$ is already a difficult task. Finding a numerical RMHD scheme that provably preserves its solution within $\mathcal{G}$ is even more challenging.

To address the challenges arising from this high nonlinearity, we employ the GQL approach \cite{WuTangM3AS,WuShu2021GQL}, which establishes an equivalent linear representation of any convex set by introducing extra auxiliary variables. This representation, known as the GQL representation, will be the fundamental cornerstone for our BP analysis and the proof of related theorems.

The convexity of $\mathcal{G}$ was proven in \cite{WuTangM3AS} using differential geometry techniques.

\begin{lemma}[Convexity]\label{Lemma:convex}
	The admissible state set $\mathcal{G}$ is convex.
\end{lemma}

Thanks to this convexity, the GQL representation can be established \cite{WuTangM3AS,WuTangZAMP,WuShu2021GQL}.

\begin{lemma}[GQL Representation]\label{Lemma:Gstar}
	The admissible state set $\mathcal{G}$ is exactly equivalent to the set 
	\begin{equation}\label{eq:Gs}
		\mathcal{G}_* = \left\{
		\bm{U}=(D, \bm{m}, \bm{B}, E)^\top:~ \bm{U}\cdot \bm{n}_1>0,~ \bm{U}\cdot \bm{n}^{*} + p^*_m >0 ~~  \forall \bm{v}^*\in \mathbb{B}_1(\bm 0),  ~~ \forall \bm{B}^*\in \mathbb{R}^3 \right\},
	\end{equation}
	where $\bm{n}_1 =(1,0,0,0,0,0,0,0)^\top$,  
	$\bm{n}^{*}=\left(-\sqrt{1-|\bm{v}^*|^2}, -\bm{v}^*, -(1-|\bm{v}^*|^2)\bm{B}^*, 1\right)^{\top}$, 
	$p^*_m=\frac{1}{2}\left((1-|\bm{v}^*|^2)\bm{B}^*+(\bm{v}^* \cdot \bm{B}^*)^2\right)$, 
	and $\{\bm{v}^*, \bm{B}^*\}$ are the extra free auxiliary variables independent of $\bm U$.  
\end{lemma}

Remarkably, all the constraints in the GQL representation \eqref{eq:Gs} become {\em linear} with respect to $\bm{U}$. This highly desirable feature gives it significant advantages in analyzing and designing BP schemes. The first proof of \Cref{Lemma:Gstar} can be found in \cite{WuTangM3AS} for the ideal EOS and in \cite{WuTangZAMP} for a general equation satisfying \eqref{EOS:heq1} and \eqref{EOS:heq2}. The geometric interpretation of \Cref{eq:Gs} and the more general GQL framework were established in \cite{WuShu2021GQL} for any convex invariant region of partial differential equations.

Thanks to the GQL approach, we have the following inequalities, which will be useful in our BP analysis. 

\begin{lemma}[See \cite{WuTangM3AS,WuTangZAMP}] \label{Lemma:LF}
	Given any admissible state $\bm{U} \in \mathcal{G}$, we have
	\begin{equation}\label{eq:keyieq}
		\left(\bm{U} + \theta \bm{F}_\ell(\bm{U}) \right) \cdot \bm{n}^* + p^*_m + \theta \left(v_\ell^* p^*_m - B_\ell (\bm{v}^* \cdot \bm{B}^*)\right) \geq0,
	\end{equation} 
	for any $\theta \in[-1,1]$, any $\bm{v}^* \in \mathbb{B}_1(\bm 0)$, and any $\bm{B}^*\in \mathbb{R}^3$, where $\ell \in \{1,2,3\}$.  Here, $\{\bm{v}^*, \bm{B}^*\}$ are the extra free auxiliary variables in the GQL representation. 
\end{lemma}

We have the following two corollaries as direct consequences of \Cref{Lemma:LF}.
\begin{corollary}\label{Coro:Flux}
	Given any admissible state $\bm{U} \in \mathcal{G}$, we have 
	\begin{equation*}\label{eq:GQLCor1_1}
		\pm \bm{F}_\ell(\bm{U}) \cdot \bm{n}_1 \geq - \bm{U} \cdot \bm{n}_1,
	\end{equation*}
	\begin{equation*}\label{eq:GQLCor1_2}
		\pm \bm{F}_\ell(\bm{U}) \cdot \bm{n}^* \geq -\bm{U} \cdot \bm{n}^* - p^*_m \mp \left(v_\ell^* p^*_m - B_\ell (\bm{v}^* \cdot \bm{B}^*)\right)
	\end{equation*}
	for any $\bm{v}^* \in \mathbb{B}_1(\bm 0)$, any $\bm{B}^*\in \mathbb{R}^3$, and $\ell \in \{1,2,3\}$.
\end{corollary}
\begin{proof}
	Using the constraint $|\bm{v}|<c=1$, we can prove the first inequality by 
	$$
	\pm \bm{F}_\ell(\bm{U}) \cdot \bm{n}_1 = \pm v_l D \geq -D =  - \bm{U} \cdot \bm{n}_1.
	$$
	The second inequality can be obtained by reformulating \eqref{eq:keyieq} with the choice of $\theta=\pm 1$.
\end{proof}

\begin{corollary}\label{Coro:Flux2}
	Given any two admissible states $\bm{U} \in {\mathcal G}$ and $\tilde{\bm{U} }\in \mathcal{G}$, we have 
	\begin{equation}\label{eq:GQLCor2_1}
		-\left(\bm{F}_\ell(\bm{U}) -\bm{F}_\ell(\tilde{\bm{U}})\right)\cdot \bm{n}_1 \geq -(\bm{U} +  \tilde{\bm{U}}) \cdot \bm{n}_1,
	\end{equation}
	\begin{equation}\label{eq:GQLCor2_2}
		-\left(\bm{F}_\ell(\bm{U}) -\bm{F}_\ell(\tilde{\bm{U}})\right)\cdot \bm{n}^* \geq -\left((\bm{U} +  \tilde{\bm{U}}) \cdot \bm{n}^*  + 2 p^*_m  \right) - (  B_\ell -\tilde{B}_\ell ) (\bm{v}^* \cdot \bm{B}^*)
	\end{equation}
	for any $\bm{v}^* \in \mathbb{B}_1(\bm 0)$, any $\bm{B}^*\in \mathbb{R}^3$, and $\ell \in \{1,2,3\}$. 
\end{corollary}
\begin{proof}
	The results directly follow from \Cref{Coro:Flux}.
\end{proof}

\Cref{Coro:Flux2} will play a key role in analyzing the impact of fluxes on the BP property. In particular, the last term in the inequality \eqref{eq:GQLCor2_2} is crucial, as it relates the BP property to a discrete DF condition.

The following lemma will be valuable in analyzing the impact of the discrete symmetrization source terms on the BP property. Its proof can be found in \cite{WuShu2020NumMath}.

\begin{lemma}[]\label{Lemma:SourT}
	For any $\bm{U} \in \mathcal{G}$, any $\bm{v}^* \in \mathbb{B}_1(\bm 0)$, any $\bm{B}^* \in \mathbb{R}^3$, and $\zeta \in \mathbb{R}$, we have 
	\begin{equation}
		\zeta \left( \bm{S}(\bm{U})\cdot \bm{n}^* \right) \geq -\zeta (\bm{v}^* \cdot \bm{B}^*) - \frac{|\zeta|}{\sqrt{\rho h}} \left( \bm{U} \cdot \bm{n}^* + p^*_m \right).
	\end{equation}
Here, $\{\bm{v}^*, \bm{B}^*\}$ are the extra free auxiliary variables in the GQL representation. 
\end{lemma}

\section{1D Provably BP CDG Schemes}\label{1DBPCDG}
This section presents the BP CDG schemes for the RMHD equations \eqref{eq:RMHD} in one dimension, where the 
DF constraint \eqref{eq:DF} reduces to  
$$
B_1(x,t) \equiv B_{const} \quad \forall x, t\ge 0, 
$$
where $x$ represents the 1D spatial coordinate variable, and $B_{const}$ is a constant. Assume that the computational domain is uniformly partitioned into a (primal) mesh $\{I_j:=(x_{j-\halfone}, x_{j+\halfone})\}$ with the constant spatial step-size  $\dx=x_{j+\halfone}-x_{j-\halfone}$. To define the CDG scheme, we also introduce a dual mesh  $\{J_{j+\halfone}=(x_{j}, x_{j+1})\}$, where $x_j = \frac{1}{2}(x_{j-\halfone}+x_{j+\halfone})$ is the center of the cell $I_{j}$. 

\subsection{1D CDG Method}
Define the DG finite element spaces ${\mathbb V}_h^{I}$ and ${\mathbb V}_h^{J}$ as follows:
\begin{gather*}
	{\mathbb V}_h^{I} = \left \{\bm{w}=(w_1,\cdots,w_8)^\top \in [L^2(\Omega)]^8: w_\ell|_{I_j} \in \mathbb{P}^k(I_j) ~~ \forall 1\leq \ell \leq 8, ~\forall j \right \}, \\
	{\mathbb V}_h^{J} = \left\{\bm{u}=(u_1,\cdots,u_8)^\top\in [L^2(\Omega)]^8: u_\ell|_{J_{j+\halfone}} \in \mathbb{P}^k(J_{j+\halfone}) ~~ \forall 1\leq \ell \leq 8, ~\forall j \right\}, 
\end{gather*}
where $\mathbb{P}^k(I_j)$ and $\mathbb{P}^k(I_{j+\halfone})$ are the polynomial of the degree up to $k$ over $I_j$ and $I_{j+\halfone}$, respectively.
Then the semi-discrete CDG scheme is given as follows: seek the numerical solution $\bm{U}_h^I(x,t) \in {\mathbb V}_h^{I}$ on the primal mesh and $\bm{U}_h^J(x,t) \in {\mathbb V}_h^{J}$ on the dual mesh such that 
\begin{equation}\label{CDG1D:IntI}
	\begin{aligned}
		\int_{I_{j}} \partial_t(\bm{U}_h^I) \cdot \bm{w} \dd x =&  \frac{1}{\tau_{\rm max}} \int_{I_{j}} (\bm{U}_h^J - \bm{U}_h^I)\cdot \bm{w} \dd x +\int_{I_{j}} \bm{F}_1(\bm{U}_h^J) \cdot \partial_x \bm{w} \dd x    \\
		&+  \bm{F}_1(\bm{U}_h^J(x_{j-\halfone},t)) \cdot \bm{w}(x_{j-\halfone}^+) - \bm{F}_1(\bm{U}_h^J(x_{j+\halfone},t)) \cdot \bm{w}(x_{j+\halfone}^-) \quad \forall \bm{w} \in {\mathbb V}_h^{I},
	\end{aligned}
\end{equation}
\begin{equation}\label{CDG1D:IntJ}
	\begin{aligned}
		\int_{J_{j+\halfone}} \partial_t(\bm{U}_h^J) \cdot \bm{u} \dd x =&  \frac{1}{\tau_{\rm max}} \int_{J_{j+\halfone}} (\bm{U}_h^I - \bm{U}_h^J)\cdot \bm{u} \dd x +\int_{J_{j+\halfone}} \bm{F}_1(\bm{U}_h^I) \cdot \partial_x \bm{u} \dd x   \\
		& +  \bm{F}_1(\bm{U}_h^I(x_{j},t)) \cdot \bm{u}(x_{j}^+) - \bm{F}_1(\bm{U}_h^I(x_{j+1},t)) \cdot \bm{u}(x_{j+1}^-) \quad \forall \bm{u} \in {\mathbb V}_h^{J}.
	\end{aligned}
\end{equation}
Here, $\bm{w}(x_{j+\halfone}^{-})$ and $\bm{w}(x_{j+\halfone}^{+})$ denote the left- and right-hand limits of $\bm{w}$ at the interface $x_{j+\halfone}$, respectively; similarly, $\bm{u}(x_{j}^{\pm})$ are the left- and right-hand limits of $\bm{u}$ at the interfaces $x_{j}$.

In the equations \eqref{CDG1D:IntI} and \eqref{CDG1D:IntJ}, the first term at the right-hand side represents additional numerical dissipation, which is crucial for maintaining the stability of CDG schemes. This term is also instrumental in ensuring the BP property, as shown later. The parameter $\tau_{\rm max}$ in \eqref{CDG1D:IntI} and \eqref{CDG1D:IntJ} is the maximum allowable time step-size for sustaining stability. This parameter is determined in accordance with specific CFL conditions; see \eqref{CDG1D:CFLCond}. 
For further insights into the significance of $\tau_{\rm max}$ within the context of CDG methods, one may refer to 
 \cite{LiuYJCDG2007,LiuShuTadmorZhang2008,Reyna2015}.

 \begin{remark}
 	 A distinctive aspect of the CDG method is that it does not need any numerical flux which typically relies on exact or approximate Riemann solvers. This is because the solutions or fluxes are computed at the cell interfaces of the primal (or respectively, dual) mesh, corresponding to the centers of the dual (or primal) mesh, where the numerical solution $\bm{U}_h^J(x,t)$ (or $\bm{U}_h^I(x,t)$) exhibits continuity.
 \end{remark}

Consider a local orthogonal basis $\{\Psi_j^{(\ell)}(x)\}_{\ell=0}^{k}$ of the polynomial space $\mathbb{P}^k(I_j)$ and $\{\Psi_{j+\halfone}^{(\ell)}(x)\}_{\ell=0}^{k}$ of the polynomial space $\mathbb{P}^k(I_{j+\halfone})$. For instance, one can employ the scaled Legendre polynomials 
\begin{equation*}
	\Psi_{j_0}^{(0)}(x) = 1, \quad \Psi_{j_0}^{(1)}(x) = \xi_{j_0}, \quad \Psi_{j_0}^{(2)}(x) =  \xi_{j_0}^2 - \frac{1}{3}, \quad 
	\Psi_{j_0}^{(3)}(x) =  \xi_{j_0}^3  - \frac{3}{5} \xi_{j_0}, ~~
	\dots
\end{equation*}
where $\xi_{j_0} := 2(x-x_{j_0})/\dx$ and the index $j_0 = j$ or $j+\halfone$.
Then the numerical solutions $\bm{U}_h^I(x,t)$ and $\bm{U}_h^J(x,t)$ in the spaces ${\mathbb V}_h^{I}$ and ${\mathbb V}_h^{J}$, respectively, can be written as
\begin{equation}\label{CDG1D:SolI}
	\bm{U}_h^I(x,t) = \sum_{\ell=0}^{k} \bm{U}_j^{I,(\ell)}(t) \Psi_j^{(\ell)}(x) ~~~ \forall x\in I_j,
\end{equation} 
\begin{equation}\label{CDG1D:SolJ}
	\bm{U}_h^J(x,t) = \sum_{\ell=0}^{k} \bm{U}_{j+\halfone}^{J,(\ell)}(t) \Psi_{j+\halfone}^{(\ell)}(x) ~~~ \forall x\in I_{j+\halfone},
\end{equation} 
where $\bm{U}_j^{I,(\ell)}(t)$ and $\bm{U}_{j+\halfone}^{J,(\ell)}(t)$ are the degrees of freedom. Define 
$${a}_j^{(\ell)} = \frac{1}{\dx}\int_{I_j} \left(\Psi_j^{(\ell)}(x)\right)^2 \dd x, \quad
{a}_{j+\halfone}^{(\ell)} =\frac{1}{\dx}\int_{J_{j+\halfone}} \left(\Psi_{j+\halfone}^{(\ell)}(x)\right)^2 \dd x. $$
Since $\bm{U}_h^J(x,t)$ may exhibit discontinuity at $x=x_{j}$, we should split the integration $\int_{I_{j}} \bm{F}_1(\bm{U}_h^J) \cdot \partial_x \bm{w} \dd x$ in \eqref{CDG1D:IntI} into two parts $\int_{x_{j-\frac12}}^{x_j} \bm{F}_1(\bm{U}_h^J) \cdot \partial_x \bm{w} \dd x$ and $\int_{x_{j}}^{x_{j+\frac12}} \bm{F}_1(\bm{U}_h^J) \cdot \partial_x \bm{w} \dd x$, which are then approximated by using a quadrature rule of sufficiently high accuracy. 
Let $\{{x}_{j\pm\fourone}^\mu\}_{\mu=1}^Q$ be the $Q$-point Gauss quadrature nodes in  $[x_{j\pm\fourone}-\frac{\dx}{4}, x_{j\pm\fourone}+\frac{\dx}{4}]$, where $x_{j\pm \frac14}:=x_{j}\pm \frac14 \dx$. Let $\{ {\omega}_{\mu} \}_{\mu=1}^{Q}$ be the  quadrature weights on the interval $[-\halfone,\halfone]$ such that $\sum_{\mu=1}^{Q} {\omega}_{\mu} =1$. 
For accuracy requirement, we take $Q=k+1$. 
Then the semi-discrete CDG scheme \eqref{CDG1D:IntI} on the primal mesh can be reformulated as 
\begin{equation}\label{CDG1D:DisI}
	\begin{aligned}
		{a}_j^{(\ell)}\frac{\dd \bm{U}_j^{I,(\ell)}(t)}{\dd t} 
		&= 
		\frac{1}{\tau_{\rm max} \dx } \int_{I_{j}} (\bm{U}_h^J - \bm{U}_h^I) \Psi_j^{(\ell)}(x) \dd x + \sum_{m=\pm1} \sum_{\mu=1}^{Q} \frac{{\omega}_{\mu}}{2}\bm{F}_1(\bm{U}_h^J({x}_{j+\mfourone}^{\mu},t)) \frac{\dd \Psi_j^{(\ell)}({x}_{j+\mfourone}^{\mu}) }{\dd x}    \\
		&+  \frac{1}{
			\dx}\left(\bm{F}_1(\bm{U}_h^J(x_{j-\halfone},t)) \Psi_j^{(\ell)}(x_{j-\halfone}^+) - \bm{F}_1(\bm{U}_h^J(x_{j+\halfone},t)) \Psi_j^{(\ell)}(x_{j+\halfone}^-) \right), \quad   \ell=0,\cdots,k. 
	\end{aligned}
\end{equation}
Similarly,  the semi-discrete CDG scheme \eqref{CDG1D:IntJ} on the dual mesh can be rewritten as
\begin{equation}\label{CDG1D:DisJ}
	\begin{aligned}
		{a}_{j+\halfone}^{(\ell)} \frac{\dd \bm{U}_{j+\halfone}^{J,(\ell)}(t)}{\dd t} =&  \frac{1}{ \tau_{\rm max} \dx} \int_{J_{j+\halfone}} (\bm{U}_h^I - \bm{U}_h^J) \Psi_{j+\halfone}^{(\ell)}(x)  \dd x +\sum_{m=\pm1} \sum_{\mu=1}^{Q} \frac{{\omega}_{\mu}}{2} \bm{F}_1(\bm{U}_h^I({x}_{j+\halfone+\mfourone}^\mu,t)) \frac{\dd \Psi_{j+\halfone}^{(\ell)}({x}_{j+\halfone+\mfourone}^{\mu})}{\dd x}  \\
		& +  \frac{1}{\dx} \left(\bm{F}_1(\bm{U}_h^I(x_{j},t)) \Psi_{j+\halfone}^{(\ell)}(x_{j}^+) - \bm{F}_1(\bm{U}_h^I(x_{j+1},t)) \Psi_{j+\halfone}^{(\ell)}(x_{j+1}^-) \right),  \quad  \ell=0,\cdots,k.
	\end{aligned}
\end{equation}

The equations \eqref{CDG1D:DisI} and \eqref{CDG1D:DisJ} constitute a nonlinear system of ordinary differential equations for the degrees of freedom $\bm{U}_{j}^{J,(\ell)}(t)$ and $\bm{U}_{j+\halfone}^{J,(\ell)}(t)$. This system can be formally represented as $\frac{\dd \bm{U}}{\dd t} = \mathcal{L}(\bm{U})$. To achieve a fully discrete, high-order accurate CDG scheme, one can employ high-order strong-stability-preserving (SSP) time discretization \cite{GottliebShuTadmor2001}. For example, the SSP third-order Runge--Kutta method is described by
\begin{equation}\label{RKTime}
	\begin{aligned}
		{\bm{U}}^{[1]} &= {\bm{U}}^{n} + \dt \mathcal{L}({\bm{U}}^{n}), \\
		{\bm{U}}^{[2]} &= \frac{3}{4}{\bm{U}}^{n} + \frac{1}{4} \left({\bm{U}}^{[1]} + \dt \mathcal{L}({\bm{U}}^{[1]}) \right), \\
		{\bm{U}}^{n+1} &= \frac{1}{3}{\bm{U}}^{n} + \frac{2}{3} \left({\bm{U}}^{[2]} + \dt \mathcal{L}({\bm{U}}^{[2]}) \right),
	\end{aligned}
\end{equation}
and the SSP third-order multi-step method is given by 
\begin{equation*}
	{\bm{U}}^{n+1} = \frac{16}{27}\left({\bm{U}}^{n} + 3\dt \mathcal{L}({\bm{U}}^{n}) \right) 
	+ \frac{11}{27} \left({\bm{U}}^{n-3} + \frac{12}{11} \dt \mathcal{L}({\bm{U}}^{n-3}) \right),
\end{equation*}
where $\dt$ denotes the time step-size.

\begin{lemma}[Divergence-Free Property]\label{lem:1DDF}
	Let $B_{1,h}^{I}(x,t)$ and $B_{1,h}^{J}(x,t)$ denote the fifth component of the vector functions $\bm{U}_h^I(x,t)$ and $\bm{U}_h^J(x,t)$, respectively.  
	The 1D fully discrete CDG schemes with an SSP time discretization method exactly preserve the DF constraint
\begin{equation}\label{eq:1DgDF}
		B_{1,h}^{I}(x,t_n) = B_{const} =B_{1,h}^{J}(x,t_n) \qquad \forall x,~ n \ge 0,
\end{equation}
	with $t_n$ denoting the $n$th time level. 
\end{lemma}

\begin{proof}
	Since the fifth component of $\bm{F}_1(\bm{U})$ is identically zero, 
	the semi-discrete CDG schemes \eqref{CDG1D:DisI}--\eqref{CDG1D:DisJ} satisfy 
	$$
	\frac{\dd {B}_{1,j}^{I,(\ell)}(t)}{\dd t} = 0, \quad \frac{\dd {B}_{1,j+\frac12}^{J,(\ell)}(t)}{\dd t} = 0 \qquad \forall j,~ \ell=0,\cdots, k,
	$$
	where ${B}_{1,j}^{I,(\ell)}$ (resp.~${B}_{1,j+\frac12}^{J,(\ell)}$) is the fifth component of the vector $\bm{U}_j^{I,(\ell)}$ (resp.~$\bm{U}_{j+\halfone}^{J,(\ell)}$). 
	This implies \eqref{eq:1DgDF} and completes the proof. 
\end{proof}

\subsection{Rigorous BP Analysis of 1D CDG Method for RMHD}

Define the cell averages $\overline{\bm{U}}_j^I = \frac{1}{\dx}\int_{I_j} \bm{U}_h^I \dd x$ and $\overline{\bm{U}}_{j+\halfone}^J = \frac{1}{\dx}\int_{J_{j+\halfone}} \bm{U}_h^J \dd x$ over the cell $I_j$ and $J_{j+\halfone}$, respectively. 
Then we have $\overline{\bm{U}}_j^I=\bm{U}_j^{I,(0)}$ and $\overline{\bm{U}}_{j+\halfone}^J=\bm{U}_{j+\halfone}^{J,(0)}$ due to the orthogonality of the local basis. 
By taking $\ell=0$ in \eqref{CDG1D:DisI} and \eqref{CDG1D:DisJ}, we obtain the evolution equations of cell averages in the CDG method as follows: 
\begin{equation}\label{CDG1D:aveI}
	\begin{aligned}
		&\frac{\dd \overline{\bm{U}}_j^I}{\dd t} = \frac{\overline{\bm{U}}_j^J-\overline{\bm{U}}_j^I}{\tau_{\rm max}} - \frac{\bm{F}_1(\bm{U}^J_{j+\halfone}) - \bm{F}_1(\bm{U}^J_{j-\halfone})}{\dx} =: \mathcal{L}_j^I(\bm{U}_h^I, \bm{U}_h^J), 
	\end{aligned}
\end{equation}
\begin{equation}\label{CDG1D:aveJ}
	\begin{aligned}
		\frac{\dd \overline{\bm{U}}_{j+\halfone}^J}{\dd t} = \frac{\overline{\bm{U}}_{j+\halfone}^I-\overline{\bm{U}}_{j+\halfone}^J}{\tau_{\rm max}} - \frac{\bm{F}_1(\bm{U}^I_{j+1}) - \bm{F}_1(\bm{U}^I_{j})}{\dx}=: \mathcal{L}_{j+\halfone}^J(\bm{U}_h^J,\bm{U}_h^I),
	\end{aligned}
\end{equation}
where $\bm{U}^J_{j\pm\halfone} = \bm{U}_h^J(x_{j\pm\halfone},t)$ and $\bm{U}^I_{j+\halfone \pm \halfone} = \bm{U}_h^I(x_{j+\halfone \pm \halfone},t)$. For convenience, we will omit the temporal dependence $(t)$ from all quantities in our following discussions, provided that this does not lead to any confusion. 

Let $\{\widehat{x}_{j\pm\fourone}^\nu\}_{\nu=1}^L$ be the $L$-point Gauss--Lobatto quadrature nodes in  $[x_{j\pm\fourone}-\frac{\dx}{4}, x_{j\pm\fourone}+\frac{\dx}{4}]$ with $L=\lceil\frac{k+3}{2} \rceil$. Let $\{ \widehat{\omega}_{\nu} \}_{\nu=1}^{L}$ be the quadrature weights on the interval $[-\halfone,\halfone]$ such that $\sum_{\nu=1}^{L} \widehat{\omega}_{\nu} =1$. Then we have the following theorem. 
\begin{theorem}[Bound-Preserving Property]\label{thm:main1D}
	Assume that $\overline{\bm{U}}_{j}^I \in \mathcal{G}$ and $\overline{\bm{U}}_{j+\halfone}^J\in \mathcal{G}$ for all $j$. 
	If the numerical solutions $\bm{U}_h^I(x)$ and $\bm{U}_h^J(x)$ satisfy
	\begin{equation}\label{CDG1D:GL}
		\bm{U}_h^I(\widehat{x}_{j\pm\fourone}^\nu) \in \mathcal{G}, \quad	\bm{U}_h^J(\widehat{x}_{j\pm\fourone}^\nu)\in \mathcal{G} \qquad \forall j, ~ 1\leq \nu \leq L, 
	\end{equation}
	and the discrete DF constraint 
	\begin{equation}\label{CDG1D:B}
		(B_1)^J_{j\pm\halfone} = B_{const} = (B_1)^I_{j+\halfone\pm\halfone} \qquad \forall j, 
	\end{equation}
	then the updated cell averages satisfy
	\begin{equation}\label{CDG1D:Update}
		\overline{\bm{U}}_{j}^{I,\dt} := 
		\overline{\bm{U}}_{j}^I + \dt \mathcal{L}_{j}^I(\bm{U}_h^I,\bm{U}_h^J) \in \mathcal{G},
		\quad
		\overline{\bm{U}}_{j+\halfone}^{J,\dt} := 
		\overline{\bm{U}}_{j+\halfone}^J + \dt \mathcal{L}_{j+\halfone}^J(\bm{U}_h^J,\bm{U}_h^I) \in \mathcal{G}
		\quad \forall j 
	\end{equation}
	under the CFL condition
	\begin{equation}\label{CDG1D:CFLCond}
		0<\dt \le \frac{\theta \widehat{\omega}_1}{2} \dx
	\end{equation}
	with $\theta = \frac{\dt}{\tau_{\rm max}} \in (0,1]$.
	
\end{theorem}

\begin{proof}
	We only present the proof of $\overline{\bm{U}}_{j}^{I,\dt} \in \mathcal{G}$, as the proof of $\overline{\bm{U}}_{j+\halfone}^{J,\dt} \in \mathcal{G}$ is analogous and thus omitted. 
	Note that the $L$-point Gauss--Lobatto quadrature rule with $L=\lceil\frac{k+3}{2} \rceil$ is exact for all polynomials of degree not exceeding $k$. 
	This implies the following 1D CAD: 
\begin{align} \label{eq:1D}
			\overline{\bm{U}}_{j}^{J} &= \sum_{\nu=1}^{L} \frac{\widehat{\omega}_\nu}{2} \left(\bm{U}_h^J(\widehat{x}_{j-\fourone}^\nu) + \bm{U}_h^J(\widehat{x}_{j+\fourone}^\nu)\right) 
		= \frac{\widehat{\omega}_1}{2}\left(\bm{U}_{j-\halfone}^J + \bm{U}_{j+\halfone}^J\right) + {\bm \Pi}_{j}^J , 
		\\ \nonumber
		{\bm \Pi}_j^J &:= \sum_{\nu=2}^{L} \frac{\widehat{\omega}_\nu}{2}\bm{U}_h^J(\widehat{x}_{j-\fourone}^\nu) + \sum_{\nu=1}^{L-1} \frac{\widehat{\omega}_\nu}{2}\bm{U}_h^J(\widehat{x}_{j+\fourone}^\nu),
\end{align}
	where we have used $\bm{U}_{j-\halfone}^J = \bm{U}_h^J(\widehat{x}_{j-\fourone}^1)$, $\bm{U}_{j+\halfone}^J = \bm{U}_h^J(\widehat{x}_{j+\fourone}^L)$, and $\widehat{\omega}_1 = \widehat{\omega}_L$. 
	Then $\overline{\bm{U}}_{j}^{I,\dt}$ can be rewritten as
	\begin{equation}\label{WKL1}
		\begin{aligned}
			\overline{\bm{U}}_{j}^{I,\dt} &= (1-\theta)\overline{\bm{U}}_{j}^{I} + \theta \overline{\bm{U}}_{j}^{J} - \frac{\dt}{\dx}\left(\bm{F}_1(\bm{U}^J_{j+\halfone}) - \bm{F}_1(\bm{U}^J_{j-\halfone})\right) \\
			&= (1-\theta)\overline{\bm{U}}_{j}^{I} + \theta {\bm \Pi}_{j}^J
			+ \frac{\theta \widehat{\omega}_1}{2}\left(\bm{U}_{j-\halfone}^J + \bm{U}_{j+\halfone}^J\right) 
			+\frac{\dt}{\dx} {\bm \Pi}_{\bm{F}} 
		\end{aligned}
	\end{equation}
	with
	$$
	{\bm \Pi}_{\bm{F}} :=-\left(
	\bm{F}_1(\bm{U}^J_{j+\halfone}) - \bm{F}_1(\bm{U}^J_{j-\halfone}) \right).
	$$		
	Under the condition \eqref{CDG1D:GL}, we have 
	\begin{equation}\label{pf:CDG1Dcon1}
		\overline{\bm{U}}_{j}^{I} \in \mathcal{G}, \quad
		\bm{U}_{j\pm\halfone}^J \in \mathcal{G}, \quad 
		\frac{1}{1-\widehat{\omega}_1}{\bm \Pi}_{j}^J \in \mathcal{G}.
	\end{equation}
	We first prove that $\overline{\bm{U}}_{j}^{I,\dt} \cdot \bm{n}_1 >0$. Using \eqref{eq:GQLCor2_1} in \Cref{Coro:Flux2} gives 
	$$
	{\bm \Pi}_{\bm{F}} \cdot \bm{n}_1  \geq  -\left( \bm{U}_{j-\halfone}^J + \bm{U}_{j+\halfone}^J\right) \cdot \bm{n}_1,
	$$
	which, along with \eqref{WKL1}, implies 
	\begin{equation*}
		\begin{aligned}
			\overline{\bm{U}}_{j}^{I,\dt} \cdot \bm{n}_1 
			& \geq (1-\theta)\overline{\bm{U}}_{j}^{I} \cdot \bm{n}_1  + \theta {\bm \Pi}_{j}^J \cdot \bm{n}_1 + \frac{\theta \widehat{\omega}_1}{2}\left(\bm{U}_{j-\halfone}^J + \bm{U}_{j+\halfone}^J\right) \cdot \bm{n}_1 - \frac{\dt}{\dx} \left( \bm{U}_{j-\halfone}^J + \bm{U}_{j+\halfone}^J\right) \cdot \bm{n}_1 \\
			& > \left(\frac{\theta \widehat{\omega}_1}{2} - \frac{\dt}{\dx} \right) \left( \bm{U}_{j-\halfone}^J + \bm{U}_{j+\halfone}^J\right) \cdot \bm{n}_1 
			\ge 0,
		\end{aligned}
	\end{equation*}
	where the last step follows from the CFL constraint \eqref{CDG1D:CFLCond} and $\bm{U}_{j\pm \halfone}^J \cdot \bm{n}_1 >0$ according to $\bm{U}_{j\pm\halfone}^J \in \mathcal{G}$ and the GQL representation \eqref{eq:Gs}. 
	We then prove that $\overline{\bm{U}}_{j}^{I,\dt} \cdot \bm{n}^* + p^*_m>0$ for any free auxiliary variables $\bm{v}^* \in \mathbb{B}_1(\bm 0)$ and $\bm{B}^*\in \mathbb{R}^3$. 
	Using \eqref{eq:GQLCor2_2} in \Cref{Coro:Flux2} gives 
	\begin{equation*}
		\begin{aligned}
		 {\bm \Pi}_{\bm{F}} \cdot \bm{n}^* 
			&\geq - \left(  (\bm{U}_{j-\halfone}^J + \bm{U}_{j+\halfone}^J) \cdot \bm{n}^* + 2 p^*_m \right)-\left((B_1)_{j+\halfone}^J - (B_1)_{j-\halfone}^J\right) (\bm{v}^*\cdot \bm{B}^*) \\ 
			&=- \left( (\bm{U}_{j-\halfone}^J + \bm{U}_{j+\halfone}^J) \cdot \bm{n}^* + 2 p^*_m \right),
		\end{aligned}
	\end{equation*}
	where the equality follows from the discrete DF constraint \eqref{CDG1D:B}. 
	Therefore, 
	\begin{equation*}
		\begin{aligned}
			\overline{\bm{U}}_{j}^{I,\dt} \cdot \bm{n}^* + p^*_m
			&=	(1-\theta)(\overline{\bm{U}}_{j}^{I} \cdot \bm{n}^* + p^*_m)  + \theta ({\bm \Pi}_{j}^J \cdot \bm{n}^* + (1-\widehat{\omega}_1)p^*_m ) \\
			& ~~~~~~ + \frac{\theta \widehat{\omega}_1}{2}\left((\bm{U}_{j-\halfone}^J + \bm{U}_{j+\halfone}^J) \cdot \bm{n}^* + 2 p^*_m\right) + \frac{\dt}{\dx}  {\bm \Pi}_{\bm{F}} \cdot \bm{n}^*  \\
			&> \left(\frac{\theta \widehat{\omega}_1}{2} - \frac{\dt}{\dx}\right)\left((\bm{U}_{j-\halfone}^J + \bm{U}_{j+\halfone}^J) \cdot \bm{n}^* + 2p^*_m\right)
			\ge 0,
		\end{aligned}
	\end{equation*}
	where we have invoked the CFL condition \eqref{CDG1D:CFLCond} and the fact that
	$\bm{U}_{j\pm \halfone}^J \cdot \bm{n}^* + p^*_m >0$ due to $\bm{U}_{j\pm\halfone}^J \in \mathcal{G}$ and the GQL representation \eqref{eq:Gs}. 
	Thanks to \Cref{Lemma:Gstar}, we obtain $\overline{\bm{U}}_{j}^{I,\dt}\in \mathcal{G}_{*}=\mathcal{G}$. 
	Similar arguments yield  $\overline{\bm{U}}_{j+\frac12}^{J,\dt}\in \mathcal{G}$. 
	The proof is completed.	
\end{proof}

\begin{remark}
	\Cref{thm:main1D} provides the sufficient conditions \eqref{CDG1D:GL} and \eqref{CDG1D:B} under which the 1D CDG schemes maintain the 
	BP property of the updated cell averages, when the forward Euler method is employed for time discretization. 
	Furthermore, the BP property remains true in the 1D CDG schemes utilizing high-order SSP time discretization, as an SSP method can be formulated as a convex combination of the forward Euler method. 
\end{remark}

\begin{remark}
	\Cref{thm:main1D} indicates that the BP property of the 1D CDG schemes is linked to the discrete DF condition \eqref{CDG1D:B}. According to \Cref{lem:1DDF}, the 1D CDG schemes inherently satisfy the exact 1D DF property
	 \eqref{eq:1DgDF}. This implies that the discrete DF condition \eqref{CDG1D:B} is naturally met in the 1D case, rendering it a trivial aspect. However, as we will see, the corresponding discrete DF condition in the two-dimensional (2D) case presents notable differences and is fairly nontrivial. 
\end{remark}


\subsection{1D BP Limiter}
The BP condition \eqref{CDG1D:GL} may not be inherently fulfilled by the CDG solutions. To ensure compliance with this BP condition \eqref{CDG1D:GL}, the implementation of a local scaling BP limiter is often required. Such types of BP limiters were originally proposed in \cite{zhang2010} for scalar conservation laws and have been generalized to the compressible Euler equations \cite{zhang2010b}, relativistic hydrodynamics \cite{QinShu2016,WuTang2017ApJS,Wu2017}, and RMHD \cite{WuTangM3AS,WuShu2020NumMath}.

As discussed in Section \ref{sec:2.1}, the two implicit functions $p(\bm{U})$ and $\bm{v}(\bm{U})$, which are involved in $\mathcal{G}$, are highly nonlinear and cannot be explicitly formulated. Consequently, determining whether a given state $\bm U$ is in $\mathcal G$ and enforcing the BP condition \eqref{CDG1D:GL} are both difficult. This complexity significantly complicates the design of the BP limiter for RMHD.

To construct the BP limiter, we invoke the following equivalent form of the set $\mathcal{G}$, which was proven in \cite{WuTangM3AS} for the ideal EOS and extended to a general EOS in \cite{WuTangZAMP}.

\begin{lemma}\label{Lemma:G0}
	The admissible state set $\mathcal{G}$ has the following explicit equivalent representation:  
	\begin{equation}\label{eq:G0}
		{\mathcal{G}} = \left\{
		\bm{U}=(D, \bm{m}, \bm{B}, E)^\top:~ D>0, ~  q(\bm{U}) >0, ~ \Phi(\bm{U}) > 0 \right\},
	\end{equation}
	where 
	\begin{align*}
		q(\bm{U}) &:=E-\sqrt{D^2+|\bm{m}|^2},
		\\
		\Phi(\bm{U}) &:=\left(\phi(\bm{U})+2E-2|\bm{B}|^2\right)\sqrt{\phi(\bm{U})+|\bm{B}|^2-E}-\sqrt{\frac{27}{2}\left(D^2|\bm{B}|^2+(\bm{m}\cdot \bm{B})^2\right)}
	\end{align*}
	with $\phi(\bm{U}):=\sqrt{(|\bm{B}|^2-E)^2+3(E^2-D^2-|\bm{m}|^2)}$.
\end{lemma}

	To prevent the influence of round-off errors on the BP property, we define
	\begin{equation}\label{eq:Geps}
		{\mathcal{G}}_\varepsilon = \left\{
		\bm{U}=(D, \bm{m}, \bm{B}, E)^\top:~ D\geq \varepsilon_D, ~  q(\bm{U}) \geq\varepsilon_q, ~ \Phi_\varepsilon(\bm{U}) \geq 0 \right\},
	\end{equation}
	where $\Phi_\varepsilon(\bm{U}):=\Phi(\bm{U}_\varepsilon)$ with  $\bm{U}_\varepsilon:=(D,\bm{m},\bm{B},E-\varepsilon)^\top$; $\varepsilon_D$, $\varepsilon_q$, and $\varepsilon$ are small positive numbers and will be specified later. 
	It was shown in \cite{WuTangM3AS} that ${\mathcal{G}}_\varepsilon$ is also a convex set and ${\mathcal{G}}_\varepsilon \subset {\mathcal{G}}=\mathcal{G}_*$.

	We now design a BP limiter to enforce the following conditions: 
	\begin{equation}\label{eq:1DBPcond}
		\bm{U}^I_{h}(x)\in {\mathcal{G}}_\varepsilon ,~~ \bm{U}^J_{h}(x) \in {\mathcal{G}}_\varepsilon \qquad \forall x\in 	\mathcal{S}_j=\widehat{\mathcal{Q}}_j \cup \mathcal{Q}_j, \quad \forall j, 
	\end{equation}	
	where
	$$
	\widehat{\mathcal{Q}}_j = \left\{\widehat{x}_{j-\fourone}^\nu\right\}_{\nu=0}^L \cup
	\left\{\widehat{x}_{j+\fourone}^\nu\right\}_{\nu=0}^L, \quad
	\mathcal{Q}_j = \left\{{x}_{j-\fourone}^\mu\right\}_{\mu=0}^Q \cup \left\{{x}_{j+\fourone}^\mu\right\}_{\mu=1}^{Q}.
	$$
	Our BP limiter is independently performed in each primal and dual mesh cell. 
	For convenience, we detail only the implementation for $\bm{U}^I_h(x)$,  as the implementation for  $\bm{U}^J_{h}(x)$ is identical. 
	Define $\bm{U}^I_j(x):=(D_j(x),\bm{m}_j(x),\bm{B}_j(x),E_j(x))^\top$ and $\overline{\bm{U}}_j^I:=(\overline{D}_j, \overline{\bm{m}}_j, \overline{\bm{B}}_j, \overline{E}_j)^\top$. 
	For each cell $I_{j}$, we modify the polynomial solution $\bm{U}^I_j(x)$ to $\widetilde{\bm{U}}^I_j(x)$ such that $\widetilde{\bm{U}}^I_j(x)\in \mathcal{G}_\varepsilon$ for all $x\in\mathcal{S}_j$, via the following three steps:
	\begin{description}
		\item[Step 1] 	First, enforce the mass density $D\geq\varepsilon_D$ with $\varepsilon_D = \min\{ 10^{-13}, \overline{D}_j \}$. 
		We modify $D_j(x)$ as 
		\begin{equation*}\label{BP:D}
			\widehat{D}_j(x) = \theta_1(D_j(x)-\overline{D}_j) + \overline{D}_j, 
		\end{equation*}
		where 
		$\theta_1 = \min\left\{\frac{\overline{D}_j-\varepsilon_D}{\overline{D}_j-D_{\rm min}},1 \right\}$ and 
		$D_{\rm min} = \underset{x\in\mathcal{S}_j}{\min} D_j(x).$
		Note that ${D}_j(x)$ is limited only when $D_{\rm min} < \varepsilon_D$.
		\item[Step 2] 	Next, enforce $q(\bm{U})\geq\varepsilon_q$ with $\varepsilon_q = \min\{ 10^{-13}, q (\overline{\bm{U}}_j^I) \}$. Denote $\widehat{\bm{U}}_j(x):=(\widehat{D}_j(x),\bm{m}_j(x),\bm{B}_j(x),E_j(x))^\top$. We modify 
		$\widehat{\bm{U}}_j(x)$ as 
		\begin{equation*}\label{BP:qU}
			\check{\bm{U}}_j(x) = (\theta_2(\widehat{D}_j(x)-\overline{D}_j) + \overline{D}_j, \theta_2(\widehat{\bm{m}}_j(x)-\overline{\bm{m}}_j) + \overline{\bm{m}}_j,
			\widehat{\bm{B}}_j(x) ,
			\theta_2({E}_j(x)-\overline{E}_j) + \overline{E}_j)
		\end{equation*}
		with	
		$
		\theta_2 = \min\left\{\frac{q(\overline{\bm{U}}_j)-\varepsilon_q}{q(\overline{\bm{U}}_j)-q_{\rm min}},1 \right\}$ and $  
		q_{\rm min} = \underset{x\in\mathcal{S}_j}{\min} q(\widehat{\bm{U}}_j(x)).
		$
		\item[Step 3] 	Finally, enforce $\Phi_\varepsilon(\bm{U}) \ge 0$ with $\varepsilon = 10^{-13} \overline E_j$. We modify 
		$\check{\bm{U}}_j(x)$ as 
		\begin{equation}\label{BP:phiU}
			\widetilde{\bm{U}}^I_j(x) = \theta_3(\check{\bm{U}}_j(x)-\overline{\bm{U}}_j^I) + \overline{\bm{U}}_j^I,
		\end{equation}
		where  
		$
		\theta_3 = \underset{x\in\mathcal{S}_j}{\min} \tilde{\theta}(x), 
		$ 
		and the function $\tilde{\theta}(x)$ is defined for all $x\in \mathcal{S}_{j}$ as follows:
		\begin{equation*} 
			\tilde{\theta}(x) :=	\begin{cases}
				1, \quad & {\rm if}\quad \Phi_\varepsilon(\check{\bm{U}}_j(x)) \ge 0,
				\\
				\tilde {\theta} \in [0,1) \mbox{ is the unique root of } \Phi_\varepsilon\left({	\tilde{\theta}}(\check{\bm{U}}_j(x)-\overline{\bm{U}}_j^I)+\overline{\bm{U}}_j^I\right), \quad & \mbox{otherwise}. 
			\end{cases}
		\end{equation*}
		The uniqueness of this root in $[0,1)$ is guaranteed by the convexity of ${\mathcal{G}}_\varepsilon$.
	\end{description}

\begin{remark}
	The above BP limiter maintains the discrete DF condition \eqref{CDG1D:B}. 
	Furthermore, such type of limiters  
	may not destroy the high-order accuracy, as demonstrated by \Cref{Ex:Smooth}; some theoretical justification for this can also be found in the references  \cite{zhang2010,zhang2011b,ZHANG2017301}. 
\end{remark}

\begin{remark}
	After applying the BP limiter to the CDG solutions ${\bm{U}}_h^I(x)$ and ${\bm{U}}^J_{h}(x)$, the values of the limited solutions $\widetilde{\bm{U}}_h^I(x)$ and $\widetilde{\bm{U}}^J_{h}(x)$ at all points in $\cup_j \mathcal{S}_j$ belong to $\mathcal{G}$. 
	This ensures the BP condition \eqref{CDG1D:GL} in \Cref{thm:main1D}. 
	Consequently, by 
	 utilizing these limited polynomials $\widetilde{\bm{U}}_j^I(x)$ and $\widetilde{\bm{U}}^J_{j+\halfone}(x)$ to respectively replace ${\bm{U}}_j^I(x)$ and ${\bm{U}}^J_{j+\halfone}(x)$ at each Runge--Kutta stage \eqref{RKTime}, we obtain the fully discrete CDG schemes that are both high-order accurate and satisfy the BP property under the CFL condition \eqref{CDG1D:CFLCond}.
\end{remark}

\begin{remark}\label{rem:1Dmorepoints}
	It is crucial to emphasize that our BP limiter includes the Gauss points $\mathcal{Q}_j$ within $\mathcal{S}_j$ to ensure adherence to the BP condition at these points. Although these Gauss points are not explicitly required in the BP condition \eqref{CDG1D:GL} of \Cref{thm:main1D}, their inclusion is critical for the robust recovery of primitive variables at these locations.
	This necessity arises because the flux must be evaluated at these points, as indicated in \eqref{CDG1D:DisI} and \eqref{CDG1D:DisJ}.
	As we have discussed, for the RMHD system, the flux cannot be explicitly derived from the conserved variables. Consequently, it becomes necessary to first recover the primitive variables from the conserved ones before computing the flux. 
	The success of this recovery process hinges on ensuring that the conserved variables at the Gauss points are within the admissible state set $\mathcal G$.
	This specific requirement is not necessary in the non-relativistic MHD case.
\end{remark}

\section{2D Provably BP and Locally DF CDG Schemes}\label{2DBPCDG}

In this section, we propose 2D provably BP schemes with a suitable locally DF CDG finite element discretization for the modified symmetrizable form \eqref{eq:ModRMHD} of the RMHD equations. Our theoretical analysis will demonstrate that our locally DF approximations and proper discretization of the symmetrization source terms are crucial in achieving the BP property. While we focus exclusively on the 2D case, the proposed schemes and numerical analysis can be directly extended to the three-dimensional scenarios.

Let $(x, y)$ represent the 2D spatial coordinate variables. Assume that the 2D computational domain is uniformly partitioned into the primal mesh cells ${ I_{ij} = (x_{i-\frac{1}{2}}, x_{i+\frac{1}{2}}) \times (y_{j-\frac{1}{2}}, y_{j+\frac{1}{2}}) }$, where $\dx$ and $\dy$ are the constant spatial step-sizes in the $x$ and $y$ directions, respectively. Let $x_i = \frac{1}{2}(x_{i-\frac{1}{2}} + x_{i+\frac{1}{2}})$ and $y_j = \frac{1}{2}(y_{j-\frac{1}{2}} + y_{j+\frac{1}{2}})$. Then the dual mesh cells are defined as ${ J_{i+\frac{1}{2},j+\frac{1}{2}} = (x_{i}, x_{i+1}) \times (y_{j}, y_{j+1}) }$.

\subsection{2D Locally DF CDG Schemes for Modified RMHD Equations}

First, we introduce the locally DF finite element spaces \cite{Li2005,Yakovlev2013}, which are defined as 
\begin{gather*}
	\mathbb{V}_h^{I,k} = \left \{\bm{w}=(w_1,\cdots,w_8)^\top \in [L^2(\Omega)]^8: w_\ell|_{I_{ij}} \in \mathbb{P}^k(I_{ij}), ~~ (w_5,w_6)^\top \in \mathbb{B}^k(I_{ij}) ~~\forall \ell, i , j\right \}, \\
	\mathbb{V}_h^{J,k} = \left\{\bm{u}=(u_1,\cdots,u_8)^\top\in [L^2(\Omega)]^8: u_\ell|_{J_{i+\halfone,j+\halfone}} \in \mathbb{P}^k(J_{i+\halfone,j+\halfone}), ~~ (u_5,u_6)^\top \in \mathbb{B}^k(I_{i+\halfone,j+\halfone}) ~~\forall \ell, i, j \right\},
\end{gather*}
where $\mathbb{P}^k(I_{ij})$ and $\mathbb{P}^k(I_{i+\halfone,j+\halfone})$ are the polynomial of the degree up to $k$ over $I_{ij}$ and $I_{i+\halfone,j+\halfone}$, respectively; the DF polynomial spaces  $\mathbb{B}^k(I_{ij})$ and $\mathbb{B}^k(I_{i+\halfone,j+\halfone})$ are given by 
\begin{gather*}
	\mathbb{B}^k(I_{ij}) = \left\{({w}_5,{w}_6)^\top \in \left[\mathbb{P}^k(I_{ij}) \right]^2: \frac{\partial {w}_5}{\partial x} + \frac{\partial {w}_6}{\partial y}=0 ~~ \forall (x,y)\in I_{ij}\right\},\\
	\mathbb{B}^k(I_{i+\halfone,j+\halfone})=\left\{({u}_5,{u}_6)^\top\in \left[\mathbb{P}^k(I_{i+\halfone,j+\halfone})\right]^2: \frac{\partial {u}_5}{\partial x} + \frac{\partial {u}_6}{\partial y}=0 ~~ \forall (x,y)\in I_{i+\halfone,j+\halfone}\right\}.
\end{gather*}
Different from the locally DF schemes in \cite{Li2005,Yakovlev2013} for the conservative non-relativistic MHD system, 
our locally DF CDG schemes for the 2D modified symmetrizable RMHD equations \eqref{eq:ModRMHD} are defined as follows: find the numerical solutions $\bm{U}_h^I(x,y,t) \in \mathbb{V}_h^{I,k}$ and $\bm{U}_h^J(x,y,t) \in \mathbb{V}_h^{J,k}$ such that for any test functions $\bm{w}(x,y) \in \mathbb{V}_h^{I,k}$ and $\bm{u}(x,y) \in \mathbb{V}_h^{J,k}$, 
\begin{align}
	\int_{I_{ij}} \partial_t(\bm{U}_h^I) \cdot \bm{w} \dd x \dd y 
	=&  \frac{1}{\tau_{\rm max}} \int_{I_{ij}} (\bm{U}_h^J - \bm{U}_h^I)\cdot \bm{w} \dd x \dd y +\int_{I_{ij}} \bm{F}(\bm{U}_h^J) \cdot \nabla \bm{w} \dd x \dd y   \nonumber\\
	&- \int_{y_{j-\halfone}}^{y_{j+\halfone}} \left(\bm{F}_1(\bm{U}_h^J(x_{i+\halfone},y)) \cdot \bm{w}(x_{i+\halfone}^-,y) - \bm{F}_1(\bm{U}_h^J(x_{i-\halfone},y)) \cdot \bm{w}(x_{i-\halfone}^+,y)\right) \dd y    \nonumber\\
	& - \int_{x_{i-\halfone}}^{x_{i+\halfone}} \left(\bm{F}_2(\bm{U}_h^J(x,y_{j+\halfone})) \cdot \bm{w}(x,y_{j+\halfone}^-) - \bm{F}_2(\bm{U}_h^J(x,y_{j-\halfone})) \cdot \bm{w}(x,y_{j-\halfone}^+)\right) \dd x \nonumber\\
	& + \int_{y_{j-\halfone}}^{y_{j+\halfone}} (-\jump{B_{1,h}^J(x_i,y)}) \bm{S}\left(\average{\bm{U}_h^J(x_i,y)}\right) \cdot \bm{w}(x_i,y) \dd y \nonumber\\
	& + \int_{x_{i-\halfone}}^{x_{i+\halfone}}  (-\jump{B_{2,h}^J(x,y_j)}) \bm{S}\left(\average{\bm{U}_h^J(x,y_j)}\right) \cdot \bm{w}(x,y_j) \dd x, \label{CDG2D:IntgI}
\end{align}
\begin{align}
	\int_{I_{i+\halfone,j+\halfone}} \partial_t(\bm{U}_h^J) \cdot \bm{u} \dd x \dd y 
	=&  \frac{1}{\tau_{\rm max}} \int_{I_{i+\halfone,j+\halfone}} (\bm{U}_h^I - \bm{U}_h^J)\cdot \bm{u} \dd x \dd y+\int_{I_{i+\halfone,j+\halfone}} \bm{F}(\bm{U}_h^I) \cdot \nabla \bm{u} \dd x \dd y   \nonumber\\
	&- \int_{y_{j}}^{y_{j+1}} \left(\bm{F}_1(\bm{U}_h^I(x_{i+1},y)) \cdot \bm{w}(x_{i+1}^-,y) - \bm{F}_1(\bm{U}_h^I(x_{i},y)) \cdot \bm{w}(x_{i}^+,y)\right) \dd y   \nonumber\\
	& - \int_{x_{i}}^{x_{i+1}} \left(\bm{F}_2(\bm{U}_h^I(x,y_{j+1})) \cdot \bm{w}(x,y_{j+1}^-) - \bm{F}_2(\bm{U}_h^I(x,y_{j})) \cdot \bm{w}(x,y_{j}^+)\right) \dd x  \nonumber\\
	& + \int_{y_{j}}^{y_{j+1}} (-\jump{B_{1,h}^I(x_{i+\halfone},y)}) \bm{S}\left(\average{\bm{U}_h^I(x_{i+\halfone},y)}\right) \cdot \bm{u}(x_{i+\halfone},y) \dd y \nonumber\\
	& + \int_{x_{i}}^{x_{i+1}} (-\jump{B_{2,h}^I(x,y_{j+\halfone})}) \bm{S}\left(\average{\bm{U}_h^I(x,y_{j+\halfone})}\right) \cdot \bm{u}(x,y_{j+\halfone}) \dd x, \label{CDG2D:IntgJ}
\end{align}
where the discretization of the symmetrization source terms in the modified RMHD system \eqref{eq:ModRMHD} is crucial for the theoretical BP property, and 
the time variable $t$ has been omitted here for convenience. The notations $\jump{\cdot}$ and $\average{\cdot}$ respectively denote the jump and average of the limiting values at a cell interface: 
\begin{gather*}
	\jump{B_{1,h}^J(x_i,y)} := B_{1,h}^J(x_i^+,y) - B_{1,h}^J(x_i^-,y), 
	\quad
	\jump{B_{2,h}^J(x,y_j)}:= B_{2,h}^J(x,y_j^+) - B_{2,h}^J(x,y_j^-), \\
	\average{\bm{U}_h^J(x_i,y)} := \frac{1}{2} \left( \bm{U}_h^J(x_i^+,y) + \bm{U}_h^J(x_i^-,y) \right), 
	\quad
	\average{\bm{U}_h^J(x,y_j)} := \frac{1}{2} \left( \bm{U}_h^J(x,y_j^+) + \bm{U}_h^J(x,y_j^-) \right), \\
	\jump{B_{1,h}^I(x_{i+\halfone},y)} :=B_{1,h}^I(x_{i+\halfone}^+,y) - B_{1,h}^I(x_{i+\halfone}^-,y), 
	\quad
	\jump{B_{2,h}^I(x,y_{j+\halfone})}:= B_{2,h}^I(x,y_{j+\halfone}^+) - B_{2,h}^I(x,y_{j+\halfone}^-), \\
	\average{\bm{U}_h^I(x_{i+\halfone},y)} := \frac{1}{2} \left( \bm{U}_h^I(x_{i+\halfone}^+,y) + \bm{U}_h^I(x_{i+\halfone}^-,y) \right), 
	\quad
	\average{\bm{U}_h^I(x,y_{j+\halfone})} := \frac{1}{2} \left( \bm{U}_h^I(x,y_{j+\halfone}^+) + \bm{U}_h^I(x,y_{j+\halfone}^-) \right).
\end{gather*}

To implement the locally DF CDG schemes, it is necessary to split the conservative vector $\bm{U}$ into two parts:
\begin{equation*}
	\bm{U}^{R} = (D, \bm{m}, B_3, E)^\top, \quad \bm{U}^{B} = (B_1, B_2)^\top.
\end{equation*}
Accordingly, we partition the fluxes $\bm{F}_\ell(\bm{U})$ into $\bm{F}_\ell^{R}(\bm{U})$ and $\bm{F}_\ell^{B}(\bm{U})$ for $\ell = 1, 2$. Similarly, $\bm{S}(\bm{U})$ is divided into $\bm{S}^{R}(\bm{U})$ and $\bm{S}^{B}(\bm{U})$. The decomposition of the numerical solution $\bm{U}_h^{I}$ is represented by $\bm{U}_h^{R,I}$ and $\bm{U}_h^{B,I}$. Likewise, the decomposition of the numerical solution $\bm{U}_h^{J}$ is denoted by $\bm{U}_h^{R,J}$ and $\bm{U}_h^{B,J}$.

Let $\{\Psi_{ij}^{(\ell)}\}_{\ell=0}^{K_R}$ and $\{\Psi_{i+\halfone,j+\halfone}^{(\ell)}\}_{\ell=0}^{K_R}$ be the local orthogonal bases of the polynomial spaces $\mathbb{P}^{k}(I_{ij})$ and $\mathbb{P}^{k}(I_{i+\halfone,j+\halfone})$, respectively. Here, $K_R=\frac{k(k+3)}{2}$. 
For example, the orthogonal basis of the polynomial space $\mathbb{P}^{k}(I_{i_0,j_0})$ can be taken as 
\begin{equation*}
	\begin{aligned}
		&\Psi^{(0)}_{i_0,j_0}(x,y)=1, \quad 
		\Psi^{(1)}_{i_0,j_0}(x,y) = \xi_{i_0}, \quad 
		\Psi^{(2)}_{i_0,j_0}(x,y) = \eta_{j_0}, \quad
		\Psi^{(3)}_{i_0,j_0}(x,y)=\xi_{i_0}^2-\frac{1}{3}, \\
		&\Psi^{(4)}_{i_0,j_0}(x,y)=\xi_{i_0} \eta_{j_0}, \quad 
		\Psi^{(5)}_{i_0,j_0}(x,y)=\eta_{j_0}^2-\frac{1}{3}, \quad
		\Psi^{(6)}_{i_0,j_0}(x,y)=\xi_{i_0}^3-\frac{3}{5}\xi_{i_0}, \\
		&\Psi^{(7)}_{i_0,j_0}(x,y)=(\xi_{i_0}^2-\frac{1}{3})\eta_{j_0} \quad
		\Psi^{(8)}_{i_0,j_0}(x,y)=\xi_{i_0}(\eta_{j_0}^2-\frac{1}{3}) \quad
		\Psi^{(9)}_{i_0,j_0}(x,y)=\eta_{j_0}^3-\frac{3}{5}\eta_{j_0}, ~~
		\dots
	\end{aligned}
\end{equation*}
where $\xi_{i_0} = 2(x-x_{i_0})/{\dx}$, $\eta_{j_0} = 2(y-y_{j_0})/{\dy}$, and $(i_0,j_0)=(i,j)$ or $(i+\halfone,j+\halfone)$.
Then, $\bm{U}_h^{R,I}(x,y,t)$ and $\bm{U}_h^{R,J}(x,y,t)$ can be expressed as 
\begin{align}\label{CDG2D:SolRI}
	\bm{U}_h^{R,I}(x,y,t) &= \sum_{\ell=0}^{K_R} \bm{U}_{ij}^{R,I,(\ell)}(t) \Psi_{ij}^{(\ell)}(x,y) \quad \forall (x,y)\in I_{ij},
\\
\label{CDG2D:SolRJ}
	\bm{U}_h^{R,J}(x,y,t) &= \sum_{\ell=0}^{K_R} \bm{U}_{i+\halfone,j+\halfone}^{R,J,(\ell)}(t) \Psi_{i+\halfone,j+\halfone}^{(\ell)}(x,y) \quad \forall (x,y)\in I_{i+\halfone,j+\halfone},
\end{align} 
where the degrees of freedom are 
\begin{align*}
	\bm{U}_{ij}^{R,I,(\ell)}(t)&=({D}_{ij}^{R,I,(\ell)}, \bm{m}_{ij}^{R,I,(\ell)}, {(B_3)}_{ij}^{R,I,(\ell)}, {E}_{ij}^{R,I,(\ell)})^\top, \\
	\bm{U}_{i+\halfone,j+\halfone}^{R,J,(\ell)}(t)&=({D}_{i+\halfone,j+\halfone}^{R,J,(\ell)},\bm{m}_{i+\halfone,j+\halfone}^{R,J,(\ell)},{(B_3)}_{i+\halfone,j+\halfone}^{R,J,(\ell)},{E}_{i+\halfone,j+\halfone}^{R,J,(\ell)})^\top.
\end{align*}
Let  $\{\bm{\Psi}_{ij}^{(\ell)} \}_{\ell=0}^{K_B}$ and  $\{\bm{\Psi}_{i+\halfone,j+\halfone}^{(\ell)} \}_{\ell=0}^{K_B}$ be the local orthogonal bases of the DF polynomial spaces $\mathbb{B}^k(I_{ij})$ and $\mathbb{B}^k(I_{i+\halfone,j+\halfone})$, respectively. Here, $K_B=\frac{(k+1)(k+4)}{2}-1$ is the space dimensionality. Specifically, these bases (cf.~\cite{Li2005}) can be taken as 
\begin{equation*}
	\begin{aligned}
		&\bm{\Psi}^{(0)}_{i_0,j_0} = 
		\begin{pmatrix}
			0 \\
			1
		\end{pmatrix},
		\quad
		\bm{\Psi}^{(1)}_{i_0,j_0} = 
		\begin{pmatrix}
			1 \\
			0
	\end{pmatrix}, \quad 
		\bm{\Psi}^{(2)}_{i_0,j_0} = 
		\begin{pmatrix}
			0 \\
			\xi_{i_0}
	\end{pmatrix}, \quad 
		\bm{\Psi}^{(3)}_{i_0,j_0} = 
		\begin{pmatrix}
			\eta_{j_0} \\
			0
	\end{pmatrix}, \quad 
		\bm{\Psi}^{(4)}_{i_0,j_0} = 
		\begin{pmatrix}
			\dx \xi_{i_0} \\
			-\dy \eta_{j_0}
	\end{pmatrix}, \\
		&\bm{\Psi}^{(5)}_{i_0,j_0} = 
		\begin{pmatrix}
			0 \\
			\xi_{i_0}^2-\frac{1}{3}
	\end{pmatrix}, \quad
		\bm{\Psi}^{(6)}_{i_0,j_0} = 
		\begin{pmatrix}
			\eta_{j_0}^2-\frac{1}{3} \\
			0
	\end{pmatrix}, \quad  
		\bm{\Psi}^{(7)}_{i_0,j_0} = 
		\begin{pmatrix}
			\dx(\xi_{i_0}^2-\frac{1}{3}) \\
			-2\dy \xi_{i_0} \eta_{j_0}
	\end{pmatrix}, \quad 
		\bm{\Psi}^{(8)}_{i_0,j_0} = 
		\begin{pmatrix}
			-2\dx \xi_{i_0} \eta_{j_0} \\
			\dy (\eta_{j_0}^2-\frac{1}{3})
		\end{pmatrix}, \\
		&\bm{\Psi}^{(9)}_{i_0,j_0} = 
		\begin{pmatrix}
			0 \\
			\xi_{i_0}^3-\frac{3}{5}\xi_{i_0}
	\end{pmatrix}, \quad
		\bm{\Psi}^{(10)}_{i_0,j_0} = 
		\begin{pmatrix}
			\eta_{j_0}^3-\frac{3}{5}\eta_{j_0}\\
			0
	\end{pmatrix}, \quad
		\bm{\Psi}^{(11)}_{i_0,j_0} = 
		\begin{pmatrix}
			\frac{\dx}{3}(\xi_{i_0}^3-\xi_{i_0})\\
			-\dy(\xi_{i_0}^2-\frac{1}{3})\eta_{j_0}
	\end{pmatrix}, \\
		&\bm{\Psi}^{(12)}_{i_0,j_0} = 
		\begin{pmatrix}
			\dx\xi_{i_0}(\eta_{j_0}^2-\frac{1}{3})\\
			-\frac{\dy}{3}(\eta_{j_0}^3-\eta_{j_0})
	\end{pmatrix}, \quad
		\bm{\Psi}^{(13)}_{i_0,j_0} = 
	\begin{pmatrix}
			\dx(\xi_{i_0}^2-\frac{1}{3})\eta_{j_0}\\
			-\dy\xi_{i_0}(\eta_{j_0}^2-\frac{1}{3})
	\end{pmatrix}, \quad
		\dots
	\end{aligned}
\end{equation*}
Then, $\bm{U}_h^{B,I}(x,y,t)$ and $\bm{U}_h^{B,J}(x,y,t)$ can be expressed as 
\begin{equation}\label{CDG2D:SolBI}
	\bm{U}_h^{B,I}(x,y,t) = \sum_{\ell =0}^{K_B} {U}_{ij}^{B,I,(\ell)}(t) \bm{\Psi}_{ij}^{(\ell)}(x,y) \quad \forall (x,y)\in I_{ij},
\end{equation} 
\begin{equation}\label{CDG2D:SolBJ}
	\bm{U}_h^{B,J}(x,y,t) = \sum_{\ell =0}^{K_B} {U}_{i+\halfone,j+\halfone}^{B,J,(\ell)}(t) \bm{\Psi}_{i+\halfone,j+\halfone}^{(\ell)}(x,y) \quad \forall (x,y)\in I_{i+\halfone,j+\halfone},
\end{equation} 
where the scalars ${U}_{ij}^{B,I,(\ell)}(t)$ and ${U}_{i+\halfone,j+\halfone}^{B,J,(\ell)}(t)$ denote the degrees of freedom.

For the $\mathbb{P}^k$-based CDG scheme, we approximate the flux integrals in \eqref{CDG2D:IntgI} and \eqref{CDG2D:IntgJ} using the Gauss quadrature rule of $Q=k+1$ points to meet the algebraic precision requirement. 
Let $\{{y}_{j\pm\fourone}^\mu\}_{\mu=1}^Q$ denote the Gauss quadrature nodes in the interval $[y_{j\pm\fourone}-\frac{\dx}{4}, y_{j\pm\fourone}+\frac{\dx}{4}]$, where $y_{j\pm\fourone} = y_j + \frac14 \dy$.  
Substituting \eqref{CDG2D:SolRI} and \eqref{CDG2D:SolRJ} into \eqref{CDG2D:IntgI} and \eqref{CDG2D:IntgJ},  we obtain, for $\ell=0,\cdots, K_R$, that 
\begin{align}\label{CDG2D:DisRI}
	&{a}_{ij}^{R,(\ell)} \frac{\dd \bm{U}_{ij}^{R,I,(\ell)}(t)}{\dd t} 
	= \frac{1}{\tau_{\rm max} \dx \dy} \int_{I_{ij}} (\bm{U}_h^{R,J} - \bm{U}_h^{R,I}) \Psi_{ij}^{(\ell)}(x,y) \dd x \dd y \nonumber\\
	& ~~~ + \sum_{m=\pm1, \sigma=\pm1} \sum_{\mu=1}^{Q}  \sum_{\tilde{\mu}=1}^{Q} \frac{\omega_{\mu}}{2} \frac{\omega_{\tilde{\mu}}}{2} \bm{F}^R\left(\bm{U}_h^J(x_{i+\mfourone}^\mu,y_{j+\sfourone}^{\tilde{\mu}})\right) \cdot \nabla \Psi_{ij}^{(\ell)}(x_{i+\mfourone}^\mu,y_{j+\sfourone}^{\tilde{\mu}}) \nonumber\\
	& ~~~ - \frac{1}{\dx} \sum_{m=\pm1} \sum_{\mu=1}^{Q} \frac{\omega_{\mu}}{2} \left( \bm{F}_1^R(\bm{U}_h^J(x_{i+\halfone},y_{j+\mfourone}^\mu))  \Psi_{ij}^{(\ell)}(x_{i+\halfone}^-,y_{j+\mfourone}^\mu) \right.  
	\left.- \bm{F}_1^R(\bm{U}_h^J(x_{i-\halfone},y_{j+\mfourone}^\mu))  \Psi_{ij}^{(\ell)}(x_{i-\halfone}^+,y_{j+\mfourone}^\mu) \right)    \nonumber\\
	& ~~~ - \frac{1}{\dy} \sum_{m=\pm1} \sum_{\mu=1}^{Q} \frac{\omega_{\mu}}{2} \left( \bm{F}_2^R(\bm{U}_h^J(x_{i+\mfourone}^\mu,y_{j+\halfone}))  \Psi_{ij}^{(\ell)}(x_{i+\mfourone}^\mu,y_{j+\halfone}^-) \right. 
	\left. - \bm{F}_2^R(\bm{U}_h^J(x_{i+\mfourone}^\mu,y_{j-\halfone}))  \Psi_{ij}^{(\ell)}(x_{i+\mfourone}^\mu,y_{j-\halfone}^+)\right) \nonumber\\
	& ~~~ + \frac{1}{\dx} \sum_{m=\pm1} \sum_{\mu=1}^{Q} \frac{\omega_{\mu}}{2}  (-\jump{B_{1,h}^J(x_i,y_{j+\mfourone}^\mu)}) \bm{S}^R\left(\average{\bm{U}_h^J(x_i,y_{j+\mfourone}^\mu)}\right) \Psi_{ij}^{(\ell)}(x_i,y_{j+\mfourone}^\mu) \nonumber\\
	& ~~~ + \frac{1}{\dy} \sum_{m=\pm1} \sum_{\mu=1}^{Q} \frac{\omega_{\mu}}{2}  (-\jump{B_{2,h}^J(x_{i+\mfourone}^\mu,y_j)}) \bm{S}^R\left(\average{\bm{U}_h^J(x_{i+\mfourone}^\mu,y_j)}\right) \Psi_{ij}^{(\ell)}(x_{i+\mfourone}^\mu,y_j) , 
\end{align}
\begin{align}\label{CDG2D:DisRJ}
	&{a}_{i+\halfone,j+\halfone}^{R,(\ell)} \frac{\dd \bm{U}_{i+\halfone,j+\halfone}^{R,J,(\ell)}(t)}{\dd t}
	=  \frac{1}{\tau_{\rm max} \dx \dy} \int_{I_{i+\halfone,j+\halfone}} (\bm{U}_h^{R,I} - \bm{U}_h^{R,J}) \Psi_{i+\halfone,j+\halfone}^{(\ell)}(x,y) \dd x \dd y \nonumber\\
	& ~~~  + \sum_{m=\pm1, \sigma=\pm1} \sum_{\mu=1}^{Q} \sum_{\tilde{\mu}=1}^{Q} \frac{\omega_{\mu}}{2} \frac{\omega_{\tilde{\mu}}}{2} \bm{F}^R\left(\bm{U}_h^I(x_{i+\halfone+\mfourone}^\mu,y_{j+\halfone+\sfourone}^{\tilde{\mu}})\right) \cdot \nabla \Psi_{i+\halfone,j+\halfone}^{(\ell)}(x_{i+\halfone+\mfourone}^\mu,y_{j+\halfone+\sfourone}^{\tilde{\mu}}) \nonumber\\
	&~~~ - \frac{1}{\dx} \sum_{m=\pm1} \sum_{\mu=1}^{Q} \frac{\omega_{\mu}}{2} \left( \bm{F}_1^R(\bm{U}_h^{I}(x_{i+1},y_{j+\halfone+\mfourone}^\mu)) \Psi_{i+\halfone,j+\halfone}^{(\ell)}(x_{i+1}^-,y_{j+\halfone+\mfourone}^\mu) \right. 
	\left. - \bm{F}_1^R(\bm{U}_h^{I}(x_{i},y_{j+\halfone+\mfourone}^\mu)) \Psi_{i+\halfone,j+\halfone}^{(\ell)}(x_{i}^+,y_{j+\halfone+\mfourone}^\mu)\right) \nonumber\\
	& ~~~  - \frac{1}{\dy} \sum_{m=\pm1} \sum_{\mu=1}^{Q} \frac{\omega_{\mu}}{2} \left( \bm{F}_2^R(\bm{U}_h^{I}(x_{i+\halfone+\mfourone}^\mu,y_{j+1})) \Psi_{i+\halfone,j+\halfone}^{(\ell)}(x_{i+\halfone+\mfourone}^\mu,y_{j+1}^-) \right. 
	\left. - \bm{F}_2^R(\bm{U}_h^{I}(x_{i+\halfone+\mfourone}^\mu,y_{j})) \Psi_{i+\halfone,j+\halfone}^{(\ell)}(x_{i+\halfone+\mfourone}^\mu,y_{j}^+) \right),  \nonumber\\
	& ~~~  + \frac{1}{\dx} \sum_{m=\pm1} \sum_{\mu=1}^{Q} \frac{\omega_{\mu}}{2} (-\jump{B_{1,h}^I(x_{i+\halfone},y_{j+\halfone+\mfourone}^\mu)}) \bm{S}^R\left(\average{\bm{U}_h^I(x_{i+\halfone},y_{j+\halfone+\mfourone}^\mu)}\right) \Psi_{i+\halfone,j+\halfone}^{(\ell)}(x_{i+\halfone},y_{j+\halfone+\mfourone}^\mu)  \nonumber\\
	& ~~~ + \frac{1}{\dy} \sum_{m=\pm1} \sum_{\mu=1}^{Q} \frac{\omega_{\mu}}{2}  (-\jump{B_{2,h}^I(x_{i+\halfone+\mfourone}^\mu,y_{j+\halfone})}) \bm{S}^R\left(\average{\bm{U}_h^I(x_{i+\halfone+\mfourone}^\mu,y_{j+\halfone})}\right) \Psi_{i+\halfone,j+\halfone}^{(\ell)}(x_{i+\halfone+\mfourone}^\mu,y_{j+\halfone}) ,
\end{align}
where 
$${a}_{ij}^{R,(\ell)} = \frac{1}{\dx \dy}\int_{I_{ij}} \left(\Psi_{ij}^{(\ell)}(x,y)\right)^2 \dd x \dd y, \quad 
{a}_{i+\halfone,j+\halfone}^{R,(\ell)} =\frac{1}{\dx \dy}\int_{J_{i+\halfone,j+\halfone}} \left(\Psi_{i+\halfone,j+\halfone}^{(\ell)}(x,y)\right)^2 \dd x \dd y.$$
Similarly, for $\ell=0,\cdots,K_B$, we derive
\begin{align}\label{CDG2D:DisBI}
	{a}_{ij}^{B,(\ell)} \frac{\dd {U}_{ij}^{B,I,(\ell)}(t)}{\dd t} 
	&= \frac{1}{\tau_{\rm max} \dx \dy} \int_{I_{ij}} (\bm{U}_h^{B,J} - \bm{U}_h^{B,I}) \cdot \bm{\Psi}_{ij}^{(\ell)}(x,y) \dd x \dd y \nonumber \\
	& + \sum_{m=\pm1, \sigma=\pm1} \sum_{\mu=1}^{Q}  \sum_{\tilde{\mu}=1}^{Q} \frac{\omega_{\mu}}{2} \frac{\omega_{\tilde{\mu}}}{2} \bm{F}^B\left(\bm{U}_h^J(x_{i+\mfourone}^\mu,y_{j+\sfourone}^{\tilde{\mu}})\right) : \nabla \bm{\Psi}_{ij}^{(\ell)}(x_{i+\mfourone}^\mu,y_{j+\sfourone}^{\tilde{\mu}}) \nonumber \\
	&- \frac{1}{\dx} \sum_{m=\pm1} \sum_{\mu=1}^{Q} \frac{\omega_{\mu}}{2} \left( \bm{F}_1^B(\bm{U}_h^J(x_{i+\halfone},y_{j+\mfourone}^\mu))  \cdot \bm{\Psi}_{ij}^{(\ell)}(x_{i+\halfone}^-,y_{j+\mfourone}^\mu) \right. 
	\left.- \bm{F}_1^B(\bm{U}_h^J(x_{i-\halfone},y_{j+\mfourone}^\mu))  \cdot \bm{\Psi}_{ij}^{(\ell)}(x_{i-\halfone}^+,y_{j+\mfourone}^\mu) \right)     \nonumber \\
	& - \frac{1}{\dy} \sum_{m=\pm1} \sum_{\mu=1}^{Q} \frac{\omega_{\mu}}{2} \left( \bm{F}_2^B(\bm{U}_h^J(x_{i+\mfourone}^\mu,y_{j+\halfone}))  \cdot \bm{\Psi}_{ij}^{(\ell)}(x_{i+\mfourone}^\mu,y_{j+\halfone}^-) \right.  
	\left. - \bm{F}_2^B(\bm{U}_h^J(x_{i+\mfourone}^\mu,y_{j-\halfone}))  \cdot \bm{\Psi}_{ij}^{(\ell)}(x_{i+\mfourone}^\mu,y_{j-\halfone}^+)\right) \nonumber \\
	& + \frac{1}{\dx} \sum_{m=\pm1} \sum_{\mu=1}^{Q} \frac{\omega_{\mu}}{2}  (-\jump{B_{1,h}^J(x_i,y_{j+\mfourone}^\mu)}) \bm{S}^B\left(\average{\bm{U}_h^J(x_i,y_{j+\mfourone}^\mu)}\right) \cdot \bm{\Psi}_{ij}^{(\ell)}(x_i,y_{j+\mfourone}^\mu) \nonumber \\
	& + \frac{1}{\dy} \sum_{m=\pm1} \sum_{\mu=1}^{Q} \frac{\omega_{\mu}}{2}  (-\jump{B_{2,h}^J(x_{i+\mfourone}^\mu,y_j)}) \bm{S}^B\left(\average{\bm{U}_h^J(x_{i+\mfourone}^\mu,y_j)}\right) \cdot \bm{\Psi}_{ij}^{(\ell)}(x_{i+\mfourone}^\mu,y_j) , 
\end{align}
\begin{align}\label{CDG2D:DisBJ}
	&a_{i+\halfone,j+\halfone}^{B,(\ell)} \frac{\dd {U}_{i+\halfone,j+\halfone}^{B,J,(\ell)}(t)}{\dd t} 
	= \frac{1}{\tau_{\rm max} \dx \dy} \int_{I_{i+\halfone,j+\halfone}} (\bm{U}_h^{B,I} - \bm{U}_h^{B,J}) \cdot \bm{\Psi}_{i+\halfone,j+\halfone}^{(\ell)}(x,y) \dd x \dd y \nonumber\\
	& ~~~  + \sum_{m=\pm1, \sigma=\pm1} \sum_{\mu=1}^{Q}  \sum_{\tilde{\mu}=1}^{Q} \frac{\omega_{\mu}}{2} \frac{\omega_{\tilde{\mu}}}{2}  \bm{F}^B\left(\bm{U}_h^I(x_{i+\halfone+\mfourone}^\mu,y_{j+\halfone+\sfourone}^{\tilde{\mu}})\right) : \nabla \bm{\Psi}_{i+\halfone,j+\halfone}^{(\ell)}(x_{i+\halfone+\mfourone}^\mu,y_{j+\halfone+\sfourone}^{\tilde{\mu}}) \nonumber\\
	&~~~ - \frac{1}{\dx} \sum_{m=\pm1} \sum_{\mu=1}^{Q} \frac{\omega_{\mu}}{2} \left( \bm{F}_1^B(\bm{U}_h^{I}(x_{i+1},y_{j+\halfone+\mfourone}^\mu)) \cdot \bm{\Psi}_{i+\halfone,j+\halfone}^{(\ell)}(x_{i+1}^-,y_{j+\halfone+\mfourone}^\mu) \right. 
	\left. - \bm{F}_1^B(\bm{U}_h^{I}(x_{i},y_{j+\halfone+\mfourone}^\mu)) \cdot \bm{\Psi}_{i+\halfone,j+\halfone}^{(\ell)}(x_{i}^+,y_{j+\halfone+\mfourone}^\mu) \right)  \nonumber \\
	& ~~~  - \frac{1}{\dy} \sum_{m=\pm1} \sum_{\mu=1}^{Q} \frac{\omega_{\mu}}{2} \left( \bm{F}_2^B(\bm{U}_h^{I}(x_{i+\halfone+\mfourone}^\mu,y_{j+1})) \cdot \bm{\Psi}_{i+\halfone,j+\halfone}^{(\ell)}(x_{i+\halfone+\mfourone}^\mu,y_{j+1}^-) \right. \left. - \bm{F}_2^B(\bm{U}_h^{I}(x_{i+\halfone+\mfourone}^\mu,y_{j})) \cdot \bm{\Psi}_{i+\halfone,j+\halfone}^{(\ell)}(x_{i+\halfone+\mfourone}^\mu,y_{j}^+) \right)  \nonumber\\
	& ~~~  + \frac{1}{\dx} \sum_{m=\pm1} \sum_{\mu=1}^{Q} \frac{\omega_{\mu}}{2} (-\jump{B_{1,h}^I(x_{i+\halfone},y_{j+\halfone+\mfourone}^\mu)}) \bm{S}^B\left(\average{\bm{U}_h^I(x_{i+\halfone},y_{j+\halfone+\mfourone}^\mu)}\right) \cdot \bm{\Psi}_{i+\halfone,j+\halfone}^{(\ell)}(x_{i+\halfone},y_{j+\halfone+\mfourone}^\mu)  \nonumber\\
	& ~~~ + \frac{1}{\dy} \sum_{m=\pm1} \sum_{\mu=1}^{Q} \frac{\omega_{\mu}}{2}  (-\jump{B_{2,h}^I(x_{i+\halfone+\mfourone}^\mu,y_{j+\halfone})}) \bm{S}^B\left(\average{\bm{U}_h^I(x_{i+\halfone+\mfourone}^\mu,y_{j+\halfone})}\right) \cdot \bm{\Psi}_{i+\halfone,j+\halfone}^{(\ell)}(x_{i+\halfone+\mfourone}^\mu,y_{j+\halfone}),
\end{align}
where ``$:$'' denotes the Frobenius inner product of two matrices, and 
$${a}_{ij}^{B,(\ell)} = \frac{1}{\dx \dy} \int_{I_{ij}} \left|\bm{\Psi}_{ij}^{(\ell)}(x,y)\right|^2 \dd x \dd y, \quad 
{a}_{i+\halfone,j+\halfone}^{B,(\ell)} = \frac{1}{\dx \dy} \int_{I_{i+\halfone,j+\halfone}} \left|\bm{\Psi}_{i+\halfone,j+\halfone}^{(\ell)}(x,y)\right|^2 \dd x \dd y.$$

Coupling the semi-discrete CDG schemes \eqref{CDG2D:DisRI}-\eqref{CDG2D:DisBJ} with 
an SSP time discretization, for example, the third-order SSP Runge--Kutta method \eqref{RKTime}, 
 we obtain the fully-discrete, high-order accurate, locally DF CDG schemes for the 2D modified symmetrizable RMHD system \eqref{eq:ModRMHD}.

\subsection{Evolution Equations for Cell Averages in Our Locally DF CDG Schemes}

Define the 2D cell averages $\overline{\bm{U}}_{ij}^I = \frac{1}{\dx \dy}\int_{I_{ij}} \bm{U}_h^I \dd x \dd y$ and $\overline{\bm{U}}_{i+\halfone,j+\halfone}^J = \frac{1}{\dx \dy}\int_{J_{i+\halfone, j+\halfone}} \bm{U}_h^J \dd x \dd y$ over the cells $I_{ij}$ and $J_{i+\halfone,j+\halfone}$, respectively. 
Due to the local orthogonality of our adopted basis, we have 
$$\overline{\bm{U}}_{ij}^I=({D}_{ij}^{R,I,(0)},\bm{m}_{ij}^{R,I,(0)},{U}_{ij}^{B,I,(0)},{U}_{ij}^{B,I,(1)},{(B_3)}_{ij}^{R,I,(0)},{E}_{ij}^{R,I,(0)})^\top,$$ $$\overline{\bm{U}}_{i+\halfone,j+\halfone}^J=({D}_{i+\halfone,j+\halfone}^{R,J,(0)},\bm{m}_{i+\halfone,j+\halfone}^{R,J,(0)},{U}_{i+\halfone,j+\halfone}^{B,J,(0)},{U}_{i+\halfone,j+\halfone}^{B,J,(1)},{(B_3)}_{i+\halfone,j+\halfone}^{R,J,(0)},{E}_{i+\halfone,j+\halfone}^{R,J,(0)})^\top.$$
Taking the equations \eqref{CDG2D:DisRI}--\eqref{CDG2D:DisRJ} with $\ell=0$ and selecting the equations \eqref{CDG2D:DisBI}--\eqref{CDG2D:DisBJ} for $\ell=0,1$, 
we then obtain the evolution equations for the cell averages in our locally DF CDG schemes: 
\begin{equation}\label{CDG2D:ave}
	\frac{\dd \overline{\bm{U}}_{ij}^I}{\dd t} =\mathcal{L}_{ij}^I(\bm{U}_h^I, \bm{U}_h^J), 
	\qquad
	\frac{\dd \overline{\bm{U}}_{i+\halfone,j+\halfone}^J}{\dd t} =\mathcal{L}_{i+\halfone,j+\halfone}^J(\bm{U}_h^J,\bm{U}_h^I)
\end{equation}
with 
\begin{equation}\label{WKL2}
	\mathcal{L}_{ij}^I(\bm{U}_h^I, \bm{U}_h^J) := \mathcal{H}_{ij}^I(\bm{U}_h^I, \bm{U}_h^J) + \mathcal{S}^J_{ij}, \quad
	\mathcal{L}_{i+\halfone,j+\halfone}^J(\bm{U}_h^J,\bm{U}_h^I) := \mathcal{H}_{i+\halfone,j+\halfone}^J(\bm{U}_h^J, \bm{U}_h^I) +  \mathcal{S}^I_{i+\halfone,j+\halfone}.
\end{equation}
Here, $\mathcal{H}_{ij}^I(\bm{U}_h^I,\bm{U}_h^J)$ and $\mathcal{H}_{i+\halfone,j+\halfone}^J(\bm{U}_h^J, \bm{U}_h^I)$ are defined by 
\begin{equation}\label{CDG2D:aveHI}  
	\begin{aligned}
		\mathcal{H}_{ij}^I(\bm{U}_h^I, \bm{U}_h^J) :=& \frac{\overline{\bm{U}}_{ij}^J-\overline{\bm{U}}_{ij}^I}{\tau_{\rm max}}  
		- \frac{1}{\dx} \sum_{m=\pm1}\sum_{\mu=1}^{Q} \frac{\omega_{\mu}}{2} \left(\bm{F}_1(\bm{U}^{J,\mu}_{i+\halfone,j+\mfourone}) - \bm{F}_1(\bm{U}^{J,\mu}_{i-\halfone,j+\mfourone})\right)  \\
		&- \frac{1}{\dy} \sum_{m=\pm1} \sum_{\mu=1}^{Q} \frac{\omega_{\mu}}{2} \left(\bm{F}_2(\bm{U}^{J,\mu}_{i+\mfourone,j+\halfone}) - \bm{F}_2(\bm{U}^{J,\mu}_{i+\mfourone,j-\halfone})\right), 
	\end{aligned}
\end{equation}
\begin{equation}\label{CDG2D:aveHJ}
	\begin{aligned}
		\mathcal{H}_{i+\halfone,j+\halfone}^J(\bm{U}_h^J,\bm{U}_h^I) :=& \frac{\overline{\bm{U}}_{i+\halfone,j+\halfone}^I-\overline{\bm{U}}_{i+\halfone,j+\halfone}^J}{\tau_{\rm max}} 
		- \frac{1}{\dx} \sum_{m=\pm1} \sum_{\mu=1}^{Q} \frac{\omega_{\mu}}{2} \left(\bm{F}_1(\bm{U}^{I,\mu}_{i+1,j+\halfone+\mfourone}) - \bm{F}_1(\bm{U}^{I,\mu}_{i,j+\halfone+\mfourone})\right) \\
		&- \frac{1}{\dy} \sum_{m=\pm1} \sum_{\mu=1}^{Q} \frac{\omega_{\mu}}{2} \left(\bm{F}_2(\bm{U}^{I,\mu}_{i+\halfone+\mfourone,j+1}) - \bm{F}_2(\bm{U}^{I,\mu}_{i+\halfone+\mfourone,j})\right)
	\end{aligned}
\end{equation}
with 
$$
\bm{U}^{J,\mu}_{i\pm \halfone,j+\mfourone} = \bm{U}_h^{J}(x_{i\pm\halfone},y_{j+\mfourone}^{\mu}), \quad
\bm{U}^{J,\mu}_{i+\mfourone,j\pm \halfone} = \bm{U}_h^{J}(x_{i+\mfourone}^\mu,y_{j\pm\halfone}),
$$
$$
\bm{U}^{I,\mu}_{i+\halfone \pm \halfone,j+\halfone+\mfourone} = \bm{U}_h^{I}(x_{i+\halfone \pm \halfone},y_{j+\halfone+\mfourone}^{\mu}), \quad
\bm{U}^{I,\mu}_{i+\halfone+\mfourone,j+\halfone \pm \halfone} = \bm{U}_h^{I}(x_{i+\halfone+\mfourone}^{\mu},y_{j+\halfone \pm \halfone}).
$$ 
The discrete source terms in \eqref{WKL2} are given by 
\begin{equation}\label{CDG2D:SourI}
	\begin{aligned}
		\mathcal{S}_{ij}^J =& - \frac{1}{\dx}\sum_{m=\pm1} \sum_{\mu=1}^{Q} \frac{\omega_{\mu}}{2}  \jump{B_{1,h}^J(x_i,y_{j+\mfourone}^{\mu})} \bm{S}\left(\average{\bm{U}_h^J(x_i,y_{j+\mfourone}^{\mu})}  \right)  \\
		&- \frac{1}{\dy} \sum_{m=\pm1} \sum_{\mu=1}^{Q} \frac{\omega_{\mu}}{2} \jump{B_{2,h}^J(x_{i+\mfourone}^{\mu},y_j)} \bm{S}\left(\average{\bm{U}_h^J(x_{i+\mfourone}^{\mu},y_j)}  \right), 
	\end{aligned}
\end{equation}
\begin{equation}\label{CDG2D:SourJ}
	\begin{aligned}
		\mathcal{S}_{i+\halfone,j+\halfone}^I =& 
		-  \frac{1}{\dx}\sum_{m=\pm1} \sum_{\mu=1}^{Q} \frac{\omega_{\mu}}{2}  \jump{B_{1,h}^I(x_{i+\halfone},y_{j+\halfone+\mfourone}^{\mu})} \bm{S}\left(\average{\bm{U}_h^I(x_{i+\halfone},y_{j+\halfone+\mfourone}^{\mu})}  \right) \\
		&-  \frac{1}{\dy} \sum_{m=\pm1} \sum_{\mu=1}^{Q} \frac{\omega_{\mu}}{2}  \jump{B_{2,h}^I(x_{i+\halfone+\mfourone}^{\mu},y_{j+\halfone})} \bm{S}\left(\average{\bm{U}_h^I(x_{i+\halfone+\mfourone}^{\mu},y_{j+\halfone})}  \right).
	\end{aligned}
\end{equation}

\subsection{Cell Average Decomposition (CAD) on Overlapping Meshes}
In this section, we introduce CAD and its extension on 2D overlapping meshes, which will play a critical role in analyzing and designing provably BP high-order CDG schemes. 
CAD splits the cell average of the numerical solution into a convex combination of certain nodal point values within the cell.
It forms the cornerstone of the theoretical framework for rigorous BP analysis \cite{zhang2010,CuiDingWu2023JCP,CuiDingWu2022SINUM} and also determines the theoretical CFL condition of the resulting BP schemes and the computational cost of the associated BP limiter.
In the 1D case, we have used the CAD \eqref{eq:1D} based on the Gauss-Lobatto quadrature, which was originally proposed by Zhang and Shu in \cite{zhang2010}. This 1D CAD was proven to be optimal for non-central DG schemes \cite{CuiDingWu2022SINUM}, in the sense of attaining the mildest BP CFL condition. 

The 2D CADs are much more complex than the 1D ones. As shown in \cite{CuiDingWu2022SINUM}, 
a generally feasible 2D symmetric CAD on the reference cell $\Omega:=[-1,1]^2$ can be defined as 
\begin{equation}\label{eq:2DCAD}
	\begin{split} 
		\frac14 \int_{\Omega} f(x,y) \dd x \dd y 
			= & \bar \omega  
		\sum_{\mu=1}^Q   \omega_\mu \left(  \frac{  \lambda_1  }{  \lambda } 
		\Big(  f(-1,  y^{\mu}  ) +   f( 1,   y^{\mu}  )   \Big)  +  \frac{  \lambda_2  }{  \lambda } 
		\Big(	 f(  x^{\mu} , -1 ) +  f( x^{\mu} , 1  )   \Big) \right)
		\\
		& + \sum_{s=1}^S \tilde \omega_s f( \tilde x^s, \tilde y^s ) \qquad \forall f \in \mathbb P^k ( \Omega ),
	\end{split}
\end{equation}
where $\lambda_i \ge 0$ will be specified in the BP analysis with $\lambda = \lambda_1 + \lambda_2>0$; 
the weights $\bar \omega > 0$ and $\widetilde \omega_s \ge 0$ satisfy $2 \bar \omega +  \sum_{s=1}^S \tilde \omega_s =1$; both $\{x^{\mu} \}_{\mu=1}^Q$ and $\{y^{\mu} \}_{\mu=1}^Q$ denote the $Q$-point Gauss quadrature nodes in $[-1,1]$; 
the set of internal nodes $\Omega^{\rm int}:=\{(\tilde x^s, \tilde y^s)\}_{s=1}^S \subset \Omega$ for all $s$.  
In the following, we give two examples of such 2D symmetric CADs on the reference cell $\Omega=[-1,1]^2$. 

\begin{expl}[2D Zhang--Shu CAD \cite{zhang2010}]\label{ex:zs}
	This 2D CAD is constructed by a tensor product of the 1D Gauss quadrature and 1D Gauss--Lobatto quadrature. 
	The corresponding weights and internal nodes are given by  
		\begin{equation}\label{CAD:1b}
			\begin{cases} \displaystyle
		\bar{\omega} = \frac{1}{L(L-1)} \quad \mbox{with}\quad  L=\left\lceil\frac{k+3}{2} \right\rceil,
		\\ \displaystyle
		\sum_{s=1}^S \tilde \omega_s f( \tilde x^s, \tilde y^s ) = \sum_{\nu=2}^{L-1} \sum_{\mu=1}^Q 
		\widehat \omega_\nu   \omega_\mu \left(
		\frac{\lambda_1 }{\lambda} f(\widehat x^\nu , y^\mu) + 
		\frac{\lambda_2}{\lambda} f(x^\mu, \widehat y^\nu ) 
		\right),
		\end{cases}
	\end{equation}
	where  $S = 2Q(L-2)$, and both $\{\widehat x^{\nu} \}_{\nu=1}^L$ and $\{\widehat y^{\nu} \}_{\nu=1}^L$ denote the $L$-point Gauss--Lobatto quadrature nodes in $[-1,1]$. The nodes are illustrated in Figure \ref{fig:CAD1a}. 
\end{expl}

\begin{expl}[2D Cui--Ding--Wu CAD \cite{CuiDingWu2023JCP,CuiDingWu2022SINUM}]\label{ex:cdw}
	We take the cases of $k=2$ and $k=3$ as example, while the CAD for more general polynomial spaces can be found in \cite{CuiDingWu2022SINUM}. In these two cases, the corresponding weights and internal nodes are given by  
\begin{align} \label{CAD:2b}
				\bar{\omega} &= \frac{1}{4+2|\delta|}, \qquad \tilde \omega_{1} = \tilde \omega_{2} = \frac{1+|\delta|}{2(2+|\delta|)}, \qquad S=2, \qquad \delta := \frac{\lambda_1 -\lambda_2}{\lambda}, 
			\\  
	\left(\tilde x^1,\tilde y^1\right) &= 
\begin{cases} \displaystyle
	\left(\sqrt{\frac{2|\delta|}{3+3|\delta|}}, 0\right), {~\rm if~} \delta\in[-1,0], \\ \displaystyle
	\left(0, \sqrt{\frac{2|\delta|}{3+3|\delta|}} \right), {~\rm if~} \delta\in[0,1], 
\end{cases}
\quad 
\left(\tilde x^2, \tilde y^2\right) = 
\begin{cases} \displaystyle
	\left(-\sqrt{\frac{2|\delta|}{3+3|\delta|}}, 0\right), {~\rm if~} \delta\in[-1,0], \\ \displaystyle
	\left(0, -\sqrt{\frac{2|\delta|}{3+3|\delta|}} \right), {~\rm if~} \delta \in[0,1].
\end{cases}
\end{align}
If $\lambda_1=\lambda_2$, then $\delta=0$, and the two internal nodes $(\tilde x^1,\tilde y^1)$ and $(\tilde x^2,\tilde y^2)$ merges to a singe node $(0,0)$; in this case, $\bar{\omega} = \frac{1}{4}$, and 
$\sum_{s=1}^S \tilde \omega_s f( \tilde x^s, \tilde y^s )=\frac12 f(0,0)$.  
The nodes of this CAD are illustrated in Figure \ref{fig:CAD2a}. 
\end{expl}

The 2D Cui--Ding--Wu CAD was proven to be optimal for BP non-central DG schemes, in the sense of achieving the mildest BP CFL condition in theory \cite{CuiDingWu2023JCP,CuiDingWu2022SINUM}. Moreover, it requires much fewer internal nodes, thereby needing less computational cost in the associated BP limiting procedure, compared to the 2D Zhang--Shu CAD. 



	\begin{figure}[!htb]
	\centering	
	\begin{subfigure}{0.48\textwidth}
		\centering
		\includegraphics[width=0.8\textwidth]{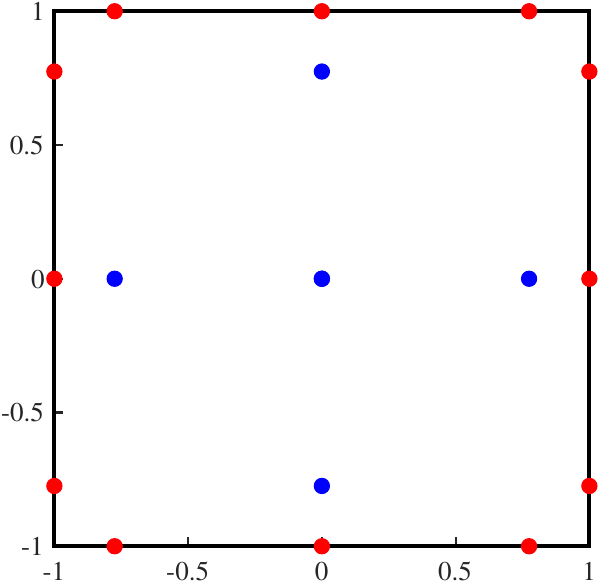}
		\caption{Zhang--Shu CAD}
		\label{fig:CAD1a}
	\end{subfigure}
	\begin{subfigure}{0.48\textwidth}
		\includegraphics[width=0.8\textwidth]{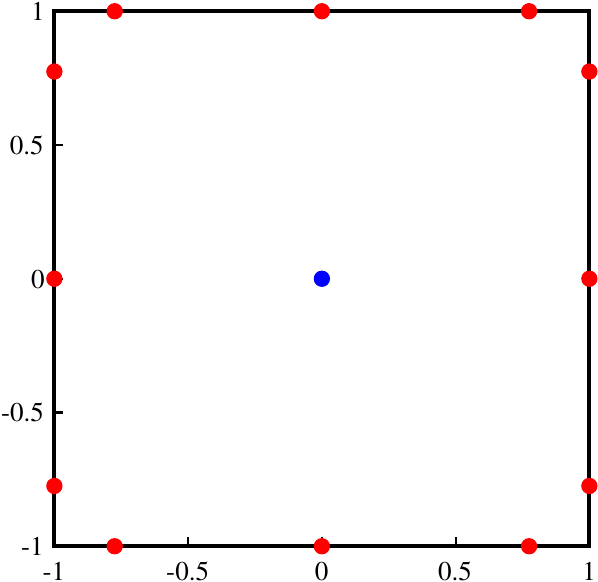}
		\caption{Cui-Ding-Wu CAD}
		\label{fig:CAD2a}
	\end{subfigure}
	\caption{Nodes of the Zhang--Shu CAD in \Cref{ex:zs} and the Cui--Ding--Wu CAD in \Cref{ex:cdw} on  the reference cell $\Omega=[-1,1]^2$ for $k=2$. Red: boundary nodes; blue: internal nodes. 
}
\label{fig:CADref}
\end{figure}

Next, we would like to extend the 2D CAD techniques to overlapping meshes, which will be useful in our BP analysis for CDG schemes. The BP analysis requires us to seek a 
2D CAD that decomposes the cell averages 
$$\overline{\bm{U}}_{ij}^J:=\frac{1}{\dx \dy} \int_{I_{ij}} {\bm{U}}_h^J(x,y) \dd x \dd y, \qquad \overline{\bm{U}}_{i+\halfone,j+\halfone}^I:= \frac{1}{\dx \dy} \int_{J_{i+\frac12,j+\frac12}} {\bm{U}}_h^I(x,y) \dd x \dd y $$ 
into some convex combinations of specific point values. 
However, the CDG solution ${\bm{U}}_h^J$ (resp. ${\bm{U}}_h^I$) is a discontinuous piecewise polynomial function on the target cell $I_{ij}$ (resp. $J_{i+\frac12,j+\frac12}$).    
This is different from and more difficult than the non-central DG case. 
An natural way to address this difficulty is by splitting the integration on the target cell into four integrals of polynomials on four quarter cells, for which the regular 2D CAD can be applied. 
For example, the cell average $\overline{\bm{U}}_{ij}^J$ can be split as 
$$
\overline{\bm{U}}_{ij}^J = \frac{1}4  \sum_{m=\pm 1, \sigma\pm1} \frac{4}{\dx \dy} \int_{\Omega_{i+\mfourone,j+\sfourone}} {\bm{U}}_h^J(x,y) \dd x \dd y, 
$$
where $\Omega_{i+\mfourone,j+\sfourone}:=[x_{i+\mfourone}-\frac{\dx}{4}, x_{i+\mfourone}+\frac{\dx}{4}] \times [y_{j+\sfourone}-\frac{\dy}{4}, y_{j+\sfourone}+\frac{\dy}{4}]$. Transforming the CAD \eqref{eq:2DCAD} on the reference cell to the quarter cell $\Omega_{i+\mfourone,j+\sfourone}$ gives 
\begin{equation*}
	\begin{split} 
		\frac{4}{\dx \dy} \int_{\Omega_{i+\mfourone,j+\sfourone}} {\bm{U}}_h^J(x,y) \dd x \dd y  
		= & \bar \omega  
		\sum_{\mu=1}^Q   \omega_\mu \bigg[ \frac{  \lambda_1  }{  \lambda } 
		\Big(  {\bm{U}}_h^J (x_{i+\mfourone}-\frac{\dx}{4},  y_{j+\sfourone}^{\mu}  ) +   {\bm{U}}_h^J( x_{i+\mfourone}+\frac{\dx}{4},   y_{j+\sfourone}^{\mu}  )   \Big) 
		\\
		& \qquad  +  \frac{  \lambda_2  }{  \lambda } 
		\Big(	 {\bm{U}}_h^J(  x_{i+\mfourone}^{\mu} , y_{j+\sfourone}-\frac{\dy}{4} ) +  {\bm{U}}_h^J( x_{i+\mfourone}^{\mu} , y_{j+\sfourone}+\frac{\dy}{4}  )   \Big) \bigg] 
		\\
		& + \sum_{s=1}^S \tilde \omega_s {\bm{U}_{h}^J( \tilde x_{i+\mfourone,j+\sfourone}^{s}, \tilde y_{i+\mfourone,j+\sfourone}^s) },
	\end{split}
\end{equation*}
where 
$
\tilde x_{i+\mfourone,j+\sfourone}^{s} = x_{i+\mfourone}  +  \frac{1}4 \tilde x^s \dx  $ and $ \tilde y_{i+\mfourone,j+\sfourone}^{s} = y_{j+\sfourone} +  \frac{1}4 \tilde y^s \dy. 
$ 
Then all the nodes for the 2D CAD on the target cell $I_{ij} = \cup_{m=\pm 1} \cup_{\sigma\pm1} \Omega_{i+\mfourone,j+\sfourone}$ form the set 
\begin{equation}\label{thm:2DQ}
	\mathcal{Q}_{ij}= \bigcup_{m=\pm 1} \bigcup_{\sigma\pm1}  \mathcal{Q}_{i+\mfourone,j+\sfourone}
\end{equation}
with 
$$	\mathcal{Q}_{i+\mfourone,j+\sfourone} :=\left\{ (x_{i+\mfourone}\pm \frac{\dx}{4},  y_{j+\sfourone}^{\mu}  ) \right\}_{\mu=1}^Q \bigcup  
\left\{ (  x_{i+\mfourone}^{\mu} , y_{j+\sfourone}\pm \frac{\dy}{4} ) \right\}_{\mu=1}^Q \bigcup \left\{   ( \tilde x_{i+\mfourone,j+\sfourone}^{s},\tilde y_{i+\mfourone,j+\sfourone}^s) \right \}_{s=1}^S.$$ 
Therefore, we obtain the following extended 2D CAD on overlapping meshes: 
\begin{equation}\label{eq:CADI}
	\overline{\bm{U}}_{ij}^J 
	= {\bm \Pi}_{ij}^{J,1} + \frac{\bar{\omega}}{2} {\bm \Pi}_{ij}^{J,2} + \frac{\bar{\omega}}{2} {\bm \Pi}_{ij}^{J,3},
\end{equation}
where $\bm{U}^{J,\pm,\mu}_{i,j+\mfourone} := \bm{U}^J_h(x_i^{\pm}, y_{j+\mfourone}^{\mu})$, $\bm{U}^{J,\mu,\pm}_{i+\mfourone,j} := \bm{U}^J_h(x_{i+\mfourone}^{\mu}, y_{j}^{\pm})$, and 
\begin{align*}
	&{\bm \Pi}_{ij}^{J,1} :=  \sum_{m=\pm 1, \sigma\pm1} \sum_{s=1}^{S} \frac{\omega_s}{4} {\bm{U}_{h}^J(\tilde x_{i+\mfourone,j+\sfourone}^{s},\tilde y_{i+\mfourone,j+\sfourone}^s) }, \\
	&{\bm \Pi}_{ij}^{J,2} :=  \sum_{m=\pm1}\sum_{\mu=1}^{Q} \frac{\omega_{\mu}}{2} 
	\left( \frac{\lambda_1}{\lambda} \left(\bm{U}^{J,-,\mu}_{i,j+\mfourone} + \bm{U}^{J,+,\mu}_{i,j+\mfourone}\right)  
	+ \frac{\lambda_2}{\lambda} \left(\bm{U}^{J,\mu,-}_{i+\mfourone,j} + \bm{U}^{J,\mu,+}_{i+\mfourone,j}\right) \right) , \\
	&{\bm \Pi}_{ij}^{J,3} := \sum_{m=\pm1}\sum_{\mu=1}^{Q} \frac{\omega_{\mu}}{2} 
	\left( \frac{\lambda_1}{\lambda} \left(\bm{U}^{J,\mu}_{i-\halfone,j+\mfourone} + \bm{U}^{J,\mu}_{i+\halfone,j+\mfourone}\right)  
	+ \frac{\lambda_2}{\lambda} \left(\bm{U}^{J,\mu}_{i+\mfourone,j-\halfone} + \bm{U}^{J,\mu}_{i+\mfourone,j+\halfone}\right) \right). 
\end{align*}
Similarly, we have 
\begin{equation}\label{eq:CADJ}
	\overline{\bm{U}}_{i+\halfone,j+\halfone}^I =  \frac{1}{\dx \dy} \int_{I_{i+\halfone,j+\halfone}} \bm{U}_{h}^I(x,y) \dd x \dd y 
	= {\bm \Pi}_{i+\halfone,j+\halfone}^{I,1} + \frac{\bar{\omega}}{2} {\bm \Pi}_{i+\halfone,j+\halfone}^{I,2} + \frac{\bar{\omega}}{2} {\bm \Pi}_{i+\halfone,j+\halfone}^{I,3},
\end{equation}
where $
	\bm{U}^{I,\pm,\mu}_{i+\halfone,j+\halfone+\mfourone} := \bm{U}^I_h(x_{i+\halfone}^{\pm}, y_{j+\halfone+\mfourone}^{\mu})$, $
	\bm{U}^{I,\mu,\pm}_{i+\halfone+\mfourone,j+\halfone} := \bm{U}^I_h(x_{i+\halfone+\mfourone}^{\mu}, y_{j+\halfone}^{\pm})$, and 
\begin{align*}
	&{\bm \Pi}_{i+\halfone,j+\halfone}^{I,1} :=  \sum_{m=\pm 1, \sigma\pm1} \sum_{s=1}^{S} \frac{\omega_s}{4} {\bm{U}_{h}^I(\tilde x_{i+\halfone+\mfourone,j+\halfone+\sfourone}^{s},\tilde y_{i+\halfone+\mfourone,j+\halfone+\sfourone}^s)},  \\ 
	&{\bm \Pi}_{i+\halfone,j+\halfone}^{I,2} := \sum_{m=\pm1}\sum_{\mu=1}^{Q} \frac{\omega_{\mu}}{2} \left( \frac{\lambda_1}{\lambda}  \left(\bm{U}^{I,-,\mu}_{i+\halfone,j+\halfone+\mfourone} + \bm{U}^{I,+,\mu}_{i+\halfone,j+\halfone+\mfourone}\right)  
	+ \frac{\lambda_2}{\lambda} \left(\bm{U}^{I,\mu,-}_{i+\halfone+\mfourone,j+\halfone} + \bm{U}^{I,\mu,+}_{i+\halfone+\mfourone,j+\halfone}\right) \right), \\
	&{\bm \Pi}_{i+\halfone,j+\halfone}^{I,3} := \sum_{m=\pm1}\sum_{\mu=1}^{Q} \frac{\omega_{\mu}}{2} \left( \frac{\lambda_1}{\lambda} \left(\bm{U}^{I,\mu}_{i,j+\halfone+\mfourone} + \bm{U}^{I,\mu}_{i+1,j+\halfone+\mfourone}\right)  
	+ \frac{\lambda_2}{\lambda} \left(\bm{U}^{I,\mu}_{i+\halfone+\mfourone,j} + \bm{U}^{I,\mu}_{i+\halfone+\mfourone,j+1}\right) \right). 
\end{align*}
The extensions of the Zhang--Shu CAD and the Cui--Ding--Wu CAD on overlapping meshes are illustrated in \Cref{fig:CADInt}.

	\begin{figure}[htbp]
	\centering	
	\begin{subfigure}{0.48\textwidth}
		\centering
		\includegraphics[width=0.8\textwidth]{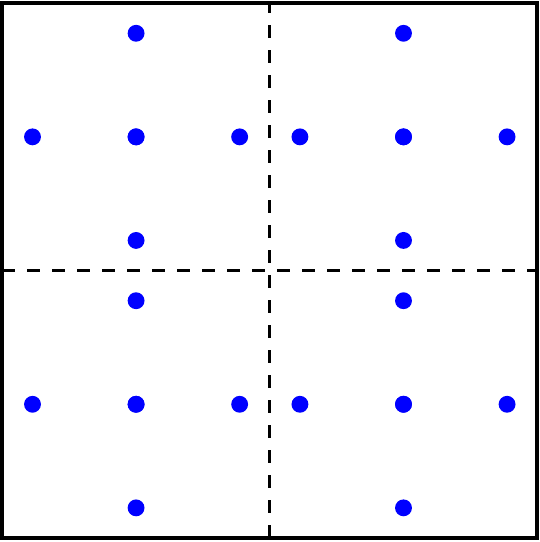}
		\caption{Extended Zhang--Shu CAD for $k=2$}
		\label{fig:CAD1}
	\end{subfigure}
	\begin{subfigure}{0.48\textwidth}
		\includegraphics[width=0.8\textwidth]{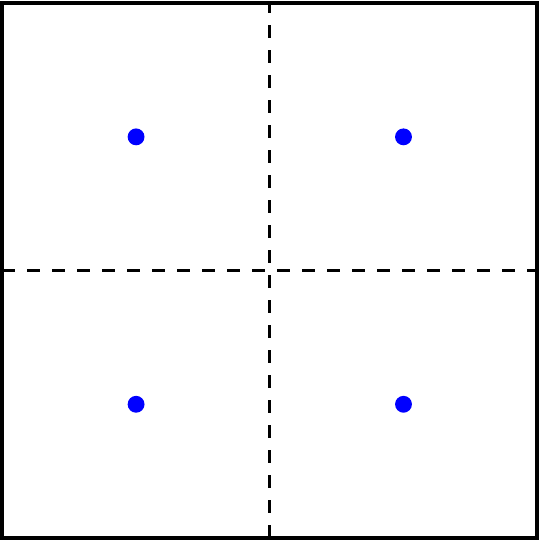}
		\caption{Extended Cui--Ding--Wu CAD for both $k=2$ and $k=3$}
		\label{fig:CAD2}
	\end{subfigure}
	
	\caption{Internal nodes of the Zhang--Shu CAD and the Cui--Ding--Wu CAD extended to overlapping meshes. 
	}
	\label{fig:CADInt}
\end{figure}

\subsection{Rigorous BP Analysis of Our Locally DF CDG Schemes}



We are now ready to rigorously analyze the BP property of the proposed locally DF CDG schemes. 

\begin{theorem}\label{thm:main2D}
	Assume $\overline{\bm{U}}_{ij}^I \in \mathcal{G}$ and $\overline{\bm{U}}_{i+\halfone,j+\halfone}^J \in \mathcal{G}$.
	If the numerical solutions $\bm{U}_h^I(x,y)$ and $\bm{U}_h^J(x,y)$ satisfy 
	\begin{equation}\label{CDG2D:QuaCond}
		\bm{U}_h^I(x,y) \in \mathcal{G}, \quad \bm{U}_h^J(x,y) \in \mathcal{G} \quad \forall(x,y) \in \mathcal{Q}_{ij}, \quad \forall i,j
	\end{equation}
	with the set $\mathcal{Q}_{ij}$ defined in \eqref{thm:2DQ}, 
	then 
	the updated cell averages satisfy
	\begin{equation*}\label{CDG2D:Update}
		\overline{\bm{U}}_{ij}^{I,\dt} := 
		\overline{\bm{U}}_{ij}^I + \dt \mathcal{L}_{ij}^I(\bm{U}_h^I,\bm{U}_h^J) \in \mathcal{G},
		\quad
		\overline{\bm{U}}_{i+\halfone,j+\halfone}^{J,\dt} := 
		\overline{\bm{U}}_{i+\halfone,j+\halfone}^J + \dt \mathcal{L}_{i+\halfone,j+\halfone}^J(\bm{U}_h^J,\bm{U}_h^I) \in \mathcal{G}
		\qquad \forall i,j 
	\end{equation*}
	under the CFL condition
	\begin{equation}\label{CDG2D:CFLCond}
		0<\dt\left(\frac{\alpha_1}{\dx} + \frac{\alpha_2}{\dy}\right) \le \frac{\theta \bar{\omega}}{2}, \quad 
		\theta = \frac{\dt}{\tau_{\rm max}} \in (0,1],
	\end{equation}
	where $\alpha_1 = \max\{1,\beta_1\}$ and $\alpha_2 = \max\{1,\beta_2\}$ with 
	\begin{equation}\label{CDG2D:beta}
		\begin{aligned}
			\beta_1 = \underset{i,j,\mu}{\max} \underset{m=\pm1}{\max} \left\{ \frac{\big|\jump{B_{1,h}^J(x_i,y^\mu_{j+\frac{m}{4}})}\big|}{2\sqrt{\average{(\rho h) (x_i,y^\mu_{j+\frac{m}{4}})}}}, \frac{\big|\jump{B_{1,h}^I(x_{i+\halfone},y^\mu_{j+\frac12+\frac{m}{4}})}\big|}{2\sqrt{\average{(\rho h) (x_{i+\halfone},y^\mu_{j+\frac12+\frac{m}{4}})}}} \right\}, \\
			\beta_2 = \underset{i,j,\mu}{\max} \underset{m=\pm1}{\max} \left\{ \frac{\big|\jump{B_{2,h}^J(x^\mu_{i+\frac{m}{4}},y_j)}\big|}{2\sqrt{\average{(\rho h) (x^\mu_{i+\frac{m}{4}},y_j)}}}, \frac{\big|\jump{B_{2,h}^I(x^\mu_{i+\frac12+\frac{m}{4}},y_{j+\halfone})}\big|}{2\sqrt{\average{(\rho h) (x^\mu_{i+\frac12+\frac{m}{4}},y_{j+\halfone})}}} \right\}.
		\end{aligned}
	\end{equation}
	
\end{theorem}

\begin{proof}
	We only present the proof of $\overline{\bm{U}}_{ij}^{I,\dt} \in \mathcal{G}$, as the proof of $\overline{\bm{U}}_{i+\halfone,j+\halfone}^{J,\dt} \in \mathcal{G}$ is analogous and thus omitted. 
	Define 
	\begin{equation}\label{pf:2DPIF}
		\begin{aligned}
			{\bm \Pi}_{F}^{J} :=& - \frac{1}{\dx} \sum_{m=\pm1}\sum_{\mu=1}^{Q} \frac{\omega_{\mu}}{2} \left(\bm{F}_1(\bm{U}^{J,\mu}_{i+\halfone,j+\mfourone}) - \bm{F}_1(\bm{U}^{J,\mu}_{i-\halfone,j+\mfourone})\right) 
			- \frac{1}{\dy} \sum_{m=\pm1} \sum_{\mu=1}^{Q} \frac{\omega_{\mu}}{2} \left(\bm{F}_2(\bm{U}^{J,\mu}_{i+\mfourone,j+\halfone}) - \bm{F}_2(\bm{U}^{J,\mu}_{i+\mfourone,j-\halfone})\right).
		\end{aligned}
	\end{equation}
	Using \eqref{WKL2}, \eqref{CDG2D:aveHI}, and \eqref{eq:CADI} gives 
	\begin{align}\nonumber
			\overline{\bm{U}}_{ij}^{I,\dt} 
			& = \overline{\bm{U}}_{ij}^I + \dt \mathcal{L}_{ij}^I(\bm{U}_h^I,\bm{U}_h^J) 
			\\ \nonumber 
			&= (1-\theta)\overline{\bm{U}}_{ij}^{I} + \theta \overline{\bm{U}}_{ij}^{J} + \dt {\bm \Pi}_{F}^{J} + \dt \mathcal{S}_{ij}^{J}, \\ \label{pf:CDG2DUbar}
			&=(1-\theta)\overline{\bm{U}}_{ij}^{I} + \theta{\bm \Pi}_{ij}^{J,1} + \frac{\theta \bar{\omega}}{2} ({\bm \Pi}_{ij}^{J,2} + {\bm \Pi}_{ij}^{J,3}) + \dt {\bm \Pi}_{F}^{J} + \dt \mathcal{S}_{ij}^{J},  
	\end{align}
	where we have employed in the last step the extended 2D CAD \eqref{eq:CADI} with $\lambda_1 = \frac{ \alpha_1}{\dx}$ and $\lambda_2 =  \frac{\alpha_2}{\dy}$. 
	Under the condition \eqref{CDG2D:QuaCond} and using the convexity of $\mathcal G$, we have  
	\begin{equation}\label{pf:CDG2Dcon1}
		\overline{\bm{U}}_{ij}^{I} \in \mathcal{G}, \quad \frac{1}{1-2\bar{\omega}}{\bm \Pi}_{ij}^{J,1} \in \mathcal{G}, \quad \frac{1}{2}{\bm \Pi}_{ij}^{J,2} \in \mathcal{G}, \quad \frac{1}{2}{\bm \Pi}_{ij}^{J,3} \in \mathcal{G}.
	\end{equation}

	We first prove that $\overline{\bm{U}}_{ij}^{I,\dt} \cdot \bm{n}_1 >0$. Using \eqref{eq:GQLCor2_1} in \Cref{Coro:Flux2} gives 
	\begin{equation}\label{pf:2DPIFdotn1}
		\begin{aligned}
			{\bm \Pi}_{F}^{J} \cdot \bm{n}_1 
			& \geq - \frac{1}{\dx} \sum_{m=\pm1}\sum_{\mu=1}^{Q} \frac{\omega_{\mu}}{2} \left(\bm{U}^{J,\mu}_{i+\halfone,j+\mfourone} + \bm{U}^{J,\mu}_{i-\halfone,j+\mfourone}\right) \cdot \bm{n}_1 
			- \frac{1}{\dy} \sum_{m=\pm1} \sum_{\mu=1}^{Q} \frac{\omega_{\mu}}{2} \left(\bm{U}^{J,\mu}_{i+\mfourone,j+\halfone} + \bm{U}^{J,\mu}_{i+\mfourone,j-\halfone}\right) \cdot \bm{n}_1 \\
			& \geq - \frac{\alpha_1}{\dx} \sum_{m=\pm1}\sum_{\mu=1}^{Q} \frac{\omega_{\mu}}{2} \left(\bm{U}^{J,\mu}_{i+\halfone,j+\mfourone} + \bm{U}^{J,\mu}_{i-\halfone,j+\mfourone}\right) \cdot \bm{n}_1 
			- \frac{\alpha_2}{\dy} \sum_{m=\pm1} \sum_{\mu=1}^{Q} \frac{\omega_{\mu}}{2} \left(\bm{U}^{J,\mu}_{i+\mfourone,j+\halfone} + \bm{U}^{J,\mu}_{i+\mfourone,j-\halfone}\right) \cdot \bm{n}_1 \\
			& =  (\lambda_1 + \lambda_2)  {\bm \Pi}_{ij}^{J,3}  \cdot \bm{n}_1 = \lambda  {\bm \Pi}_{ij}^{J,3}  \cdot \bm{n}_1.
		\end{aligned}
	\end{equation}
	This, along with \eqref{pf:CDG2DUbar} and $\bm{S}_{ij}^J\cdot \bm{n}_1=0$, implies  
	\begin{equation}\label{pf:Udotn1}
		\begin{aligned}
			\overline{\bm{U}}_{ij}^{I,\dt} \cdot \bm{n}_1 \geq (1-\theta)\overline{\bm{U}}_{ij}^{I} \cdot \bm{n}_1 + \theta{\bm \Pi}_{ij}^{J,1} \cdot \bm{n}_1 + \frac{\theta \bar{\omega}}{2} {\bm \Pi}_{ij}^{J,2} \cdot \bm{n}_1 + \left(\frac{\theta \bar{\omega}}{2}- {\lambda}\Delta t\right) {\bm \Pi}_{ij}^{J,3} \cdot \bm{n}_1 >0,
		\end{aligned}
	\end{equation}
where we have used \eqref{pf:CDG2Dcon1} and the CFL constraint \eqref{CDG2D:CFLCond}. 

	Next, we will prove that $\overline{\bm{U}}_{ij}^{I,\dt} \cdot \bm{n}^* + p^*_m >0$ for any free auxiliary variables $\bm{v}^* \in \mathbb{B}_1(\bm 0)$ and $\bm{B}^*\in \mathbb{R}^3$. 
	It follows from \eqref{pf:CDG2DUbar} that 
	\begin{align}\nonumber
	 \overline{\bm{U}}_{ij}^{I,\dt} \cdot \bm{n}^* + p^*_m  = &	(1-\theta) (\overline{\bm{U}}_{ij}^{I} \cdot \bm{n}^* + p^*_m ) + (1-2 \bar{\omega}) \theta \left( \frac{1}{1-2 \bar{\omega}} {\bm \Pi}_{ij}^{J,1} \cdot \bm{n}^* + p^*_m\right) 
	 \\ \nonumber
	 & + \frac{\theta \bar{\omega}}{2} ({\bm \Pi}_{ij}^{J,2} \cdot \bm{n}^* + 2 p^*_m  + {\bm \Pi}_{ij}^{J,3} \cdot \bm{n}^* + 2 p^*_m) + \dt {\bm \Pi}_{F}^{J} \cdot \bm{n}^* + \dt \mathcal{S}_{ij}^{J} \cdot \bm{n}^* 
	 \\ \label{WKL33} 
	  > & \frac{\theta \bar{\omega}}{2} ({\bm \Pi}_{ij}^{J,2} \cdot \bm{n}^* + 2 p^*_m  + {\bm \Pi}_{ij}^{J,3} \cdot \bm{n}^* + 2 p^*_m) + \dt {\bm \Pi}_{F}^{J} \cdot \bm{n}^* + \dt \mathcal{S}_{ij}^{J} \cdot \bm{n}^*. 
	\end{align}
	We first estimate the lower bounds for ${\bm \Pi}_{F}^{J} \cdot \bm{n}^*$ and $\bm{S}_{ij}^J \cdot \bm{n}^*$, respectively. 
	Using \eqref{eq:GQLCor2_2} in \Cref{Coro:Flux2} gives  
	\begin{align}\label{pf:2DPIFdotnstar}
		{\bm \Pi}_{F}^{J} \cdot \bm{n}^* 
		& \geq -\frac{1}{\dx} \sum_{m=\pm1}\sum_{\mu=1}^{Q} \frac{\omega_{\mu}}{2} \left( (\bm{U}^{J,\mu}_{i+\halfone,j+\mfourone} + \bm{U}^{J,\mu}_{i-\halfone,j+\mfourone}) \cdot \bm{n}^* + 2 p^*_m 
		+ \left((B_1)_{i+\halfone,j+\mfourone}^{J,\mu}-(B_1)_{i-\halfone,j+\mfourone}^{J,\mu}\right) (\bm{v}^*\cdot \bm{B}^*)\right)  \nonumber\\  
		& ~~~~- \frac{1}{\dy} \sum_{m=\pm1} \sum_{\mu=1}^{Q} \frac{\omega_{\mu}}{2} \left( (\bm{U}^{J,\mu}_{i+\mfourone,j+\halfone} + \bm{U}^{J,\mu}_{i+\mfourone,j-\halfone} ) \cdot \bm{n}^* + 2 p^*_m 
		+  \left((B_2)_{i+\mfourone,j+\halfone}^{J,\mu}-(B_2)_{i+\mfourone,j-\halfone}^{J,\mu}\right) (\bm{v}^*\cdot \bm{B}^*)\right)  \nonumber\\
		& \geq 
		- \frac{\alpha_1}{\dx} \sum_{m=\pm1}\sum_{\mu=1}^{Q} \frac{\omega_{\mu}}{2} \left( (\bm{U}^{J,\mu}_{i+\halfone,j+\mfourone} + \bm{U}^{J,\mu}_{i-\halfone,j+\mfourone}) \cdot \bm{n}^* + 2 p^*_m\right)  \nonumber\\  
		& ~~~~ - \frac{\alpha_2}{\dy} \sum_{m=\pm1} \sum_{\mu=1}^{Q} \frac{\omega_{\mu}}{2} \left( (\bm{U}^{J,\mu}_{i+\mfourone,j+\halfone} + \bm{U}^{J,\mu}_{i+\mfourone,j-\halfone} ) \cdot \bm{n}^* + 2 p^*_m \right) - (\bm{v}^*\cdot \bm{B}^*) {\rm div}_{ij}\bm{B}_h^J   \nonumber\\ 
		& = - \lambda \left({\bm \Pi}_{ij}^{J,3} \cdot \bm{n}^* + 2p^*_m\right) - (\bm{v}^*\cdot \bm{B}^*) {\rm div}_{ij}\bm{B}_h^J,
	\end{align}
	where 
	\begin{equation}\label{eq:2DDivB_J1}
		{\rm div}_{ij}\bm{B}_h^J := \sum_{m=\pm1} \sum_{\mu=1}^{Q} \frac{\omega_{\mu}}{2} \left(\frac{(B_1)_{i+\halfone,j+\mfourone}^{J,\mu}-(B_1)_{i-\halfone,j+\mfourone}^{J,\mu}}{\dx} + \frac{(B_2)_{i+\mfourone,j+\halfone}^{J,\mu}-(B_2)_{i+\mfourone,j-\halfone}^{J,\mu}}{\dy}\right).
	\end{equation}
	Thanks to \Cref{Lemma:SourT}, we have 
	\begin{equation*}
		\begin{aligned}
			-&\jump{B_{1,h}^J(x_i,y_{j+\mfourone}^{\mu})} \bm{S}\left(\average{\bm{U}_h^J(x_i,y_{j+\mfourone}^{\mu})}  \right) \cdot \bm{n}^* \\ 
			&~~~~~\geq \jump{B_{1,h}^J(x_i,y_{j+\mfourone}^{\mu})} (\bm{v}^* \cdot \bm{B}^*) - \frac{\Big|\jump{B_{1,h}^J(x_i,y_{j+\mfourone}^{\mu})}\Big|}{\sqrt{\average{\rho h}(x_i,y_{j+\mfourone}^{\mu})}} \left(\average{\bm{U}_h^J(x_i,y_{j+\mfourone}^{\mu})} \cdot \bm{n}^* + p^*_m\right) \\
			&~~~~~\geq \jump{B_{1,h}^J(x_i,y_{j+\mfourone}^{\mu})} (\bm{v}^* \cdot \bm{B}^*) - 2\beta_1 \left(\average{\bm{U}_h^J(x_i,y_{j+\mfourone}^{\mu})} \cdot \bm{n}^* + p^*_m\right). \\
		\end{aligned}
	\end{equation*}
	Similarly, one can derive 
	\begin{align*}
		-\jump{B_{1,h}^J(x_{i+\mfourone}^{\mu},y_j)} \bm{S}\left(\average{\bm{U}_h^J(x_{i+\mfourone}^{\mu},y_j)}  \right) \cdot \bm{n}^* 
		\geq \jump{B_{2,h}^J(x_{i+\mfourone}^{\mu},y_j)} (\bm{v}^* \cdot \bm{B}^*) - 2\beta_2 \left(\average{\bm{U}_h^J(x_{i+\mfourone}^{\mu},y_j)} \cdot \bm{n}^* + p^*_m\right). 
	\end{align*}
	Combining these inequalities with the equation \eqref{CDG2D:SourI}, we get 
	\begin{align}\label{pf:2DSour}
		\bm{S}_{ij}^J \cdot \bm{n}^* 
		&\geq \frac{1}{\dx} \sum_{m=\pm1}\sum_{\mu=1}^{Q} \frac{\omega_{\mu}}{2} \left( \jump{B_{1,h}^J(x_i,y_{j+\mfourone}^{\mu})} (\bm{v}^* \cdot \bm{B}^*) - 2\beta_1 \left(\average{\bm{U}_h^J(x_i,y_{j+\mfourone}^{\mu})} \cdot \bm{n}^* + p^*_m\right) \right)  \nonumber\\
		& +\frac{1}{\dy} \sum_{m=\pm1}\sum_{\mu=1}^{Q} \frac{\omega_{\mu}}{2} \left( \jump{B_{2,h}^J(x_{i+\mfourone}^{\mu},y_j)} (\bm{v}^* \cdot \bm{B}^*) - 2\beta_2 \left(\average{\bm{U}_h^J(x_{i+\mfourone}^{\mu},y_j)} \cdot \bm{n}^* + p^*_m\right) \right)  \nonumber\\
		& \geq (\bm{v}^* \cdot \bm{B}^*)  {\rm div}_{ij} \jump{\bm{B}_h^J} \nonumber\\
		& -\frac{ 2 \alpha_1}{\dx} \sum_{m=\pm1}\sum_{\mu=1}^{Q} \frac{\omega_{\mu}}{2}   \left(\average{\bm{U}_h^J(x_i,y_{j+\mfourone}^{\mu})} \cdot \bm{n}^* + p^*_m\right)
		- \frac{ 2 \alpha_2}{\dy} \sum_{m=\pm1}\sum_{\mu=1}^{Q} \frac{\omega_{\mu}}{2}  \left(\average{\bm{U}_h^J(x_{i+\mfourone}^{\mu},y_j)} \cdot \bm{n}^* + p^*_m\right)  \nonumber\\
		& = (\bm{v}^* \cdot \bm{B}^*)  {\rm div}_{ij} \jump{\bm{B}_h^J}  - {\lambda} \left( {\bm \Pi}_{ij}^{J,2} \cdot \bm{n}^* + 2p^*_m\right), 
	\end{align} 
	where we have used the fact that 
	\begin{equation*}
		\begin{aligned}
			\frac{\lambda}{2} {\bm \Pi}_{ij}^{J,2} 
			& = \frac{\alpha_1}{\dx} \sum_{m=\pm1}\sum_{\mu=1}^{Q} \frac{\omega_{\mu}}{2} \average{\bm{U}_h^J(x_i,y_{j+\mfourone}^{\mu})}  
			+ \frac{\alpha_2}{\dy} \sum_{m=\pm1} \sum_{\mu=1}^{Q} \frac{\omega_{\mu}}{2} \average{\bm{U}_h^J(x_{i+\mfourone}^{\mu},y_j)} \\
			& \geq \frac{\beta_1}{\dx} \sum_{m=\pm1}\sum_{\mu=1}^{Q} \frac{\omega_{\mu}}{2} \average{\bm{U}_h^J(x_i,y_{j+\mfourone}^{\mu})}  
			+ \frac{\beta_2}{\dy} \sum_{m=\pm1} \sum_{\mu=1}^{Q} \frac{\omega_{\mu}}{2} \average{\bm{U}_h^J(x_{i+\mfourone}^{\mu},y_j)},
		\end{aligned}
	\end{equation*}
 and the notation 
	\begin{equation}\label{eq:2DDivB_J2}
		{\rm div}_{ij} \jump{\bm{B}_h^J} := \sum_{m=\pm1} \sum_{\mu=1}^{Q} \frac{\omega_{\mu}}{2} \left(\frac{\jump{B_{1,h}^J(x_i,y_{j+\mfourone}^{\mu})}}{\dx} + \frac{\jump{B_{2,h}^J(x_{i+\mfourone}^{\mu},y_j)}}{\dy}\right).
	\end{equation}
	Combining the inequalities \eqref{pf:2DPIFdotnstar} and \eqref{pf:2DSour} with \eqref{WKL33}, we obtain 
		\begin{align} \nonumber 
			\overline{\bm{U}}_{ij}^{I,\dt} \cdot \bm{n}^* + p^*_m 
			& >  \left(\frac{\theta \bar{\omega}}{2}- {\lambda}\Delta t\right) \Big(({\bm \Pi}_{ij}^{J,2}\cdot \bm{n}^* + 2p^*_m) + ({\bm \Pi}_{ij}^{J,3} \cdot \bm{n}^* + 2p^*_m) \Big) \\ \nonumber
			& ~~~~~ + \dt (\bm{v}^* \cdot \bm{B}^*) \left( {\rm div}_{ij} \jump{\bm{B}_h^J} - {\rm div}_{ij}\bm{B}_h^J \right) \\ \label{pf:Udotn*}
			& \ge \dt (\bm{v}^* \cdot \bm{B}^*) \left({\rm div}_{ij} \jump{\bm{B}_h^J} - {\rm div}_{ij}\bm{B}_h^J \right),  
		\end{align}
	where we have used \eqref{pf:CDG2Dcon1} and the CFL constraint \eqref{CDG2D:CFLCond}. 
	With \eqref{eq:2DDivB_J1} and \eqref{eq:2DDivB_J2}, we derive 
	\begin{align*}
		{\rm div}_{ij} \jump{\bm{B}_h^J} - {\rm div}_{ij}\bm{B}_h^J
		&=-\sum_{\mu=1}^{Q} \frac{\omega_{\mu}}{2} \left(\frac{(B_1)^{J,-,\mu}_{i,j-\fourone}-(B_1)_{i-\halfone,j-\fourone}^{J,\mu}}{\dx} +  \frac{(B_2)^{J,\mu,-}_{i-\fourone,j}-(B_2)_{i-\fourone,j-\halfone}^{J,\mu}}{\dy}\right) \\
		&~~~~~-\sum_{\mu=1}^{Q} \frac{\omega_{\mu}}{2} \left(\frac{(B_1)_{i+\halfone,j-\fourone}^{J,\mu}-(B_1)^{J,+,\mu}_{i,j-\fourone}}{\dx} +  \frac{(B_2)^{J,\mu,-}_{i+\fourone,j}-(B_2)_{i+\fourone,j-\halfone}^{J,\mu}}{\dy}\right) \\
		&~~~~~-\sum_{\mu=1}^{Q} \frac{\omega_{\mu}}{2} \left(\frac{(B_1)^{J,-,\mu}_{i,j+\fourone}-(B_1)_{i-\halfone,j+\fourone}^{J,\mu}}{\dx} +  \frac{(B_2)_{i-\fourone,j+\halfone}^{J,\mu}-(B_2)^{J,\mu,+}_{i-\fourone,j}}{\dy}\right) \\
		&~~~~~-\sum_{\mu=1}^{Q} \frac{\omega_{\mu}}{2} \left(\frac{(B_1)_{i+\halfone,j+\fourone}^{J,\mu}-(B_1)^{J,+,\mu}_{i,j+\fourone}}{\dx} +  \frac{(B_2)_{i+\fourone,j+\halfone}^{J,\mu}-(B_2)^{J,\mu,+}_{i+\fourone,j}}{\dy}\right) \\
		&=-\frac{1}{\dx \dy} \sum_{m=\pm1,\sigma=\pm1}\int_{\partial \Omega_{i+\mfourone,j+\sfourone} } \bm{B}_h^J \cdot \bm{n}_{\partial \Omega_{i+\mfourone,j+\sfourone} } \dd s \\
		&=-\frac{1}{\dx \dy} \sum_{m=\pm1,\sigma=\pm1}\iint_{ \Omega_{i+\mfourone,j+\sfourone} } \nabla \cdot \bm{B}_h^J \dd x\dd y,
	\end{align*}
	where the first step is obtained by reformulation, the second step follows from the exactness of $Q$-point Gauss quadrature rule for any bivariate polynomial of degree $\le k$, and the third step is derived by applying the divergence theorem within the four quarter cells  $\Omega_{i+\mfourone,j+\sfourone}:=[x_{i+\mfourone}-\frac{\dx}{4}, x_{i+\mfourone}+\frac{\dx}{4}] \times [y_{j+\sfourone}-\frac{\dy}{4}, y_{j+\sfourone}+\frac{\dy}{4}]$ for $m=\pm1$ and $\sigma=\pm1$.
	Since $\bm{U}_h^J \in \mathbb{V}_{h}^{J,k}$, we have $\nabla \cdot \bm{B}_h^J = 0 $ within each of these four quarter cells. This implies the identity ${\rm div}_{ij} \jump{\bm{B}_h^J} - {\rm div}_{ij}\bm{B}_h^J =0$. It follows from \eqref{pf:Udotn*} that 
	$$
	\overline{\bm{U}}_{ij}^{I,\dt} \cdot \bm{n}^* + p^*_m > 0, 
	$$
	which together with \eqref{pf:Udotn1} yields $\overline{\bm{U}}_{ij}^{I,\dt} \in \mathcal{G}_* = \mathcal{G}$, according to the GQL representation in the \Cref{Lemma:Gstar}. 
	By similar arguments, we can show $\overline{\bm{U}}_{i+\halfone,j+\halfone}^{J,\dt} \in \mathcal{G}$. The proof is completed. 
\end{proof}

\begin{remark}\label{rmk:BP2D}
	Our 2D CDG schemes effectively combine two important techniques to control divergence errors: the locally DF CDG finite elements and the suitable discretization of the symmetrization source terms. The former ensures that the numerical magnetic fields $\bm{B}_h^I$ and $\bm{B}_h^J$ maintain locally zero divergence within each cell. The latter helps reduce the divergence errors across cell interfaces. As demonstrated in the proof of \Cref{thm:main2D}, the BP property of the CDG schemes highly depends on these two techniques, which together effectively eliminate the influence of divergence errors on the BP property.
\end{remark}

We would like to mention that the GQL approach is essential for establishing and rigorously proving the BP property. Without it, the BP analysis would become exceedingly challenging, if not impossible. The challenges primarily arise from the high nonlinearity in the functions $p(\bm{U})$, $\bm{v}(\bm{U})$, $q(\bm{U})$, and $\Phi(\bm{U})$ associated with the physical constraints. Due to the locally DF constraint, the states at the Gauss quadrature nodes on the cell interfaces are strongly coupled. Consequently, some standard BP analysis techniques, which typically rely on decomposing high-order multidimensional schemes into a convex combination of formally 1D BP schemes, become invalid in the RMHD case. Therefore, the BP analysis is nontrivial, and the GQL approach is key to overcoming these challenges.

\subsection{2D BP Limiter} 

	\Cref{thm:main2D} establishes a sufficient condition \eqref{CDG2D:QuaCond} for our 2D locally DF CDG schemes to preserve the BP property of the cell averages, specifically when using the forward Euler method for time discretization. Moreover, this BP property is also retained in the fully discrete 2D locally DF CDG schemes that incorporate high-order SSP time discretization, given that an SSP method can be regarded as a convex combination of the forward Euler method.

As in the 1D case, the BP condition \eqref{CDG2D:QuaCond} may not be inherently fulfilled by the 2D CDG solutions. To ensure compliance with this BP condition \eqref{CDG2D:QuaCond}, a 2D local scaling BP limiter is required. In addition, similar to the 1D discussions in \Cref{rem:1Dmorepoints}, we also need to enforce  $\bm{U}_h^{J}(x_{i\pm\fourone}^{\mu},y_{j\pm\fourone}^{\tilde{\mu}}) \in {\mathcal G}$ and $\bm{U}_h^{I}(x_{i+\halfone\pm\fourone}^\mu,y_{j+\halfone\pm\fourone}^{\tilde{\mu}}) \in {\mathcal G}$  for the robust recovery of primitive variables at these tensor-product Gauss points. This necessity arises because the flux must be evaluated at these points, as indicated in \eqref{CDG2D:DisRI}--\eqref{CDG2D:DisBJ}. In summary, the 2D BP limiter is employed to enforce the admissibility of $\bm{U}_h^{I}(x,y)$ and $\bm{U}_h^{J}(x,y)$ at the following points:  
	\begin{equation}\label{eq:Sij}
		\mathcal{S}_{ij} = \mathcal{Q}_{ij}  \cup (\mathcal{Q}_i^x \otimes \mathcal{Q}_j^y),
	\end{equation}
	where
	$$\mathcal{Q}_i^x=\{{x}_{i-\fourone}^\mu\}_{\mu=1}^Q \cup \{{x}_{i+\fourone}^\mu\}_{\mu=1}^Q, \quad \mathcal{Q}_j^y=\{{y}_{j-\fourone}^\mu\}_{\mu=1}^Q \cup \{{y}_{j+\fourone}^\mu\}_{\mu=1}^Q.$$
The implementation of the 2D BP limiter is very similar to that of the 1D case, only replacing the 1D point set $\mathcal{S}_j$ with the 2D point set $\mathcal{S}_{ij}$. Hence, the details can be omitted here.

\begin{remark}
	It is important to note that the 2D BP limiter maintains the locally DF property of the CDG solutions. After applying the BP limiter to the CDG solutions at each Runge--Kutta stage \eqref{RKTime}, we obtain the fully discrete CDG schemes that are high-order accurate, locally DF, and provably BP under the CFL condition \eqref{CDG2D:CFLCond}.
\end{remark}

\subsection{BP CFL Conditions with Different CADs}

Note that the BP condition \eqref{CDG2D:QuaCond} and the associated BP limiter depend on the nodes of the 2D CADs \eqref{eq:CADI} and \eqref{eq:CADJ} adopted in the BP analysis. Furthermore, the theoretically estimated BP CFL condition \eqref{CDG2D:CFLCond} also depends on the weight $\bar \omega$ in the 2D CADs. If one chooses the 2D Zhang--Shu CAD \cite{zhang2010} in \Cref{ex:zs}, then the following result can be obtained. 

\begin{corollary}\label{Coro:ZS}
	If the condition \eqref{CDG2D:QuaCond} is satisfied with ${\mathcal Q}_{ij}$ defined by the nodes of the 2D Zhang--Shu CAD, then the BP property in \Cref{thm:main2D} holds under the CFL condition 
		\begin{equation}\label{CDG2D:CFLCondzs}
		0<\dt\left(\frac{\alpha_1}{\dx} + \frac{\alpha_2}{\dy}\right) \le \frac{\theta \widehat{\omega}_1}{2} = \frac{\theta }{2 L (L-1)}, 
	\end{equation}
	where $\alpha_1 = \max\{1,\beta_1\}$, $\alpha_2 = \max\{1,\beta_2\}$, $\theta = {\dt}/{\tau_{\rm max}} \in (0,1]$, and $L=\lceil\frac{k+3}{2} \rceil$. Specifically, $\widehat{\omega}_1=\frac1{L (L-1)}=\frac16$ for $k=2$ and $3$; $\widehat{\omega}_1=\frac1{12}$ for $k=4$ and $5$; $\widehat{\omega}_1=\frac1{20}$ for $k=6$ and $7$. 
\end{corollary}

\begin{proof}
This is a direct consequence of \Cref{thm:main2D}. 
\end{proof}

If we choose the 2D Cui--Ding--Wu CAD in \Cref{ex:cdw}, then the following result can be obtained.

\begin{corollary}\label{Coro:cdw}
	If the condition \eqref{CDG2D:QuaCond} is satisfied with ${\mathcal Q}_{ij}$ defined by the nodes of the 2D Cui--Ding--Wu CAD, then the BP property in \Cref{thm:main2D} holds under the CFL condition 
	\begin{equation}\label{CDG2D:CFLCondcdw}
		0<\dt\left(\frac{\alpha_1}{\dx} + \frac{\alpha_2}{\dy}\right) \le \frac{\theta \bar \omega_\star }{2}, 
	\end{equation}
	where $\theta = {\dt}/{\tau_{\rm max}} \in (0,1]$. The weight $\bar \omega_\star=\bar \omega_\star( \delta, \mathbb P^k )$ depends on $\delta := \frac{\lambda_1-\lambda_2}{\lambda} = \frac{\alpha_1 \dy - \alpha_2 \dx}{ \alpha_1 \dy + \alpha_2 \dx } $ and the space $\mathbb P^k$. 
	For $k\in \{2,\dots,7\}$, it is given by  
		\begin{equation*}
	\bar \omega_\star( \delta, \mathbb P^k )=
		\begin{cases} 
			\frac12, &~ k=1;\\
			\displaystyle
			\frac{1}{4+2|\delta|}, &~ k=2,3;\\ \displaystyle
			\left[\frac{14}{3}+\frac{2}{3}\sqrt{78\,\delta^2+46} \cos \left( \frac{1}{3} \arccos\frac{1476\,\delta^2-244}{(78\,\delta^2+46)^{\frac{3}{2}}} \right) \right]^{-1}, &~ k=4,5;\\ \displaystyle
			\left[\frac{20}{3}+2\,|\delta|+\frac{2}{3}\sqrt{126\,\delta^2+96|\delta|+94}\cos\left(\frac{1}{3}\arccos\frac{864\,|\delta|^3+2916\,\delta^2+288\,|\delta|-532}{(126\,\delta^2+96|\delta|+94)^{\frac{3}{2}}}\right)\right]^{-1}, &~ k=6, 7. 
		\end{cases}
	\end{equation*}
	Particularly, in the typical case $\alpha_1 \dy = \alpha_2 \dx$, we have $\delta=0$ and 
	\begin{equation*}
		\bar \omega_\star = \frac 14 ~\mbox{  for } k=2,3; \quad ~ \bar \omega_\star = 2-\frac{\sqrt{14}}{2} \approx 0.1292 ~\mbox{  for } k=4,5; \quad ~ \bar \omega_\star = 1-\frac{\sqrt{30}}{6} \approx 0.08713 ~\mbox{  for } k=6,7. 
	\end{equation*}
	The values of $\bar \omega_\star$ for higher degree $k\ge 8$ can be effectively computed using Algorithm 5.11 of \cite{CuiDingWu2022SINUM}. 
\end{corollary}

\begin{proof}
	This corollary follows from \Cref{thm:main2D} and \cite[Theorems 5.4, 5.6, and 5.8]{CuiDingWu2022SINUM}. 
\end{proof}

One can observe that the 2D Cui--Ding--Wu CAD provides a milder BP CFL constraint and involves fewer internal nodes compared to the 2D Zhang--Shu CAD. Therefore, in our numerical experiments, we opt for the Cui--Ding--Wu CAD for our 2D BP CDG schemes.

\begin{remark}
	As shown in \Cref{Coro:cdw}, the weight $\bar \omega_\star$ has simple expressions when $\alpha_1 \dy = \alpha_2 \dx$, which is a typical case always encountered in our computations. Recall that
	$\alpha_\ell = \max\{1,\beta_\ell\}$, where $\{\beta_\ell\}_{\ell=1}^2$ are defined by \eqref{CDG2D:beta}
	and are proportional to the jump in the normal component of the magnetic field at the cell interface. For the exact solutions, this jump is zero. Consequently, $\beta_\ell$ is small and at the level of truncation error in smooth regions.
	For discontinuous problems, we numerically observe that $\beta_\ell$ is much smaller than $1$, which equals the speed of light. Thus, in most cases (including all the cases we tested), $\beta_\ell$ remains smaller than $1$, and therefore $\alpha_\ell = 1$.
	Hence, if we take $\dx = \dy$, then $\alpha_1 \dy = \alpha_2 \dx$, $\lambda_1=\lambda_2$, and $\delta = 0$.
\end{remark}

\subsection{Why We Need the Symmetrization Source Terms?}

If we remove the discretized part of the symmetrization source terms from our 2D CDG schemes, we then obtain the 2D CDG method for the conservative RMHD system \eqref{eq:RMHD} without the symmetrization source terms. 
The evolution equations \eqref{CDG2D:ave} for the cell averages in the 2D CDG schemes  without the discretized source terms are given by 
\begin{equation}\label{CDG2D:aveNSour}
	\frac{\dd \overline{\bm{U}}_{ij}^I}{\dd t} =\mathcal{H}_{ij}^I(\bm{U}_h^I, \bm{U}_h^J), 
	\qquad
	\frac{\dd \overline{\bm{U}}_{i+\halfone,j+\halfone}^J}{\dd t} =\mathcal{H}_{i+\halfone,j+\halfone}^J(\bm{U}_h^J,\bm{U}_h^I),
\end{equation}	
where $\mathcal{H}_{ij}^I(\bm{U}_h^I, \bm{U}_h^J)$ and $\mathcal{H}_{i+\halfone,j+\halfone}^J
(\bm{U}_h^J,\bm{U}_h^I)$ are defined in \eqref{CDG2D:aveHI} and \eqref{CDG2D:aveHJ}, respectively. 
We have the following theorem on the BP property of the 2D CDG schemes without the discretized source terms.

\begin{theorem}\label{thm:IRMHD}
	Assume that $\overline{\bm{U}}_{ij}^I \in \mathcal{G}$ and $\overline{\bm{U}}_{i+\halfone,j+\halfone}^J \in \mathcal{G}$.
	If the numerical solutions $\bm{U}_h^I(x,y)$ and $\bm{U}_h^J(x,y)$ satisfy the condition \eqref{CDG2D:QuaCond}, 
	then under the CFL condition
	\begin{equation}\label{CDG2D:CFLCondNSour}
		0<\dt\left(\frac{1}{\dx} + \frac{1}{\dy}\right) \leq \frac{\theta \bar{\omega}}{2}, \quad \theta = \frac{\dt}{\tau_{\rm max}} \in (0,1], 
	\end{equation} 
	the updated cell averages
	\begin{equation*}\label{CDG2D:UpdateNSource}
		\overline{\bm{U}}_{ij}^{I,\mathcal{H},\dt} := 
		\overline{\bm{U}}_{ij}^I + \dt \mathcal{H}_{ij}^I(\bm{U}_h^I,\bm{U}_h^J) ,
		\quad
		\overline{\bm{U}}_{i+\halfone,j+\halfone}^{J,\mathcal{H},\dt} := 
		\overline{\bm{U}}_{i+\halfone,j+\halfone}^J + \dt \mathcal{H}_{i+\halfone,j+\halfone}^J(\bm{U}_h^J,\bm{U}_h^I)
	\end{equation*}
	satisfy for all $i$ and $j$ that	
	\begin{equation}\label{thm:H1}
		\overline{\bm{U}}_{ij}^{I,\mathcal{H},\dt} \cdot \bm{n}_1 >0, \quad \overline{\bm{U}}_{i+\halfone,j+\halfone}^{J,\mathcal{H},\dt} \cdot \bm{n}_1 >0, 
	\end{equation}
	\begin{equation}\label{thm:H2}
		\overline{\bm{U}}_{ij}^{I,\mathcal{H},\dt} \cdot \bm{n}^* + p^*_m > - \dt (\bm{v}^*\cdot \bm{B}^*) ({\rm div}_{ij}\bm{B}_h^{J}), 
	\end{equation}
	\begin{equation}\label{thm:H3}
		\overline{\bm{U}}_{i+\halfone,j+\halfone}^{J,\mathcal{H},\dt} \cdot \bm{n}^* + p^*_m > - \dt (\bm{v}^*\cdot \bm{B}^*) ({\rm div}_{i+\halfone,j+\halfone}\bm{B}_h^I) 
	\end{equation}
	for any free auxiliary variables $\bm{v}^*\in \mathbb{B}_1(0)$ and $\bm{B}^*\in \mathbb{R}^3$.
	
	Furthermore, 	the estimates \eqref{thm:H1}-\eqref{thm:H3} imply $\overline{\bm{U}}_{ij}^{I,\mathcal{H},\dt}\in \mathcal{G}$ and $ \overline{\bm{U}}_{i+\halfone,j+\halfone}^{J,\mathcal{H},\dt} \in \mathcal{G}$, 	
	if the numerical solutions $\bm{U}_h^{J}(x,y)$ and $\bm{U}_h^{I}(x,y)$ satisfy the following discrete DF condition
	\begin{equation}\label{DDF}
		{\rm div}_{ij}\bm{B}_h^{J} = 0, \quad {\rm div}_{i+\halfone,j+\halfone}\bm{B}_h^{I}=0,
	\end{equation}
	where ${\rm div}_{ij}\bm{B}_h^{J}$ is defined by \eqref{eq:2DDivB_J1} and ${\rm div}_{i+\halfone,j+\halfone}\bm{B}_h^{I}$ is given by
	\begin{equation}\label{eq:2DDivB_I1}
		{\rm div}_{i+\halfone,j+\halfone}\bm{B}_h^I := \sum_{m=\pm1} \sum_{\mu=1}^{Q} \frac{\omega_{\mu}}{2} \left(\frac{(B_1)_{i+1,j+\halfone+\mfourone}^{I,\mu}-(B_1)_{i,j+\halfone+\mfourone}^{I,\mu}}{\dx} + \frac{(B_2)_{i+\halfone+\mfourone,j+1}^{I,\mu}-(B_2)_{i+\halfone+\mfourone,j}^{I,\mu}}{\dy}\right).
	\end{equation}
\end{theorem}

\begin{proof}
	Consider the CAD \eqref{eq:CADI} with $\lambda_1=\frac{1}{\dx}$, $\lambda_2 = \frac{1}{\dy}$, and $\lambda = \lambda_1 + \lambda_2 =  \frac{1}{\dx} + \frac{1}{\dy}$. 
	Similar to the equation \eqref{pf:CDG2DUbar} in the proof of \Cref{thm:main2D}, one can derive 
	\begin{equation}
		\overline{\bm{U}}_{ij}^{I,\mathcal{H},\dt} 
		=(1-\theta)\overline{\bm{U}}_{ij}^{I} + \theta{\bm \Pi}_{ij}^{J,1} + \frac{\theta \bar{\omega}}{2} ({\bm \Pi}_{ij}^{J,2} + {\bm \Pi}_{ij}^{J,3}) + \dt {\bm \Pi}_{F}^{J}, 
	\end{equation}
	where ${\bm \Pi}_{F}^{J}$ is defined in \eqref{pf:2DPIF}. 
	Following the derivations of the equations \eqref{pf:2DPIFdotn1} and \eqref{pf:2DPIFdotnstar} in the proof of \Cref{thm:main2D}, we obtain ${\bm \Pi}_{F}^{J} \cdot \bm{n}_1 \geq {\lambda}{\bm \Pi}_{ij}^{J,3}  \cdot \bm{n}_1$ and ${\bm \Pi}_{F}^{J} \cdot \bm{n}^* \geq - (\bm{v}^*\cdot \bm{B}^*) {\rm div}_{ij}\bm{B}_h^J - {\lambda} \left({\bm \Pi}_{ij}^{J,3} \cdot \bm{n}^* + 2p^*_m\right) $.	
	Note that \eqref{pf:CDG2Dcon1} remains valid due to the condition \eqref{CDG2D:QuaCond}. 
	Combining \eqref{pf:CDG2Dcon1} with the CFL constraint \eqref{CDG2D:CFLCondNSour}, we have   $\overline{\bm{U}}_{ij}^{I,\mathcal{H},\dt} \cdot \bm{n}_1 >0 $ similar to \eqref{pf:Udotn1}
	and  
	\begin{equation*}
		\begin{aligned}
			\overline{\bm{U}}_{ij}^{I,\mathcal{H},\dt} \cdot \bm{n}^* + p^*_m &\geq (1-\theta)(\overline{\bm{U}}_{ij}^{I} \cdot \bm{n}^* + p^*_m) + \theta ({\bm \Pi}_{ij}^{J,1} \cdot \bm{n}^* + (1-2\bar{\omega}) p^*_m) \\
			&~~~~~+ \frac{\theta \bar{\omega}}{2}({\bm \Pi}_{ij}^{J,2}\cdot \bm{n}^* + 2p^*_m) +  \left(\frac{\theta \bar{\omega}}{2}- {\lambda}\Delta t \right) \left( {\bm \Pi}_{ij}^{J,3} \cdot \bm{n}^* + 2p^*_m \right) \\
			& ~~~~~ - \dt (\bm{v}^* \cdot \bm{B}^*) {\rm div}_{ij}(\bm{B}_h^J)  \\
			&> -\dt (\bm{v}^* \cdot \bm{B}^*) {\rm div}_{ij}(\bm{B}_h^J)
		\end{aligned}
	\end{equation*}
	for any free auxiliary variables $\bm{v}^*\in \mathbb{B}_1(0)$ and $\bm{B}^*\in \mathbb{R}^3$.
	Hence, if $\bm{U}_h^J(x,y)$ satisfies the discrete DF condition ${\rm div}_{ij}\bm{B}_h^{J} = 0$, then 
	\begin{equation*}
		\overline{\bm{U}}_{ij}^{I,\mathcal{H},\dt} \cdot \bm{n}^* + p^*_m  >0  ~~~~ \forall \bm{v}^*\in \mathbb{B}_1(0), ~ \bm{B}^*\in \mathbb{R}^3,
	\end{equation*}
	which, together with $\overline{\bm{U}}_{ij}^{I,\mathcal{H},\dt} \cdot \bm{n}_1 >0$,  implies $\overline{\bm{U}}_{ij}^{I,\mathcal{H},\dt} \in \mathcal{G}_* = \mathcal{G}$, thanks to the GQL representation in \Cref{Lemma:Gstar}. 
	Similarly, we can deduce that  $\overline{\bm{U}}_{i+\halfone,j+\halfone}^{J,\mathcal{H},\dt} \cdot \bm{n}_1 >0$ and the estimate \eqref{thm:H3}, which give rise to $\overline{\bm{U}}_{i+\halfone,j+\halfone}^{J,\mathcal{H},\dt} \in \mathcal{G}_* = \mathcal{G}$ under the discrete DF condition ${\rm div}_{i+\halfone,j+\halfone}\bm{B}_h^{I}=0$. The proof is completed.
\end{proof}

As shown in the proof of \Cref{thm:main2D}, the locally DF property is crucial in achieving the BP property for the CDG schemes for the modified RMHD equations \eqref{eq:ModRMHD} with symmetrization source terms. As \Cref{thm:IRMHD} reveals, if we remove these source terms, the BP property of the resulting CDG schemes hinges on the discrete DF condition \eqref{DDF}. However, the locally DF property alone is insufficient to ensure this DF condition \eqref{DDF}; it is met if the magnetic fields $\bm{B}^I_h$ and $\bm{B}^J_h$ are globally DF, including being locally DF within each cell and maintaining the continuity of the normal magnetic component across cell interfaces. Unfortunately, the local scaling nature of the BP limiter renders it incompatible with the globally DF property. Therefore, simultaneously enforcing both conditions \eqref{DDF} and \eqref{eq:Sij} without compromising accuracy and conservation presents a significant challenge and remains unresolved. The inclusion of symmetrization source terms relaxes the (globally coupled) DF requirement \eqref{DDF} to a locally DF condition compatible with the BP limiter, thereby effectively avoiding this issue.

\section{Numerical Experiments}\label{Numericalexperiments}
In this section, we conduct several benchmark and challenging tests to validate the accuracy and robustness of the proposed BP DF CDG schemes.   Our test cases include two smooth problems to verify the accuracy of our CDG method. Moreover, we also investigate several non-smooth problems, including three 1D Riemann problems, a 2D Orszag–Tang problem, a 2D rotor problem, a 2D shock-cloud interaction problem, three 2D blast problems, and two 2D astrophysical jets. These problems are known to contain strong discontinuities and suitable for verifying the capability of our CDG schemes in accurately resolving discontinuous solutions and capturing complex flow structures. Additionally, to suppress potential numerical oscillations, we implement the (locally DF) WENO limiter \cite{ZhaoTang2017} within some adaptively detected troubled cells, right before the BP limiter, during the simulations of these non-smooth problems. 
We focus on the $\mathbb{P}^2$-based CDG method with the third-order SSP Runge–Kutta time discretization \eqref{RKTime}. In our 2D BP CDG method, we employ the Cui–Ding–Wu CAD as described in \Cref{ex:cdw}. 
This approach allows for a milder BP CFL constraint in theory and requires fewer internal nodes.  
It is worth noting that the theoretically estimated BP CFL condition is sufficient but not always necessary. In practice, we find that our BP CDG schemes often work robustly with a larger CFL number. 
In all following tests, we take the time step-size as $\dt=C_{\rm cfl} \dx$ for the 1D problems, and $\dt=\frac{C_{\rm cfl}}{ \frac{1}{\Delta x} + \frac{1}{\Delta y}}$ for the 2D problems, with the CFL number $C_{\rm cfl}=0.25$.


To demonstrate the effectiveness and diverse applicability of our CDG method for RMHD problems with various EOSs, our numerical examples will cover four different EOSs:
\begin{equation}\label{EOS:IDEOS}
\mbox{Ideal EOS}:\qquad	h=1+\frac{\Gamma p}{(\Gamma-1) \rho},
\end{equation}
\begin{equation}\label{EOS:IPEOS}
\mbox{IP-EOS \cite{Sokolov2001}}:\qquad	h = \frac{2p}{\rho} + \sqrt{1+4 \left(\frac{p}{\rho}\right)^2},
\end{equation}
\begin{equation}\label{EOS:TMEOS}
\mbox{TM-EOS \cite{Mathews1971,Mignone2005}}:\qquad	h = \frac{5p}{2\rho} + \sqrt{1+\frac{9}{4} \left(\frac{p}{\rho}\right)^2},
\end{equation}
\begin{equation}\label{EOS:RCEOS}
\mbox{RC-EOS \cite{Ryu2006}}:\qquad	h = \frac{12p^2 + 8p \rho + 2 \rho^2}{3p \rho + 2\rho^2},
\end{equation}
where $\Gamma \in (1,2]$ in the ideal EOS \eqref{EOS:IDEOS} denotes the constant adiabatic index. All these four EOSs adhere to the conditions \eqref{EOS:heq1} and \eqref{EOS:heq2}, ensuring the validity of our theoretical analyses and the BP limiters for them. Although the ideal EOS \eqref{EOS:IDEOS} has been widely used in simulating many RMHD problems, it is borrowed from the non-relativistic case and provides a poor approximation for most relativistic flows, as it is inconsistent with relativistic kinetic theory \cite{Ryu2006}. In fact, the ideal EOS \eqref{EOS:IDEOS} is valid only for either strictly sub-relativistic or ultra-relativistic gases. Recognizing the pivotal role of the accurate EOS in relativistic hydrodynamics, several other EOSs, including \eqref{EOS:IPEOS}–\eqref{EOS:RCEOS}, have been proposed in the literature; see, e.g., \cite{Mathews1971, Mignone2005, Sokolov2001, Ryu2006}.

\begin{expl}[Smooth Problems]\label{Ex:Smooth} \rm 
	
	A 1D smooth problem and a 2D smooth problem are tested to verify the accuracy of our CDG schemes. Both problems describe the periodic propagation of Alfv{\'e}n waves with a large speed of $0.99c$ and low pressure, employing the TM-EOS \eqref{EOS:TMEOS}.
	
	For the 1D problem, we adopt the same setup as in \cite{WuTangM3AS}. The spatial domain is specified as $[0,1]$, and the exact solution is 
	\begin{align*}
		&\rho(x,t) = 1, \quad 
		p(x,t) = 0.01, \quad
		v_1(x,t) = 0, \quad 
		v_2(x,t)=0.99 \sin(2\pi (x+t/\kappa)), \\
		&v_3(x,t)=0.99 \cos(2\pi (x+t/\kappa)), \quad
		B_1(x,t) = 1, \quad
		B_2(x,t) = \kappa v_2(x,t), \quad
		B_3(x,t) = \kappa v_3(x,t), 
	\end{align*} 
	where $\kappa = \sqrt{1+\rho h W^2}$.
	
	For the 2D problem, the wave propagates periodically at an angle $\alpha = \pi /4$ with respect to the $x$-axis within the domain $[0,\sqrt{2}]^2$. This problem is analogous to the one tested in \cite{WuShu2020NumMath}, but with lower pressure. The exact solution is
	\begin{align*}
		&\rho(x,y,t) = 1, \quad 
		p(x,y,t) = 0.01, \quad
		v_1(x,y,t) = -0.99\sin(2\pi(\zeta+t/\kappa)) \sin\alpha, \\
		&v_2(x,y,t)=0.99\sin(2\pi(\zeta+t/\kappa)) \cos\alpha, \quad
		v_3(x,y,t)=0.99 \cos(2\pi (\zeta+t/\kappa)), \\
		&B_1(x,y,t) = \cos \alpha + \kappa v_1(x,y,t), \quad
		B_2(x,y,t) = \sin \alpha + \kappa v_2(x,y,t), \quad
		B_3(x,y,t) = \kappa v_3(x,y,t)
	\end{align*} 
	with $\zeta=x \cos \alpha + y \sin \alpha$.

	In our simulations, the computational domain is partitioned into uniform meshes with periodic boundary conditions. \Cref{tab:Ex-Smooth1D} lists the errors in the velocity component $v_2$ at $t=1$ for the 1D problem and the corresponding convergence orders obtained using the proposed $\mathbb P^2$-based BP CDG method at different grid resolutions. \Cref{tab:Ex-Smooth2D} presents the errors in the magnetic component $B_2$ at $t=1$ for the 2D problem and the corresponding convergence orders obtained using the proposed $\mathbb P^2$-based BP locally DF CDG method. It can be seen that the expected third-order convergence is achieved in both the 1D and 2D cases, indicating that the BP limiter and the discretized source terms do not compromise the desired accuracy of the CDG schemes.
	
	\begin{table}[!htb] 
		\centering
		\caption{Errors in $v_2$ at $t=1$ and the corresponding convergence orders for the 1D smooth problem in \Cref{Ex:Smooth} computed by the third-order BP CDG scheme with $ N$ uniform cells. 
		}
		\label{tab:Ex-Smooth1D}
		\setlength{\tabcolsep}{3mm}{
			\begin{tabular}{ccccccc}
				\toprule[1.5pt]
				\multirow{2}{*}{$ N$} &
				\multicolumn{2}{c}{$ l^{1} $ norm} &
				\multicolumn{2}{c}{$ l^{2} $ norm} &
				\multicolumn{2}{c}{$ l^{\infty} $ norm} \\
				\cmidrule(r){2-3} \cmidrule(r){4-5} \cmidrule(l){6-7}
				& error & order &  error & order &  error & order \\
				
				\midrule[1.5pt]
				20  &6.53e{-}05&{-} &7.45e{-}05&{-} &1.41e{-}04&{-} \\
				40  &8.06e{-}06&3.02&9.21e{-}06&3.02&1.77e{-}05&3.00\\
				80  &1.00e{-}06&3.00&1.15e{-}06&3.00&2.20e{-}06&3.00\\
				160 &1.25e{-}07&3.00&1.43e{-}07&3.00&2.75e{-}07&3.00\\
				320 &1.57e{-}08&3.00&1.79e{-}08&3.00&3.44e{-}08&3.00\\
				640 &1.96e{-}09&3.00&2.24e{-}09&3.00&4.30e{-}09&3.00\\

				\bottomrule[1.5pt]
			\end{tabular}
		}
	\end{table}

	\begin{table}[!htb] 
		\centering
		\caption{Errors in $B_2$ at $t=1$ and the corresponding convergence orders for the 2D smooth problem in \Cref{Ex:Smooth} computed by the third-order BP locally DF CDG scheme with $ N\times N$ uniform cells. 
		}
		\label{tab:Ex-Smooth2D}
		\setlength{\tabcolsep}{3mm}{
			\begin{tabular}{ccccccc}
				\toprule[1.5pt]
				\multirow{2}{*}{$ N\times N$} &
				\multicolumn{2}{c}{$ l^{1} $ norm} &
				\multicolumn{2}{c}{$ l^{2} $ norm} &
				\multicolumn{2}{c}{$ l^{\infty} $ norm} \\
				\cmidrule(r){2-3} \cmidrule(r){4-5} \cmidrule(l){6-7}
				& error & order &  error & order &  error & order \\
				
				\midrule[1.5pt]
				$20\times 20$  & 1.92e{-}03&{-} &2.36e{-}03&{-} &8.45e{-}03&{-}  \\
				$40\times 40$  & 1.26e{-}04&3.93&2.22e{-}04&3.41&1.07e{-}03&2.98 \\
				$80\times 80$  & 1.56e{-}05&3.01&2.78e{-}05&3.00&1.35e{-}04&2.99 \\
				$160\times 160$& 1.96e{-}06&3.00&3.48e{-}06&3.00&1.70e{-}05&2.99 \\
				$320\times 320$& 2.45e{-}07&3.00&4.35e{-}07&3.00&2.13e{-}06&2.99 \\
				$640\times 640$& 3.07e{-}08&3.00&5.44e{-}08&3.00&2.67e{-}07&3.00 \\
				
				\bottomrule[1.5pt]
			\end{tabular}
		}
	\end{table}
	
\end{expl}

\begin{expl}[1D Riemann Problem \Rmnum{1}]\label{Ex:1DRP1} \rm   
	The initial conditions are
	\begin{equation*}
		(\rho,v_1,v_2,v_3,B_1,B_2,B_3,p) = 
		\begin{cases}
			(1,0,0,0,5,26,26,30), \quad 0\leq x \leq 0.5, \\
			(1,0,0,0,5,0.7,0.7,1), \quad 0.5< x \leq 0.5, \\
		\end{cases}
	\end{equation*}
	with a strong magnetic field $|\bm{B}|\approx37.108$ in the left state. 
	The computational domain $[0,1]$ is divided into $1000$ uniform cells with outflow boundary conditions. As in \cite{WuTangM3AS}, the ideal EOS \eqref{EOS:IDEOS} with $\Gamma=5/3$ is used. \Cref{fig:Ex-1DRP1} illustrates the numerical solutions at $t=0.4$ obtained using the BP CDG method. Since it is difficult to derive the exact solution to the 1D Riemann problem of the RMHD equations \eqref{eq:RMHD}, we compute a reference solution using the first-order Lax--Friedrichs scheme on a very fine mesh of $400,000$ uniform cells. It is observed that the results demonstrate good agreement with the reference solution and exhibit high resolution. If the BP limiter is not employed to enforce the proposed BP condition \eqref{eq:1DBPcond}, nonphysical numerical solutions would be produced after a few time steps, leading to immediate simulation failure.
	
	\begin{figure}[!htbp]
		\centering
		
		\begin{subfigure}{0.48\textwidth}
			\includegraphics[width=\textwidth]{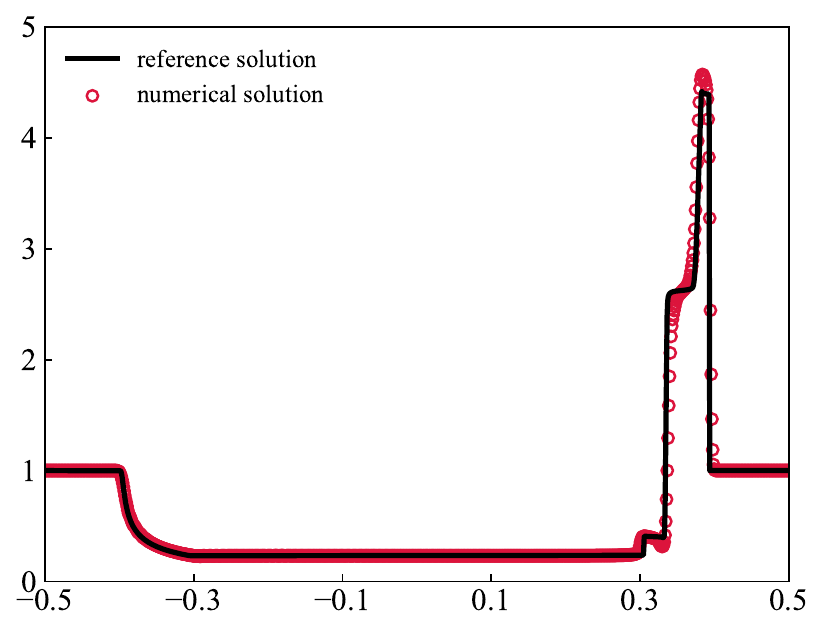}
						\caption{$\rho$}
		\end{subfigure}
		\hfill
		\begin{subfigure}{0.48\textwidth}
			\centering
			\includegraphics[width=\textwidth]{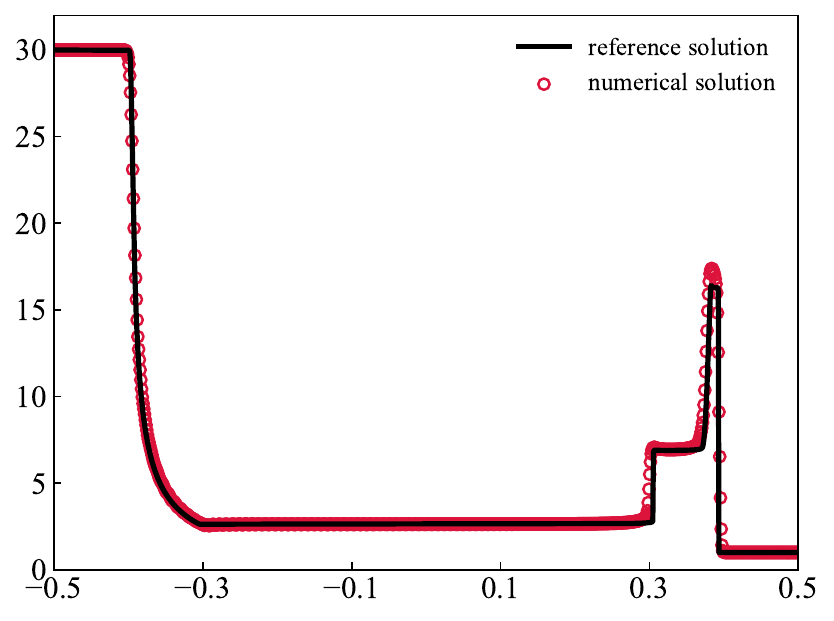}
						\caption{$p$}
		\end{subfigure}
		
		\begin{subfigure}{0.48\textwidth}
			\centering
			\includegraphics[width=\textwidth]{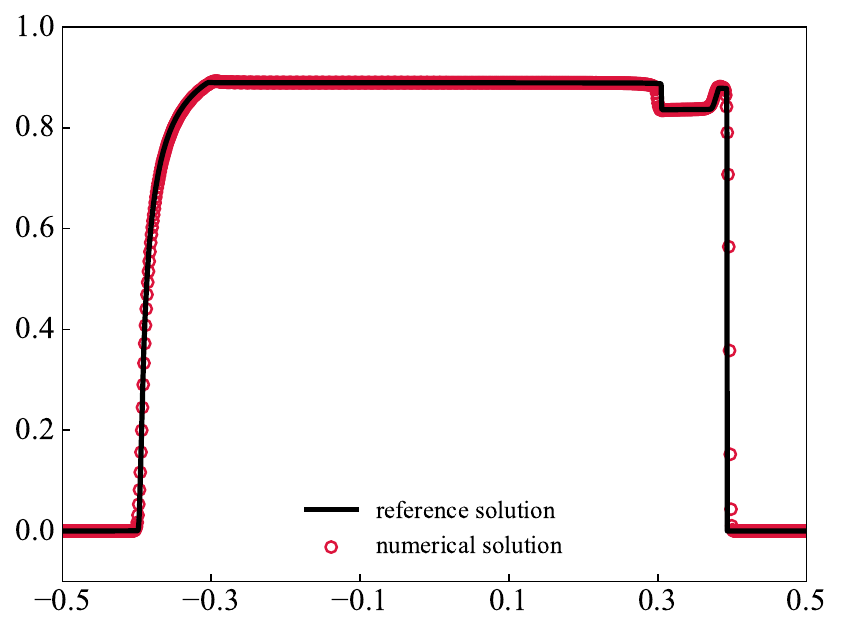}
						\caption{$v_1$}
		\end{subfigure}
		\hfill
		\begin{subfigure}{0.48\textwidth}
			\centering
			\includegraphics[width=\textwidth]{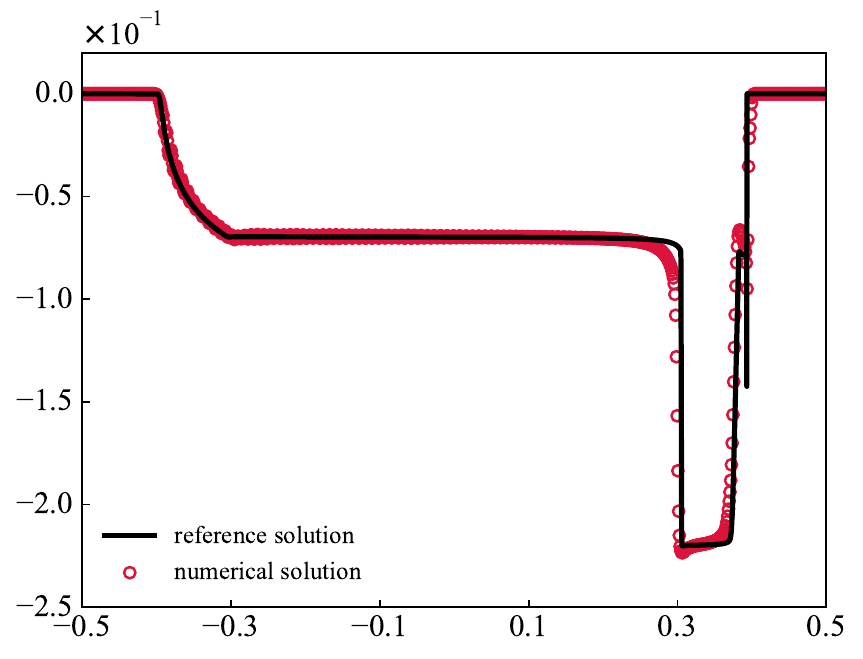}
						\caption{$v_2$}
		\end{subfigure}
		
		\begin{subfigure}{0.48\textwidth}
			\centering
			\includegraphics[width=\textwidth]{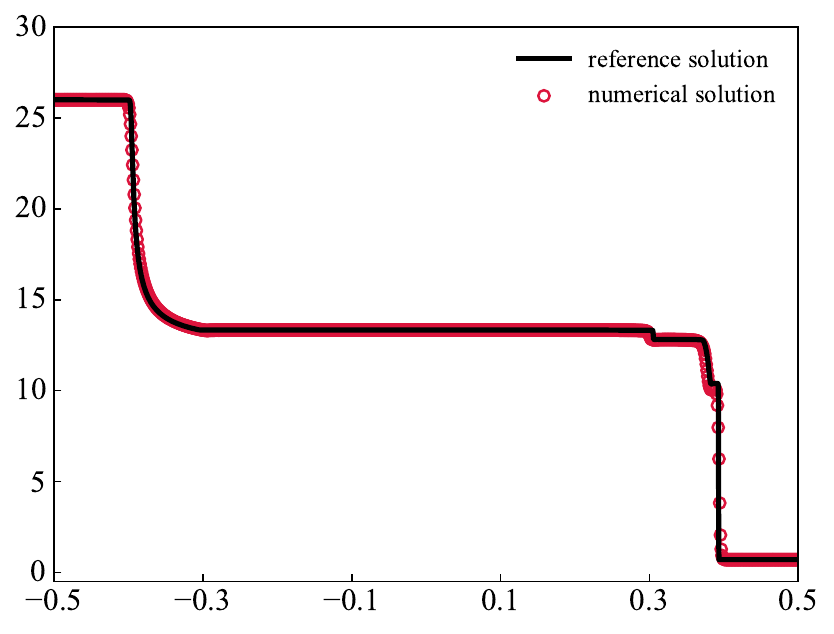}
						\caption{$B_2$}
		\end{subfigure}
		\hfill
		\begin{subfigure}{0.48\textwidth}
			\centering
			\includegraphics[width=\textwidth]{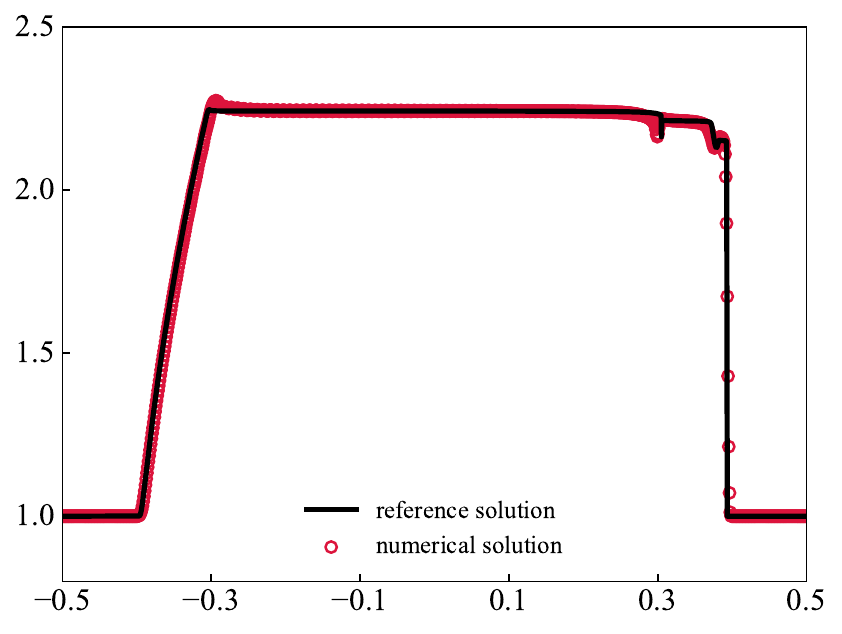}
						\caption{$W$}
		\end{subfigure}

		\caption{Numerical results obtained by the third-order BP CDG method for \Cref{Ex:1DRP1} with the ideal EOS \eqref{EOS:IDEOS}. 
		}
		\label{fig:Ex-1DRP1}
	\end{figure} 
	
\end{expl}

\begin{expl}[1D Riemann Problem \Rmnum{2}]\label{Ex:1DRP2} \rm 
		The initial conditions are given by
	\begin{equation*}
		(\rho,v_1,v_2,v_3,B_1,B_2,B_3,p) = 
		\begin{cases}
			(1,0,0,0,10,7,7,10^4), &\quad 0\leq x \leq 0.5, \\
			(1,0,0,0,10,0.7,0.7,10^{-8}), &\quad 0.5< x \leq 0.5. \\
		\end{cases}
	\end{equation*} 
This example is similar to \Cref{Ex:1DRP1}; however, it features an extremely low thermal pressure with a huge relative jump of approximately $10^{12}$. The plasma-beta is very low, being approximately $1.98\times 10^{-10}$. We employ outflow boundary conditions and adopt the RC-EOS \eqref{EOS:RCEOS}, which differs from the one used in \cite{WuTangM3AS}. \Cref{fig:Ex-1DRP2} displays the results obtained using the BP CDG method with $1000$ uniform cells at $t=0.4$. The reference solution is computed using the first-order Lax--Friedrichs scheme with $400,000$ uniform cells. The numerical results compare well against the reference solution. Again, for this example, the CDG code would also break down without the BP limiter to enforce the proposed BP condition \eqref{eq:1DBPcond}. 
	
	\begin{figure}[!htbp]
		\centering
		\begin{subfigure}{0.45\textwidth}
			\includegraphics[width=\textwidth]{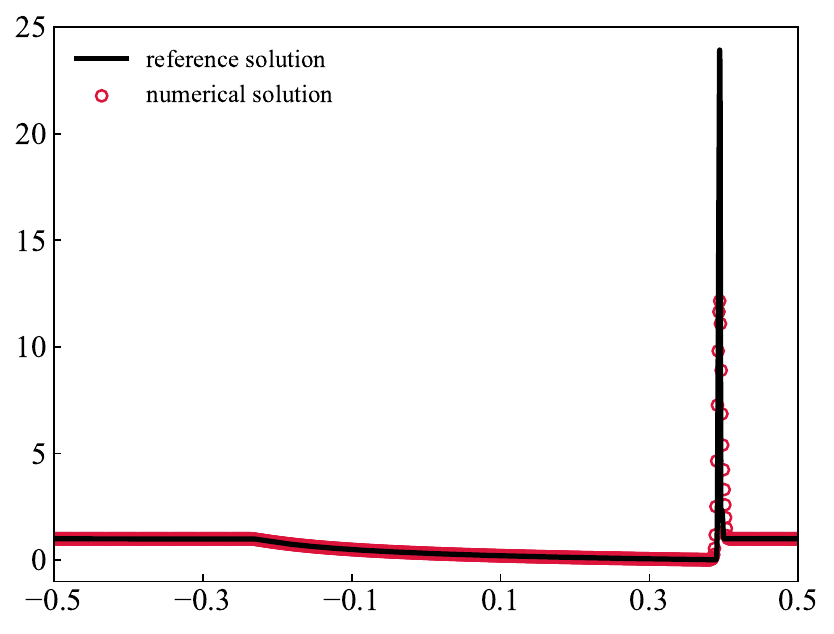}
						\caption{$\rho$}
		\end{subfigure}
		\hfill
		\begin{subfigure}{0.48\textwidth}
			\centering
			\includegraphics[width=\textwidth]{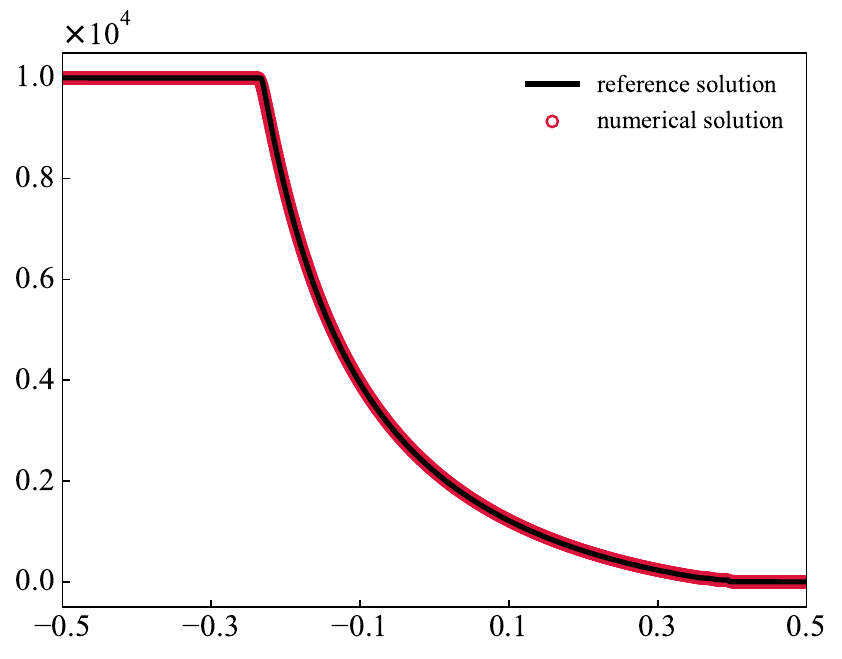}
						\caption{$p$}
		\end{subfigure}
		
		\begin{subfigure}{0.48\textwidth}
			\centering
			\includegraphics[width=\textwidth]{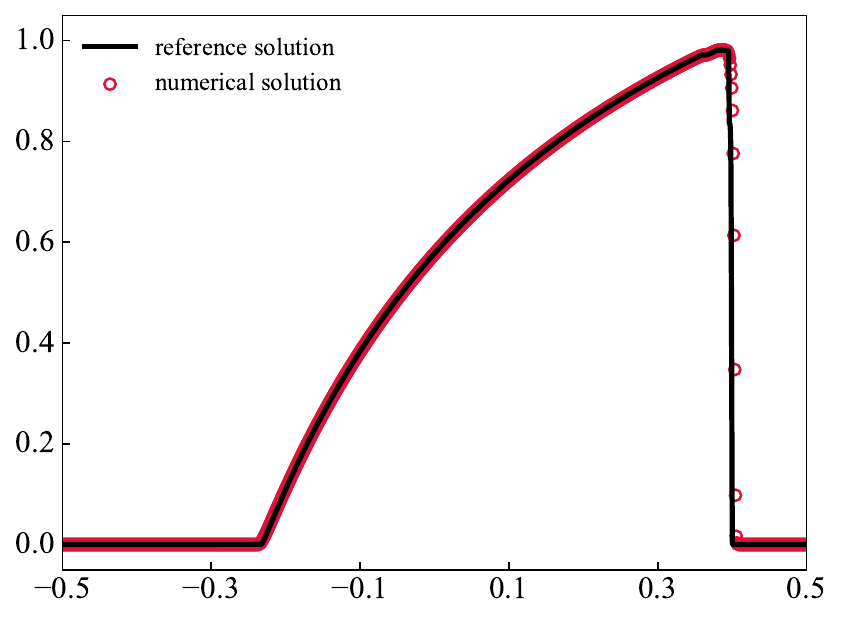}
						\caption{$v_1$}
		\end{subfigure}
		\hfill
		\begin{subfigure}{0.48\textwidth}
			\centering
			\includegraphics[width=\textwidth]{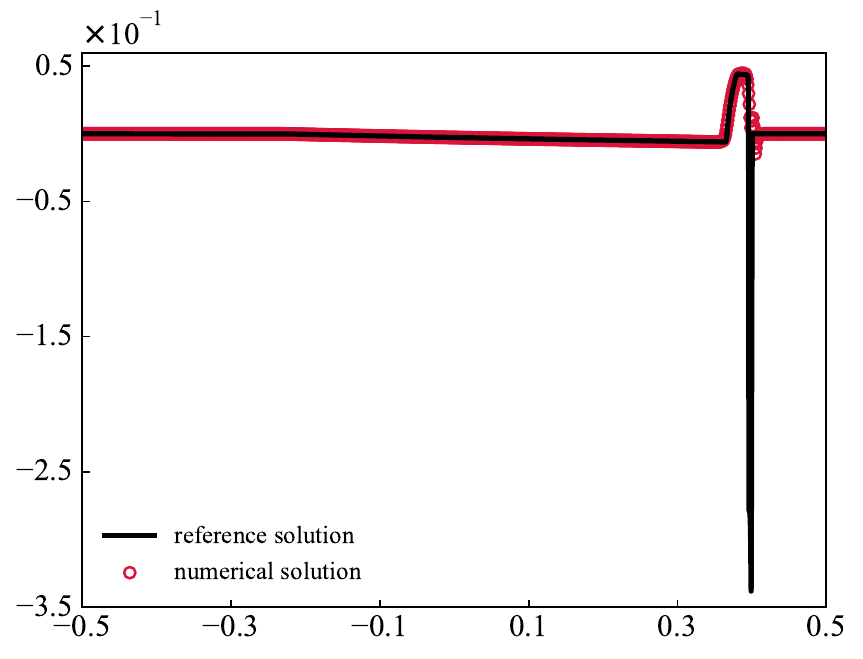}
						\caption{$v_2$}
		\end{subfigure}
		
		\begin{subfigure}{0.48\textwidth}
			\centering
			\includegraphics[width=\textwidth]{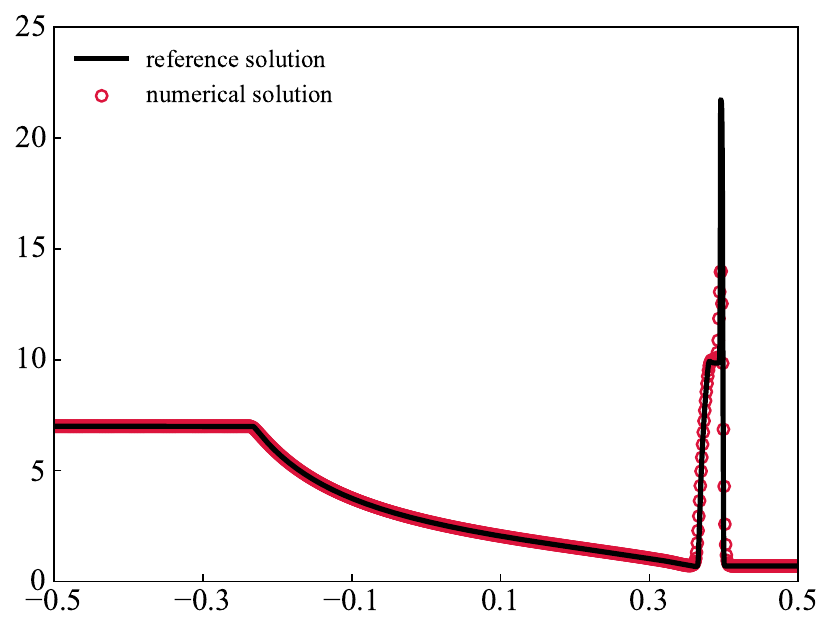}
						\caption{$B_2$}
		\end{subfigure}
		\hfill
		\begin{subfigure}{0.48\textwidth}
			\centering
			\includegraphics[width=\textwidth]{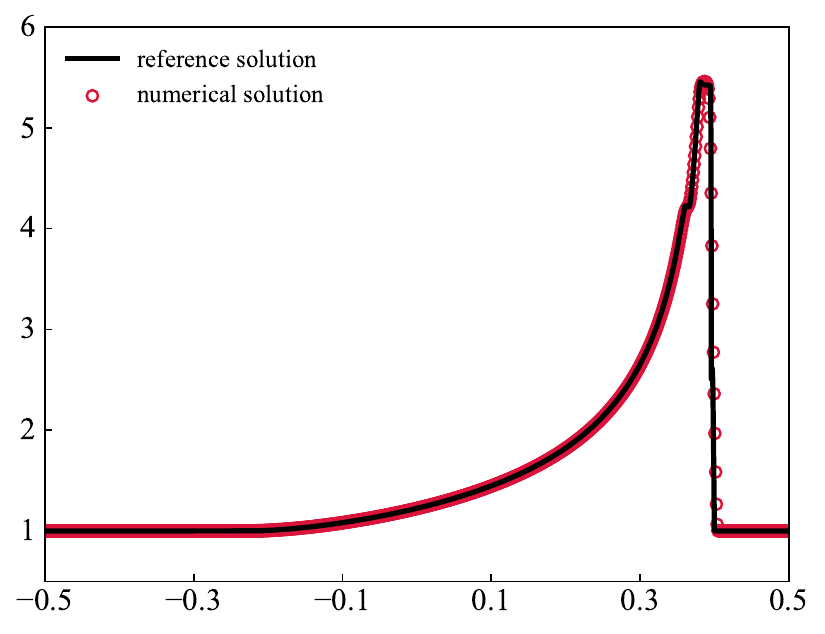}
						\caption{$W$}
		\end{subfigure}
		
		\caption{Numerical results obtained by the third-order BP CDG method for \Cref{Ex:1DRP2} with the RC-EOS \eqref{EOS:RCEOS}. 
		}
		\label{fig:Ex-1DRP2}
	\end{figure} 
	
\end{expl}

\begin{expl}[1D Riemann problem \Rmnum{3}]\label{Ex:1DRP3} \rm 
This is an ultra RMHD test case that describes a strong collision between two high-speed flows with a large Lorentz factor of $223.61$. The BP limiting procedure is also necessary for the successful simulation of this demanding example using the CDG method. This problem is initialized with
	\begin{equation*}
		(\rho,v_1,v_2,v_3,B_1,B_2,B_3,p) = 
		\begin{cases}
			(1,0.99999,0,0,100,70,70,0.1), &\quad 0\leq x \leq 0.5, \\
			(1,-0.99999,0,0,100,-70,-70,0.1), &\quad 0.5< x \leq 0.5. \\
		\end{cases}
	\end{equation*}
Over time, the exact solution involves two fast and two slow reflected shock waves, and a high-pressure region between the two slow shock waves. This problem was initially proposed and analyzed in \cite{WuTangM3AS} with the ideal EOS \eqref{EOS:IDEOS}. \Cref{fig:Ex-1DRP3-EOS3} shows the results using the BP CDG methods with the IP-EOS \eqref{EOS:IPEOS} at $t=0.4$, where a mesh of $1000$ uniform cells and outflow boundary conditions are employed. The reference solution is obtained using the first-order Lax--Friedrichs scheme with $400,000$ uniform cells. One can observe that our BP CDG method is capable of accurately capturing the strong shock waves. The wall-heating-type phenomenon can be observed in the density profile around $x=0$, as commonly reported in the literature \cite{HeTang2012RMHD,WuTangM3AS}.

	\begin{figure}[!htbp]
		\centering
		\begin{subfigure}{0.48\textwidth}
			\includegraphics[width=\textwidth]{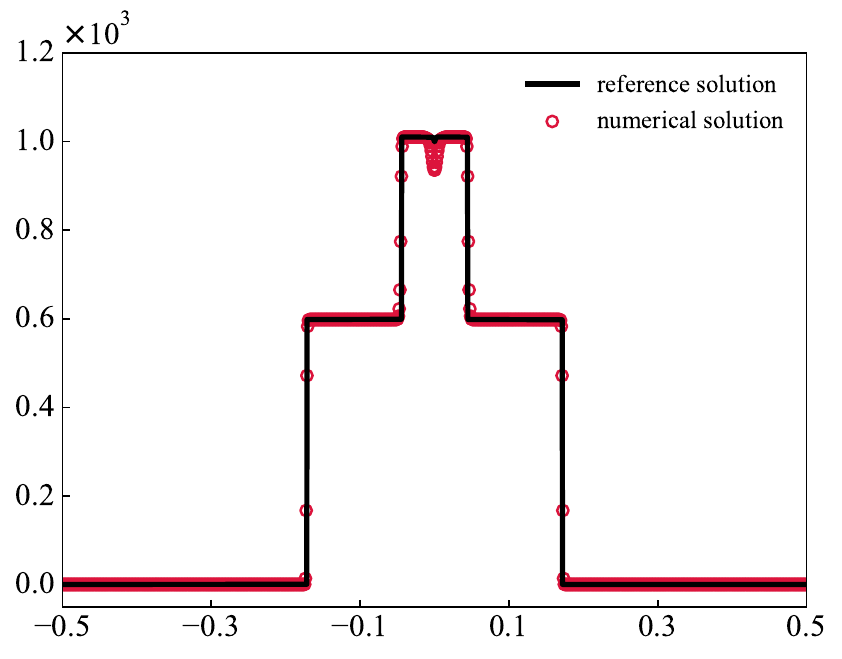}
						\caption{$\rho$}
		\end{subfigure}
		\hfill
		\begin{subfigure}{0.48\textwidth}
			\centering
			\includegraphics[width=\textwidth]{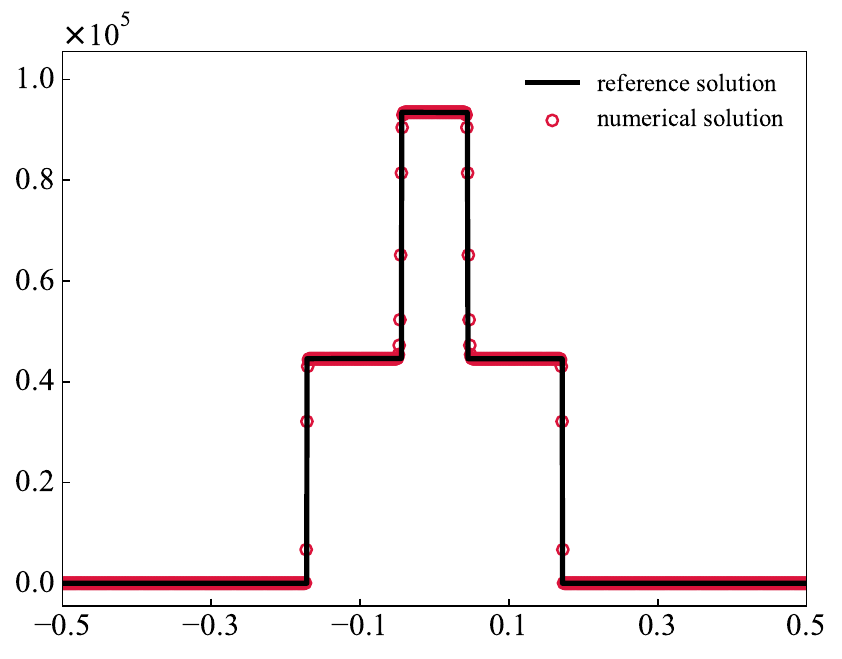}
						\caption{$p$}
		\end{subfigure}
		
		\begin{subfigure}{0.48\textwidth}
			\centering
			\includegraphics[width=\textwidth]{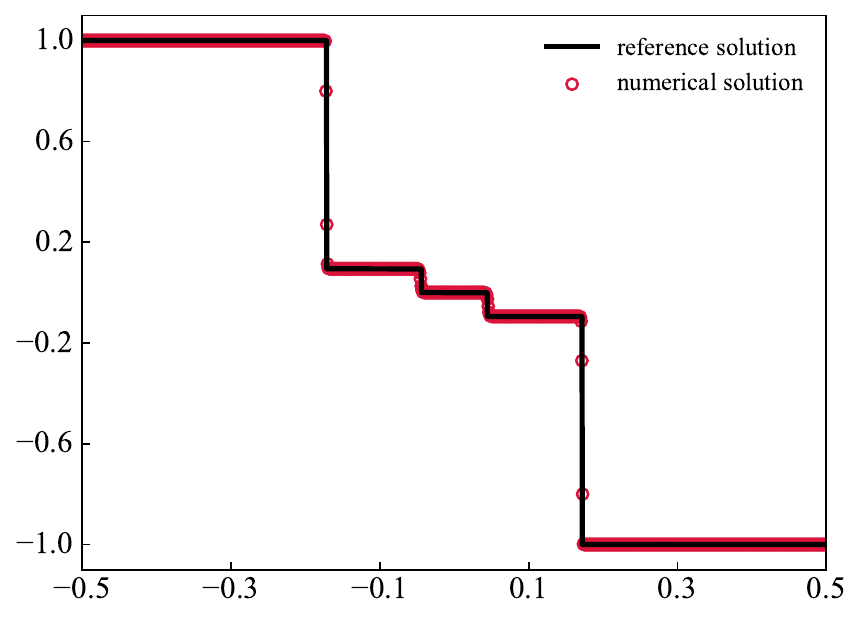}
						\caption{$v_1$}
		\end{subfigure}
		\hfill
		\begin{subfigure}{0.48\textwidth}
			\centering
			\includegraphics[width=\textwidth]{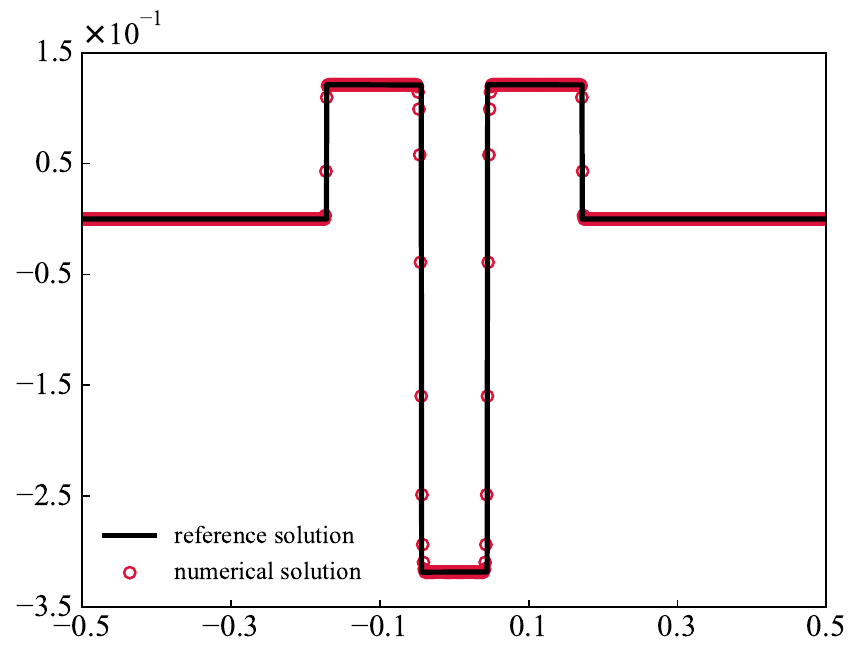}
						\caption{$v_2$}
		\end{subfigure}
		
		\begin{subfigure}{0.48\textwidth}
			\centering
			\includegraphics[width=\textwidth]{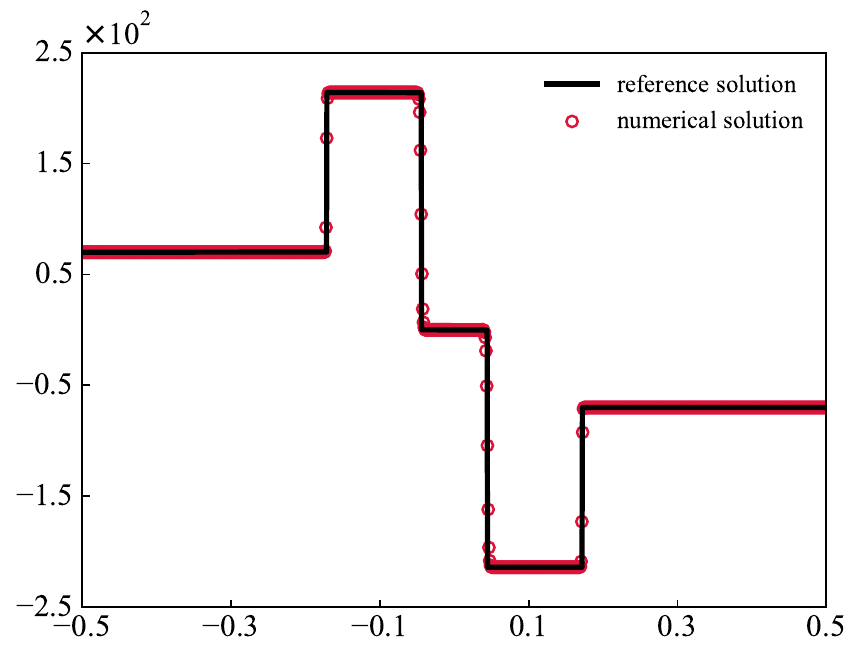}
						\caption{$B_2$}
		\end{subfigure}
		\hfill
		\begin{subfigure}{0.48\textwidth}
			\centering
			\includegraphics[width=\textwidth]{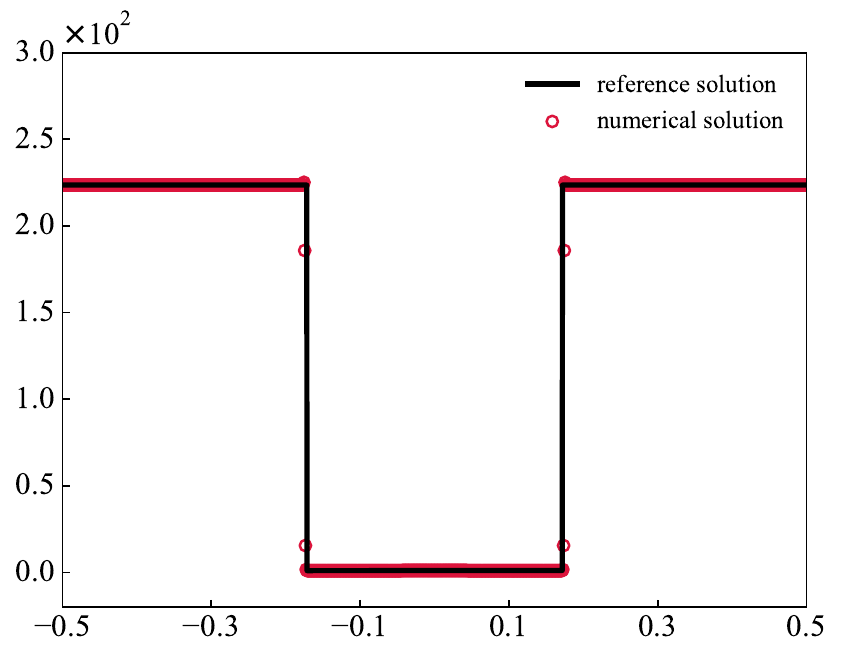}
						\caption{$W$}
		\end{subfigure}
		
		\caption{Numerical results obtained by the third-order BP CDG method for \Cref{Ex:1DRP3} with the IP-EOS \eqref{EOS:IPEOS}. 
		}
		\label{fig:Ex-1DRP3-EOS3}
	\end{figure} 
\end{expl}

\begin{expl}[Orszag--Tang Problem]\label{Ex:OT} 
We start the 2D simulations with an Orszag--Tang problem, following the setup in \cite{Host:2008,WuShu2020NumMath} but adopting the TM-EOS \eqref{EOS:TMEOS}. The smooth initial conditions are taken as
	\begin{equation*}
		(\rho, \bm{v}, \bm{B}, p) =  (1, -A\sin y, A\sin x, 0, -\sin y, \sin (2x), 0, 10),
	\end{equation*}
where the constant $A=0.99/\sqrt{2}$. In this problem, the initial maximum velocity reaches $0.99c$, corresponding to a Lorentz factor $W\approx7.09$. The computational domain $[0,2\pi]^2$ is discretized using $600\times 600$ uniform cells, and periodic boundary conditions are applied to all boundaries. \Cref{fig:Ex-OT} presents the numerical results obtained using our BP locally DF CDG method for the logarithm of the rest-mass density $\log{\rho}$ and the Lorentz factor $W$ at $t=2.82$ and $6.85$. As time progresses, complex wave structures are generated and correctly captured by our method. The results agree with those reported in \cite{Host:2008,WuShu2019SISC,WuShu2020NumMath}. By utilizing the BP limiter, the robustness of the CDG method is enhanced. However, without it, the CDG code would fail at $t\approx2.196$ due to nonphysical solutions. 
	
To further demonstrate the stability and robustness of the proposed BP locally DF CDG method, we quantitatively evaluate the evolution of the numerical divergence error over time, following the approach described in \cite{WuShu2020NumMath,WuJiangShu2022}. Define $\jump{\langle \bm{n}, \bm{B}_h^I \rangle}$ as the jump of the normal component of $\bm{B}_h$ on the edge ${\mathcal E}_h^I$ of the primal mesh $\mathcal{T}_h^I$. The global divergence error in $\bm{B}_h^I$ is evaluated as
	\begin{equation*}
		\Vert \nabla \cdot \bm{B}_h^I \Vert := \sum_{{\mathcal E}_h^I\in \mathcal{T}_h^I} \int_{{\mathcal E}_h^I}\Big|\jump{\langle \bm{n}, \bm{B}_h^I \rangle}\Big| \dd s + \sum_{i,j}\int_{I_{ij}} \big|\nabla \cdot \bm{B}_h^I \big| \dd x \dd y.
	\end{equation*}
	on the primal mesh $\mathcal{T}_h^I$. 
	Then, the global relative divergence error is defined as 
	\begin{equation*}
		\varepsilon_{\rm div} := \frac{\Vert \nabla \cdot \bm{B}_h^I \Vert}{\Vert \bm{B}_h^I \Vert },
	\end{equation*}
	where 
	\begin{equation*}
		\Vert \bm{B}_h^I \Vert := \sum_{{\mathcal E}_h^I\in \mathcal{T}_h^I} \int_{{\mathcal E}_h^I} \average{\big|\bm{B}_h^I \big|} \dd s + \sum_{i,j}\int_{I_{ij}} \big| \bm{B}_h^I \big| \dd x \dd y.
	\end{equation*}
	\Cref{fig:Ex-OT-gdivB} displays the time evolution of the divergence errors $\varepsilon_{\rm div}$, which remain small and of the order $\mathcal{O}(10^{-4})$.
	
	\begin{figure}[!htb]
		\centering
		\begin{subfigure}{0.48\textwidth}
			\includegraphics[width=\textwidth]{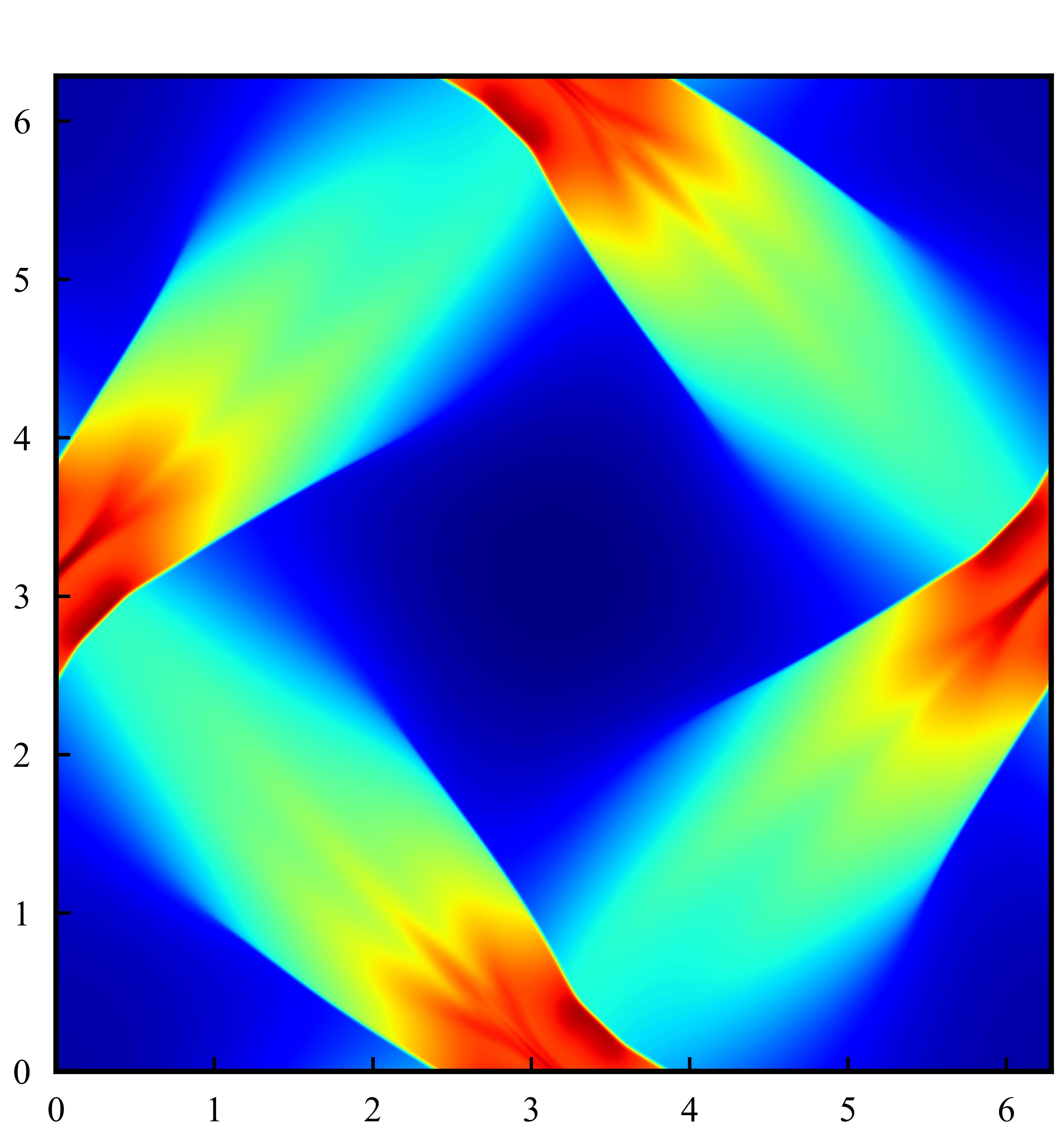}
		\end{subfigure}
		\hfill	
		\begin{subfigure}{0.48\textwidth}
			\centering
			\includegraphics[width=\textwidth]{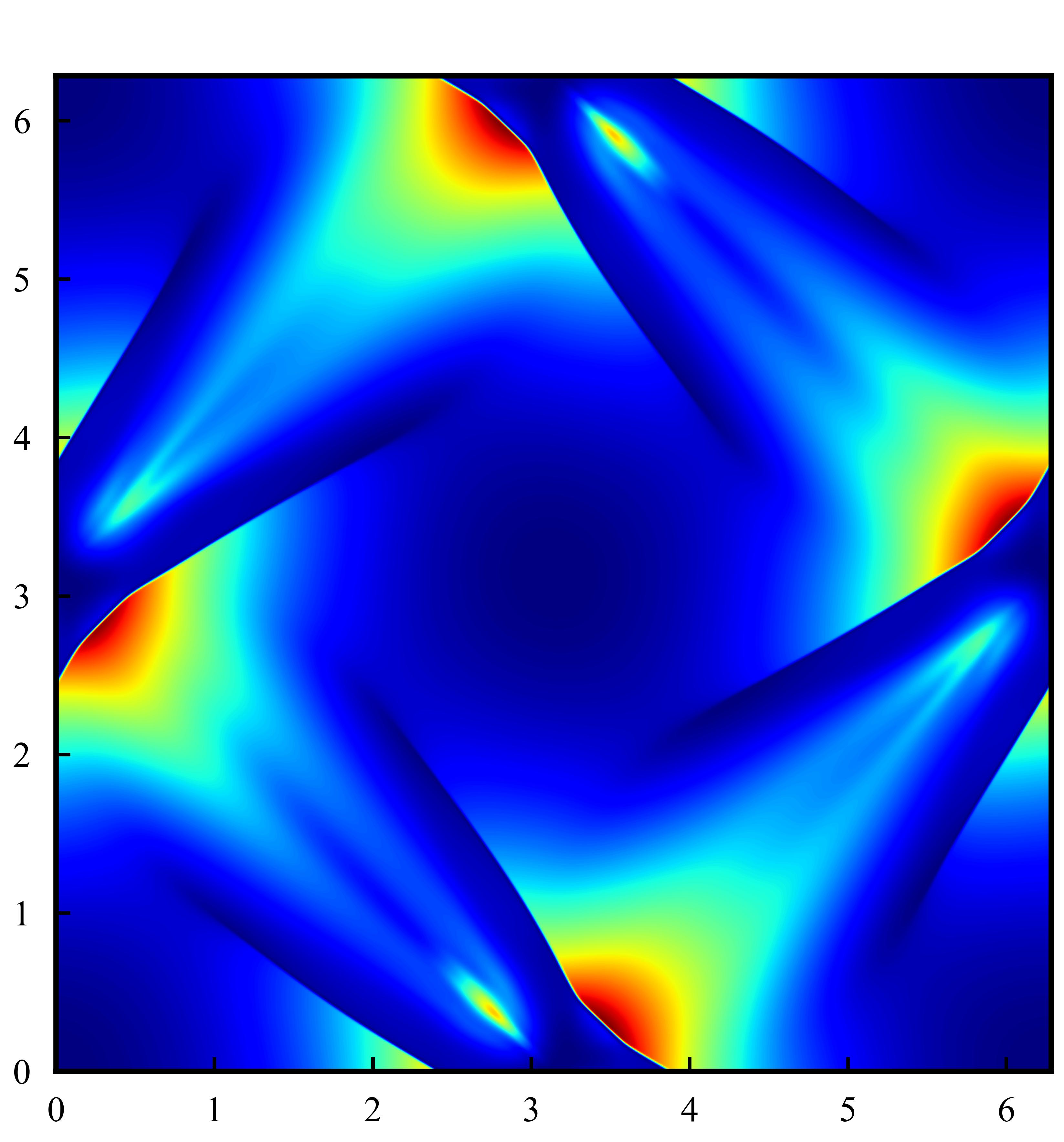}
		\end{subfigure}
		
		\begin{subfigure}{0.48\textwidth}
			\includegraphics[width=\textwidth]{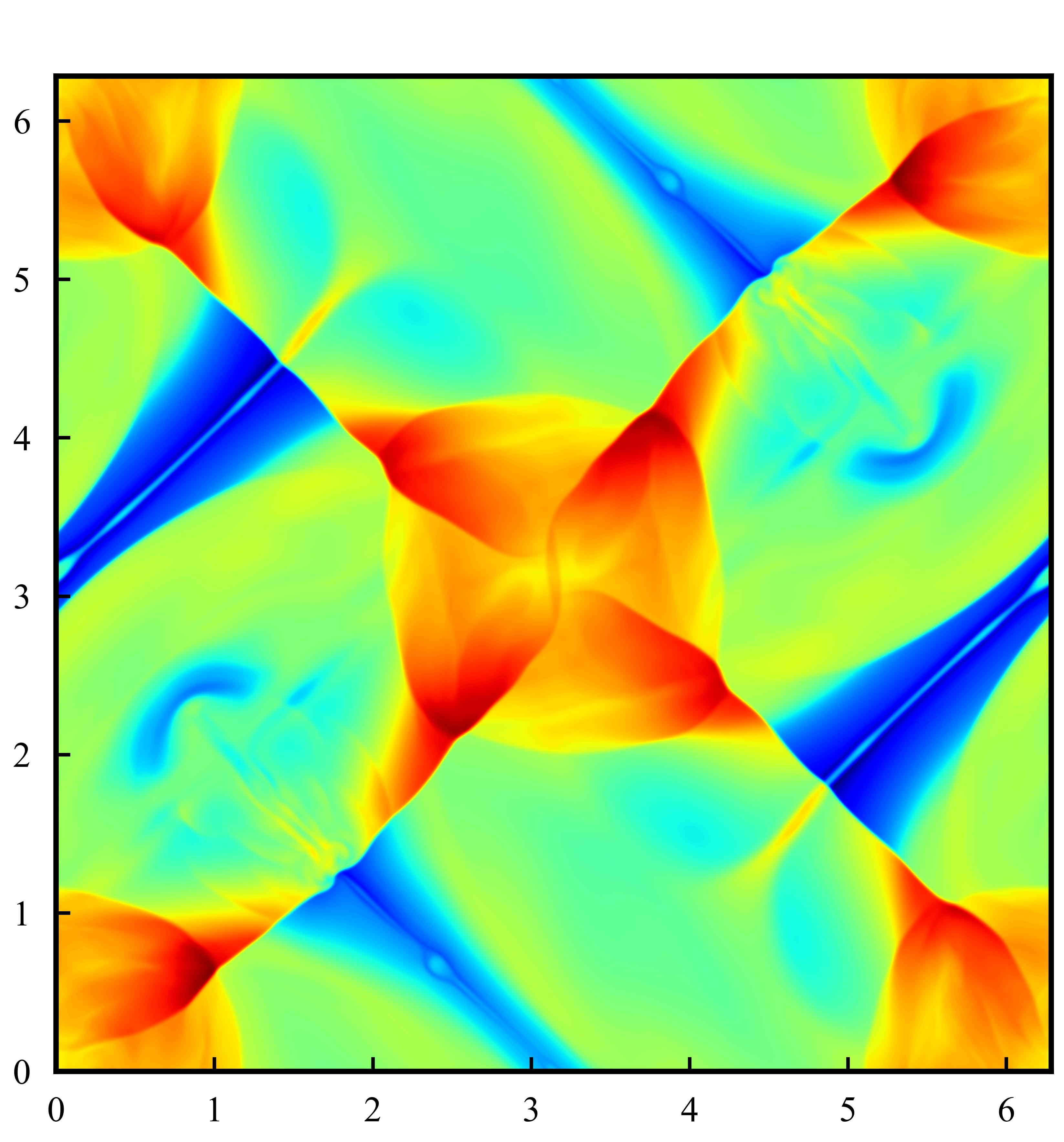}
		\end{subfigure}
		\hfill	
		\begin{subfigure}{0.48\textwidth}
			\centering
			\includegraphics[width=\textwidth]{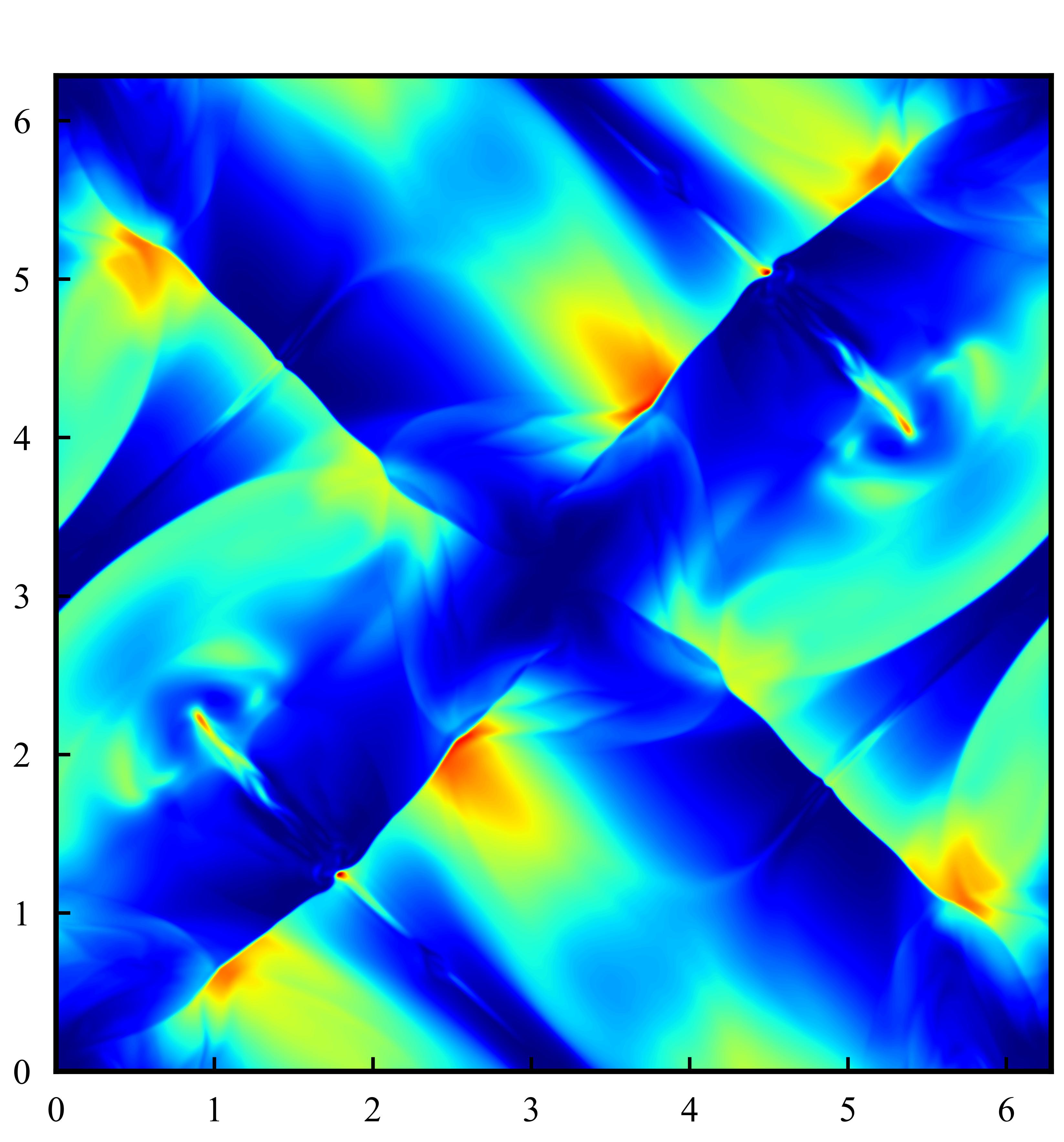}
		\end{subfigure}
		
		\caption{Density logarithm $\log{\rho}$ (left) and Lorentz factor $W$ (right) at $t=2.82$ (top) and $6.85$ (bottom) for \Cref{Ex:OT} with the TM-EOS \eqref{EOS:TMEOS}. 
		}
		\label{fig:Ex-OT}
	\end{figure} 
\end{expl}

\begin{figure}[!htb]
	\centering
	\begin{subfigure}{0.32\textwidth}
		\includegraphics[width=\textwidth]{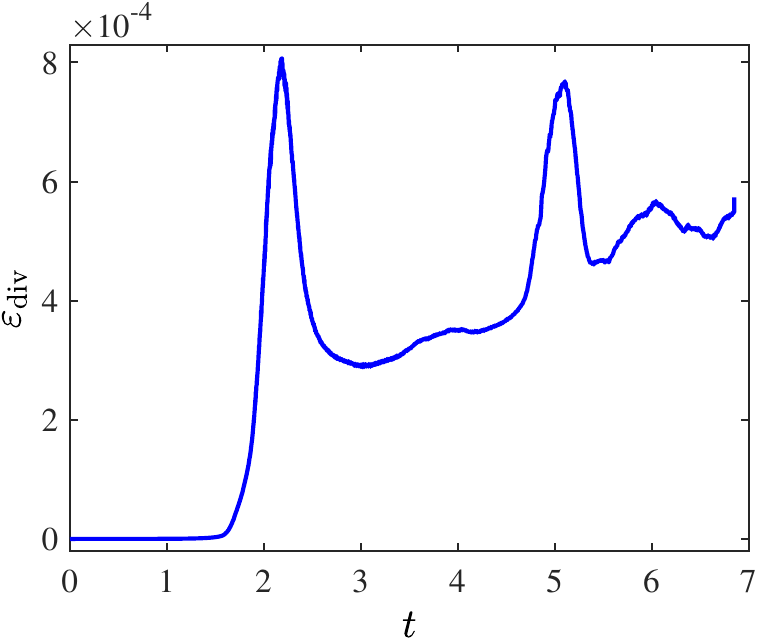}
		\caption{ \Cref{Ex:OT}}
		\label{fig:Ex-OT-gdivB}
	\end{subfigure}
	\hfill
	\begin{subfigure}{0.32\textwidth}
		\includegraphics[width=\textwidth]{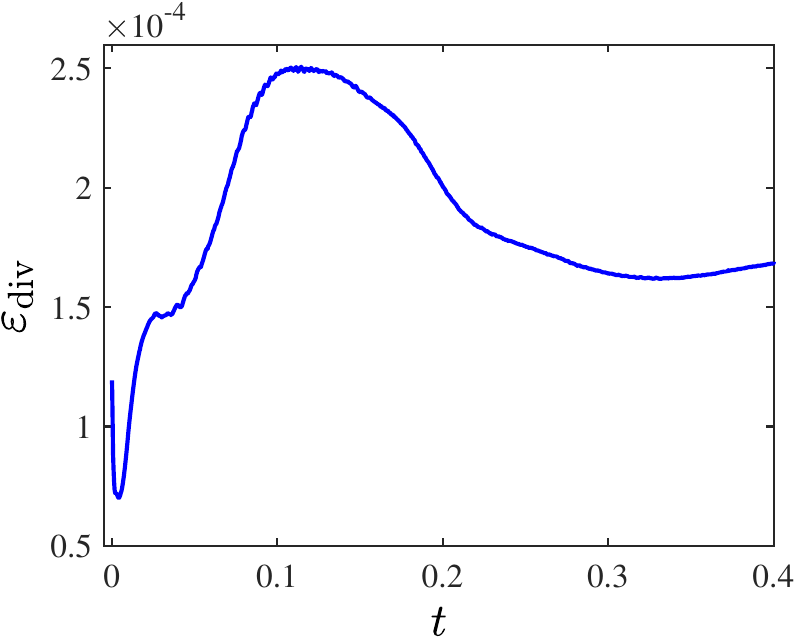}
		\caption{ \Cref{Ex:Rotor}}
		\label{fig:Ex-Rotor-gdivB}
	\end{subfigure}
	\hfill
	\begin{subfigure}{0.32\textwidth}
		\includegraphics[width=\textwidth]{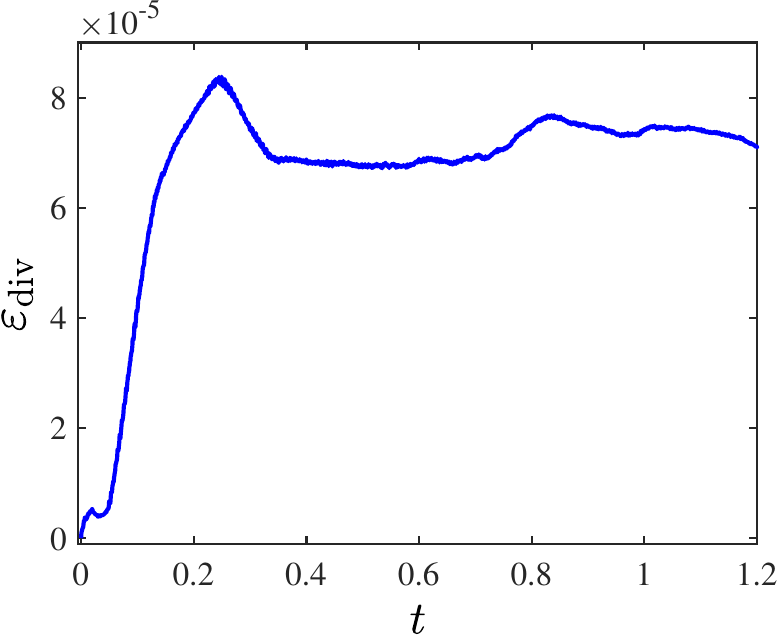}
		\caption{ \Cref{Ex:ShockCloud}}
		\label{fig:Ex-SC-gdivB}
	\end{subfigure}
	
	\caption{Time evolution of the global relative divergence error $\varepsilon_{\rm div}$.
	}
	\label{fig:Ex-OT-Rotor-SC}
\end{figure}

\begin{expl}[Rotor Problem]\label{Ex:Rotor}
	 This is a relativistic extension of the non-relativistic MHD rotor
	problem \cite{BalsaraSpicer1999,BALSARA2014172}. 
	The initial conditions are given by  
	{\begin{equation*}
			(\rho, \bm{v}, \bm{B}, p) = 
			\begin{cases}
				(10, -\alpha y, \alpha x, 0, 1, 0, 0, 1), &  r \leq r_1, \\
				(1+9 \delta, -\alpha y \delta r_1/r, \alpha x \delta r_1/r, 0, 1, 0, 0, 1), &  r_1< r \leq r_2, \\
				(1,0,0,0, 1, 0, 0, 1 ), &  r_2<r, \\
			\end{cases}
	\end{equation*}}
	where $\alpha=9.95$, $\delta=(r_2-r)/(r_2-r_1)$, $r=\sqrt{x^2+y^2}$, $r_1=0.1$, and $r_2=0.115$.
	This problem describes a dense fluid disk rotating in an ambient fluid with the RC-EOS \eqref{EOS:RCEOS}. The computational domain is taken as $[-0.5,0.5]^2$, partitioned into $400\times 400$ uniform cells, and outflow boundary conditions are used. \Cref{fig:Ex-Rotor} displays the numerical results at $t=0.4$ obtained by the proposed BP locally DF CDG scheme. We observe that an oblong-shaped shell forms at the center. The computed results agree quite well with those simulated in \cite{HeTang2012RMHD,ZhaoTang2017} with the ideal EOS, but exhibit different patterns compared to the non-relativistic MHD rotor
	\cite{BalsaraSpicer1999,BALSARA2014172}. In addition, \Cref{fig:Ex-Rotor-gdivB} presents the time evolution of the global relative divergence error $\varepsilon_{\rm div}$, which is maintained at the order of ${\mathcal O}(10^{-4})$.
	
	\begin{figure}[!htb]
		\centering		
		\begin{subfigure}{0.48\textwidth}
			\includegraphics[width=\textwidth]{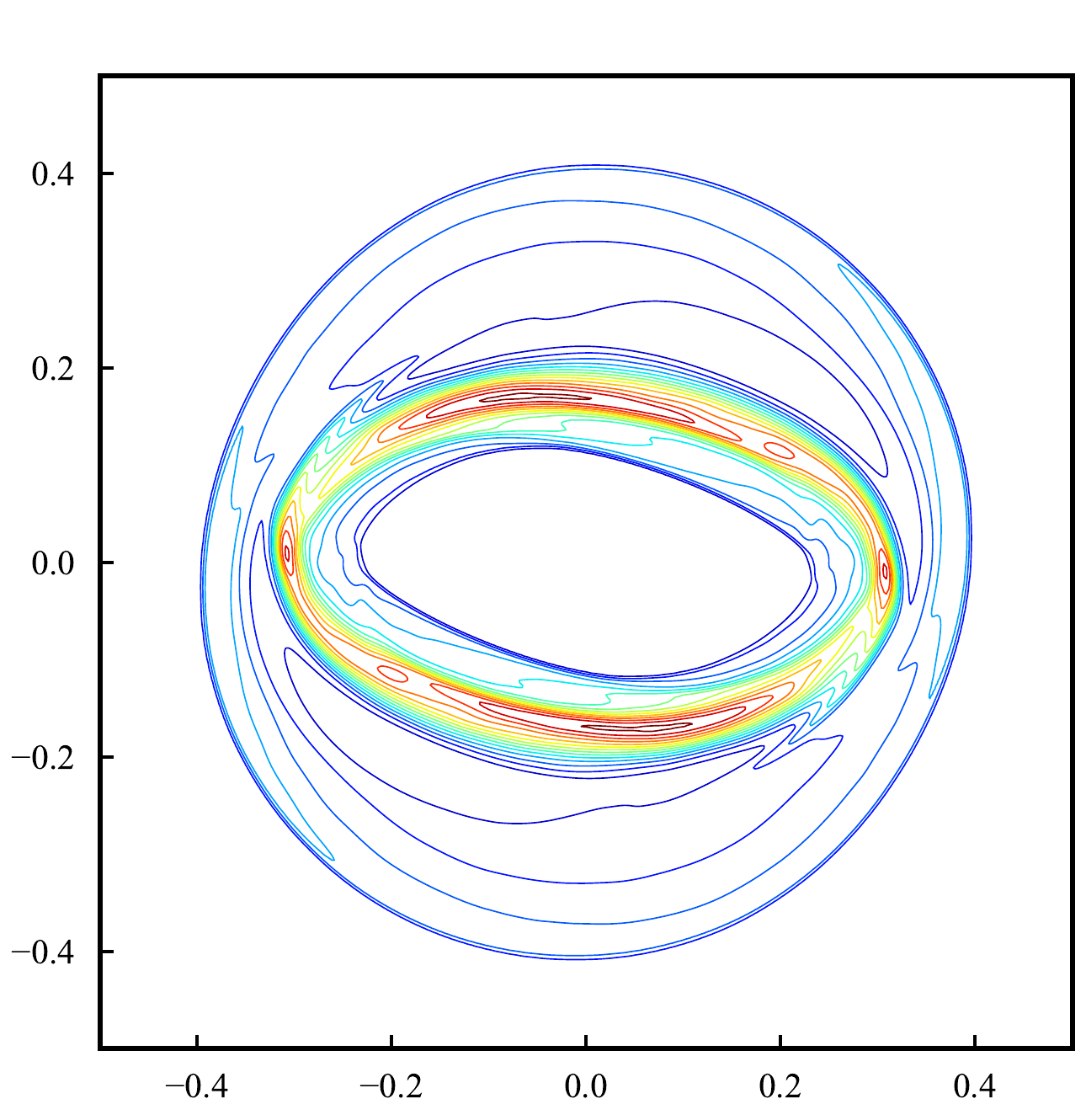}
		\end{subfigure}
		\hfill
		\begin{subfigure}{0.48\textwidth}
			\includegraphics[width=\textwidth]{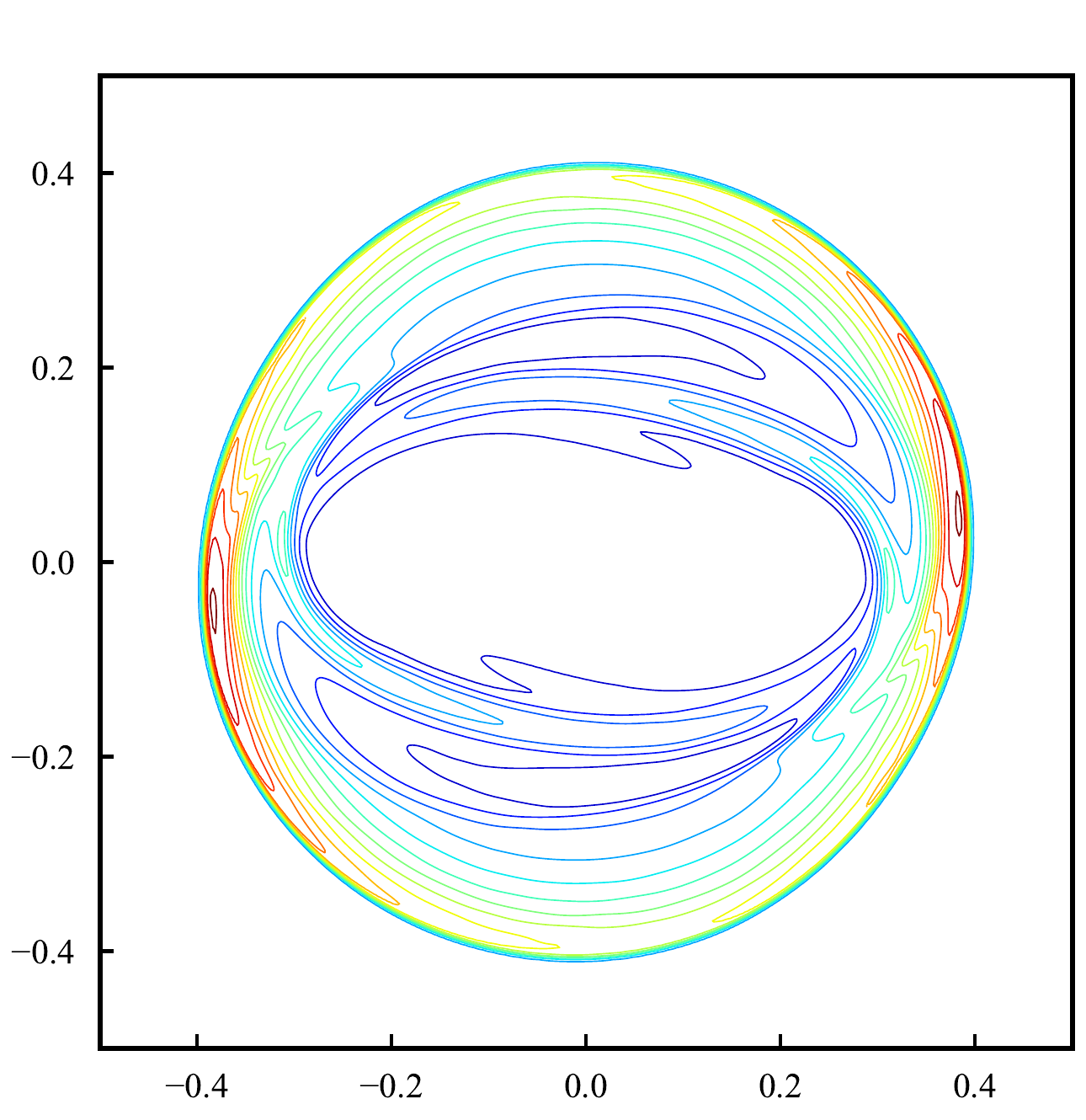}
		\end{subfigure}
		
		\begin{subfigure}{0.48\textwidth}
			\includegraphics[width=\textwidth]{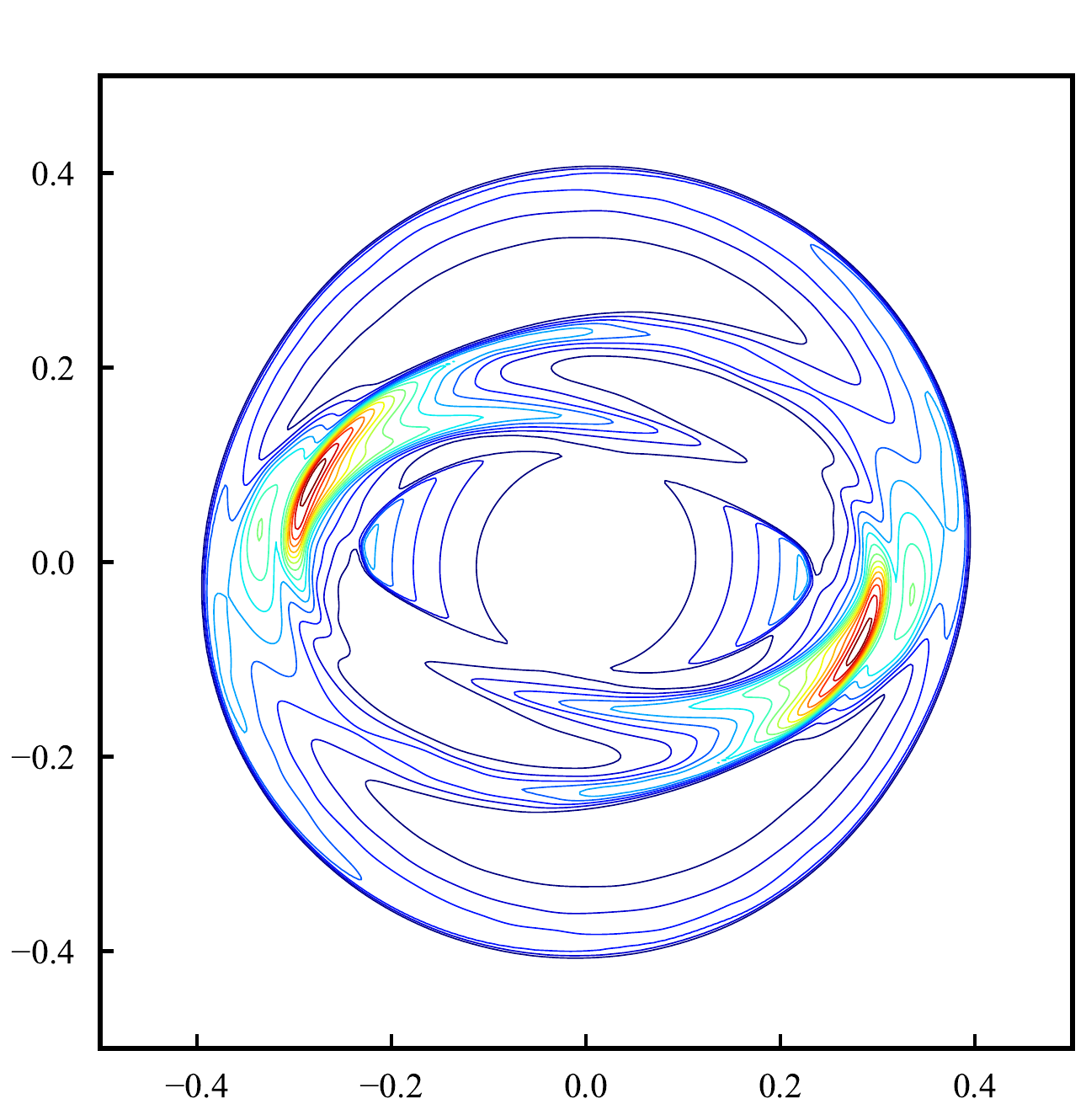}
		\end{subfigure}
		\hfill
		\begin{subfigure}{0.48\textwidth}
			\includegraphics[width=\textwidth]{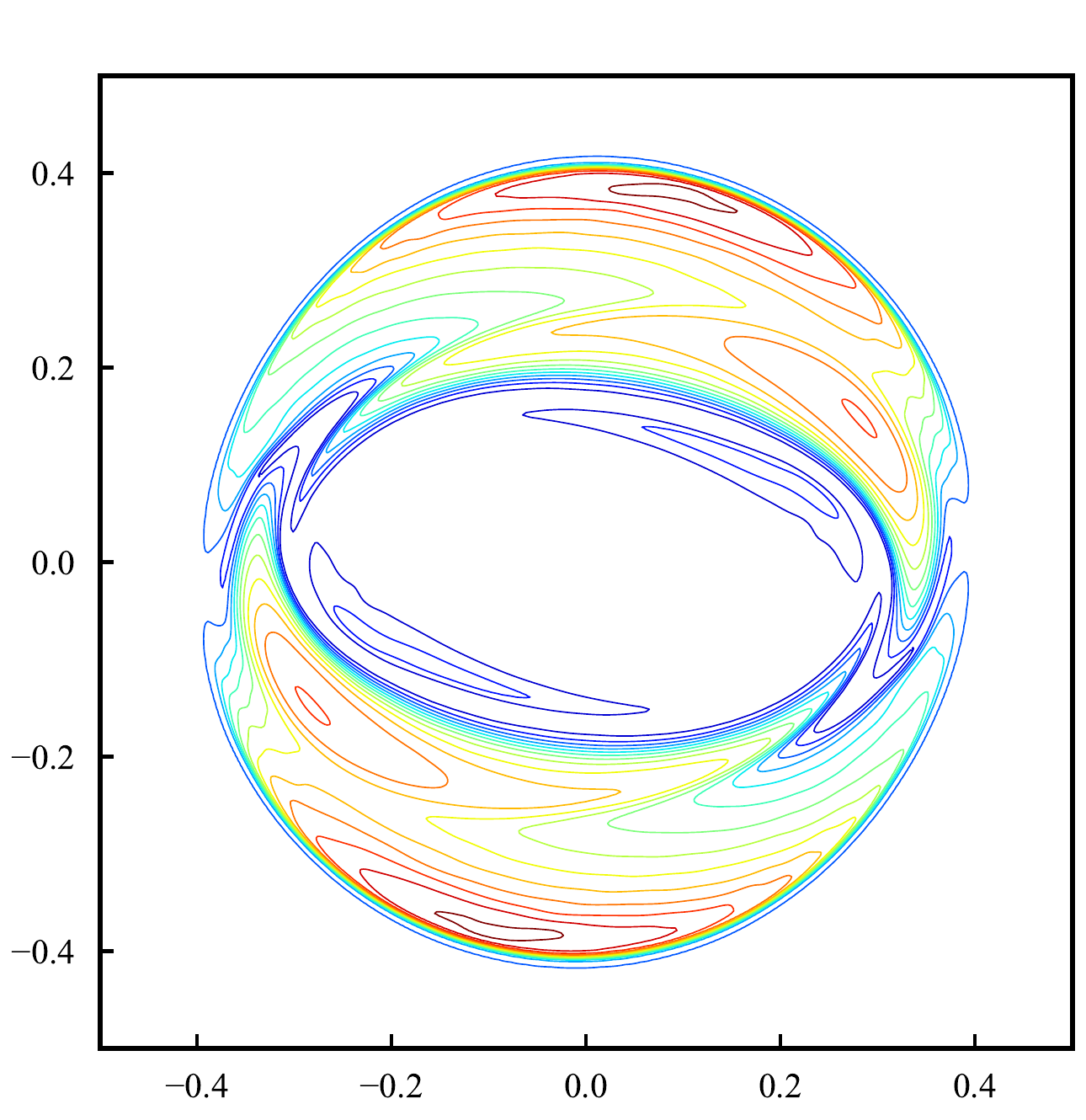}
		\end{subfigure}
		
		\caption{Rest-mass density (top-left), thermal pressure (top-right), Lorentz factor (bottom-left), and magnetic pressure (bottom-right) for \Cref{Ex:Rotor} with the RC-EOS \eqref{EOS:RCEOS}. 
		}
		\label{fig:Ex-Rotor}
	\end{figure} 
\end{expl}

\begin{expl}[Shock-Cloud Interaction Problem]\label{Ex:ShockCloud}
This benchmark problem simulates the interaction of a strong shock wave with a high-density cloud. The setup is the same as that in \cite{HeTang2012RMHD,WuTangM3AS}. The computational domain is $[-0.2,1.2]\times[0,1]$ with a resolution of $560\times 400$ uniform cells. Initially, a shock parallel to the $y$-axis is located at $x=0.05$ and moves in the right direction. The initial post-shock state is specified as
	\begin{equation*}
		(\rho, \bm{v}, \bm{B}, p) = (3.86859, 0.68,0,0, 0, 0.8498108108786, -0.8498108108786, 1.251148954517),
	\end{equation*}	
while the pre-shock condition is given by
	\begin{equation*}
		(\bm{v}, \bm{B}, p) = (0,0,0, 0, 0.1610642582333, 0.1610642582333, 0.05),  \quad
		\rho=
		\begin{cases}
			30,  &\sqrt{(x-0.25)^2+(y-0.5)^2} \leq 0.15, \\
			1,  &\text{otherwise},
		\end{cases}
	\end{equation*}
which includes a circular cloud with high density. Outflow conditions are applied to all boundaries except for the left one, where an inflow condition is imposed. The ideal EOS \eqref{EOS:IDEOS} with $\Gamma=5/3$ is used for this simulation. \Cref{fig:Ex-ShockCloud} shows the schlieren images of the rest-mass density logarithm $\log{\rho}$, thermal pressure logarithm $\log{p}$, Lorentz factor $W$, and the magnitude of the magnetic field $|\bm{B}|$ at $t=1.2$ computed by our BP locally DF CDG method. The complex wave structures are clearly captured by our scheme with high resolution, and the results are consistent with those reported in \cite{HeTang2012RMHD,WuTangM3AS}. The BP limiter is crucial for this problem, as not using it would cause the CDG code to fail at the first time step. As shown in \Cref{fig:Ex-SC-gdivB}, the global relative divergence error $\varepsilon_{\rm div}$ is maintained at about $10^{-5}$, further demonstrating
the good stability of our CDG method.
	
	\begin{figure}[!htb]
		\centering
		\begin{subfigure}{0.48\textwidth}
			\includegraphics[width=\textwidth]{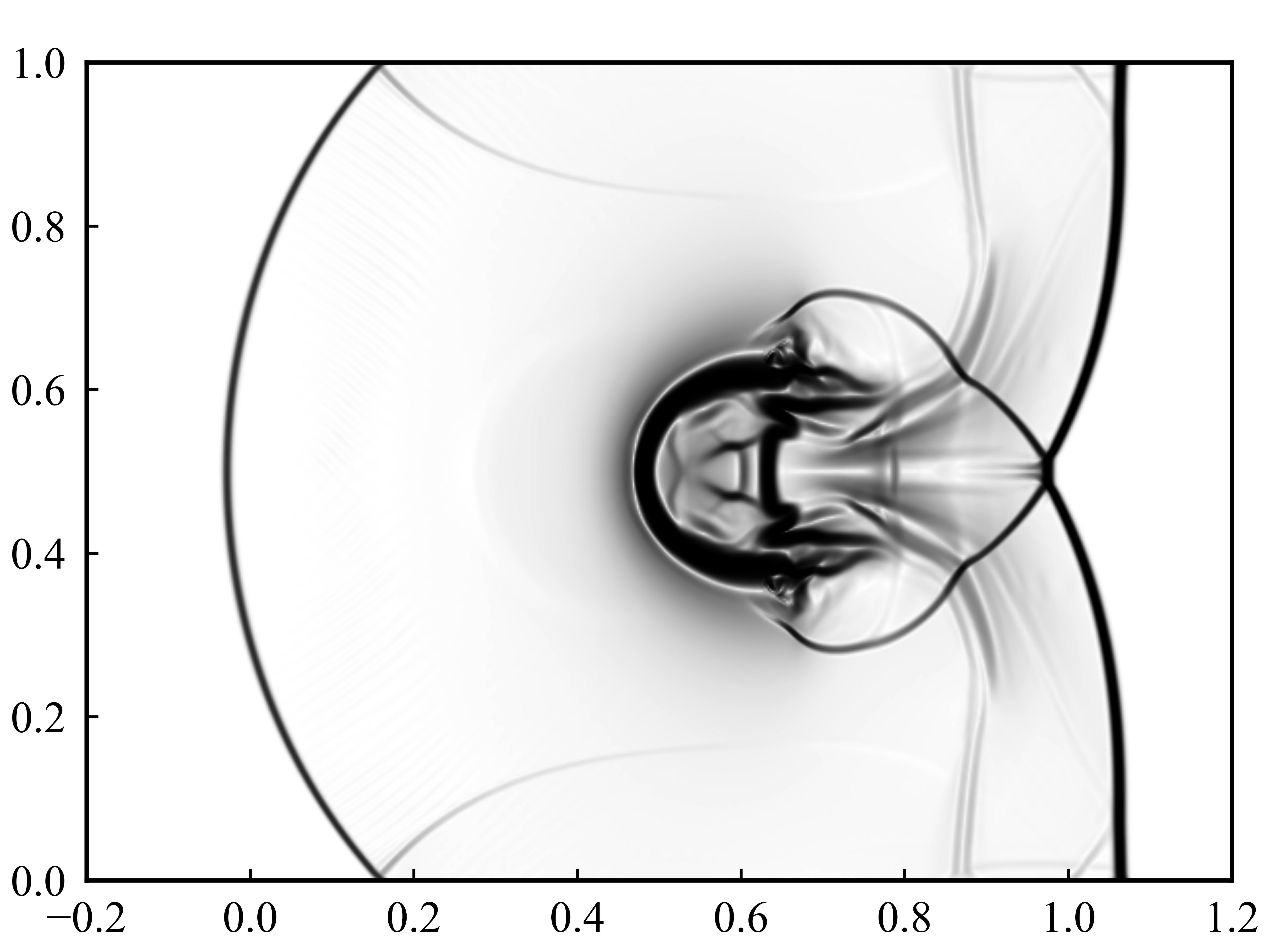}
		\end{subfigure}
		\hfill
		\begin{subfigure}{0.48\textwidth}
			\includegraphics[width=\textwidth]{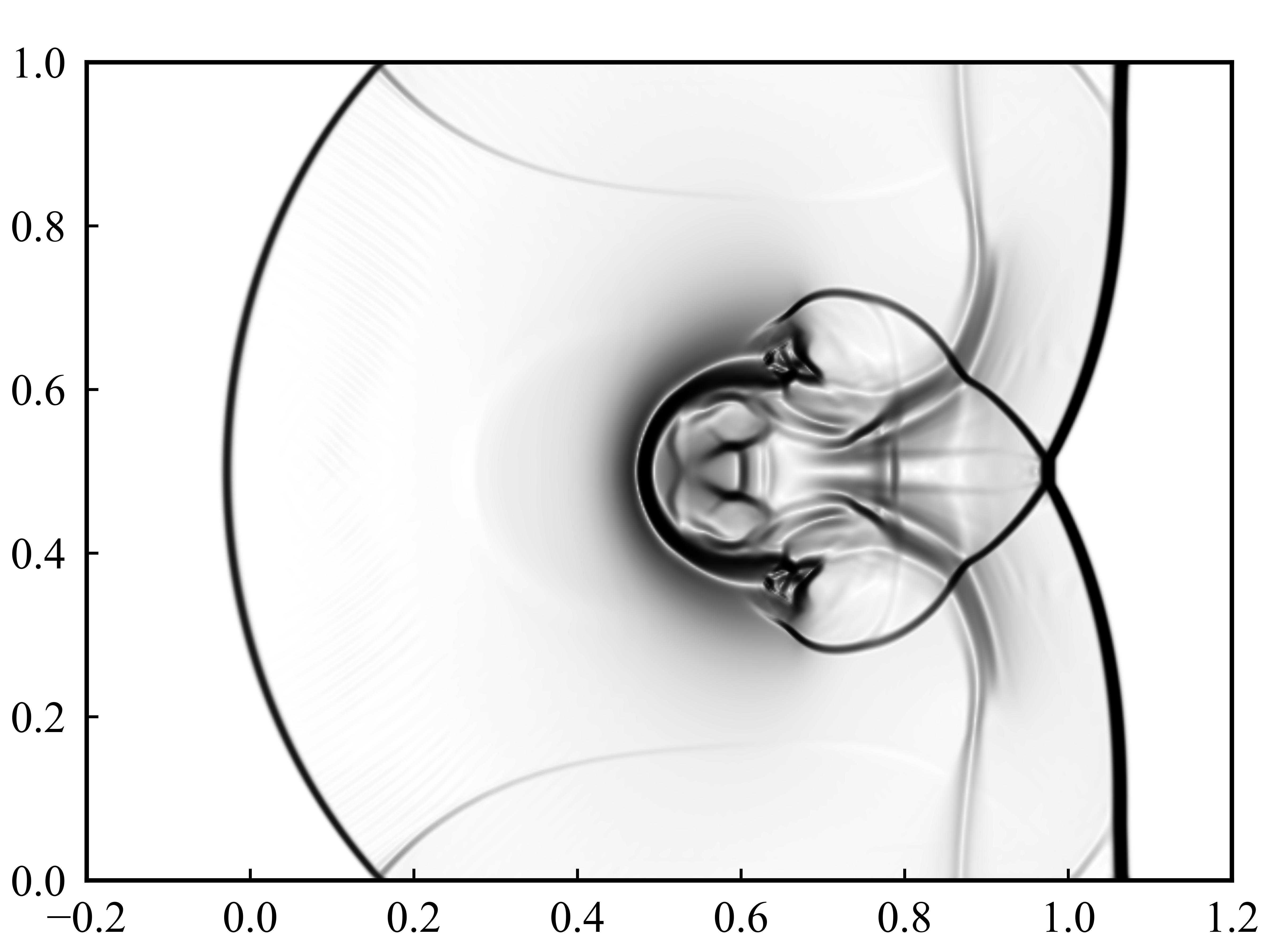}
		\end{subfigure}
		
		\begin{subfigure}{0.48\textwidth}
			\includegraphics[width=\textwidth]{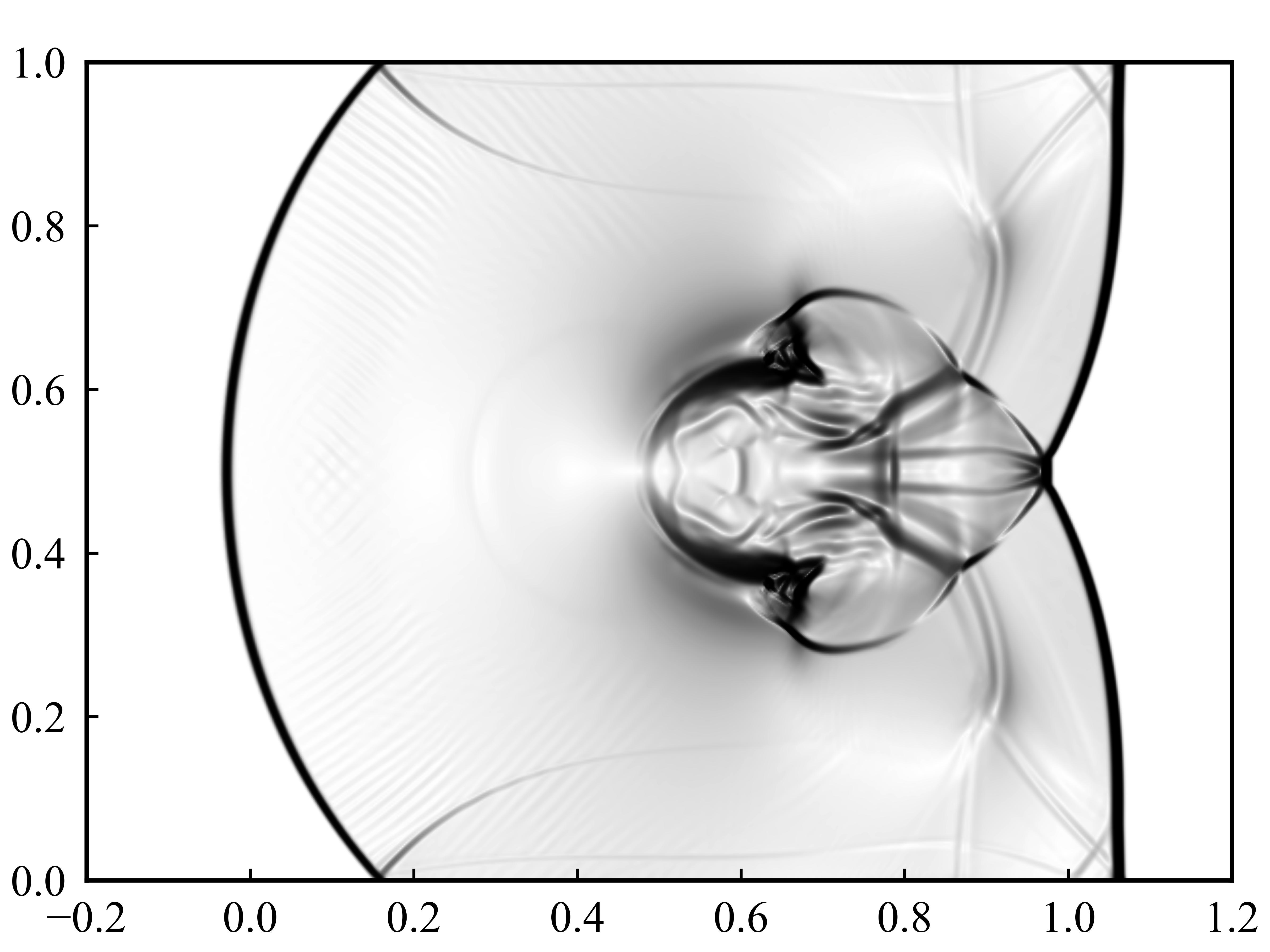}
		\end{subfigure}
		\hfill
		\begin{subfigure}{0.48\textwidth}
			\includegraphics[width=\textwidth]{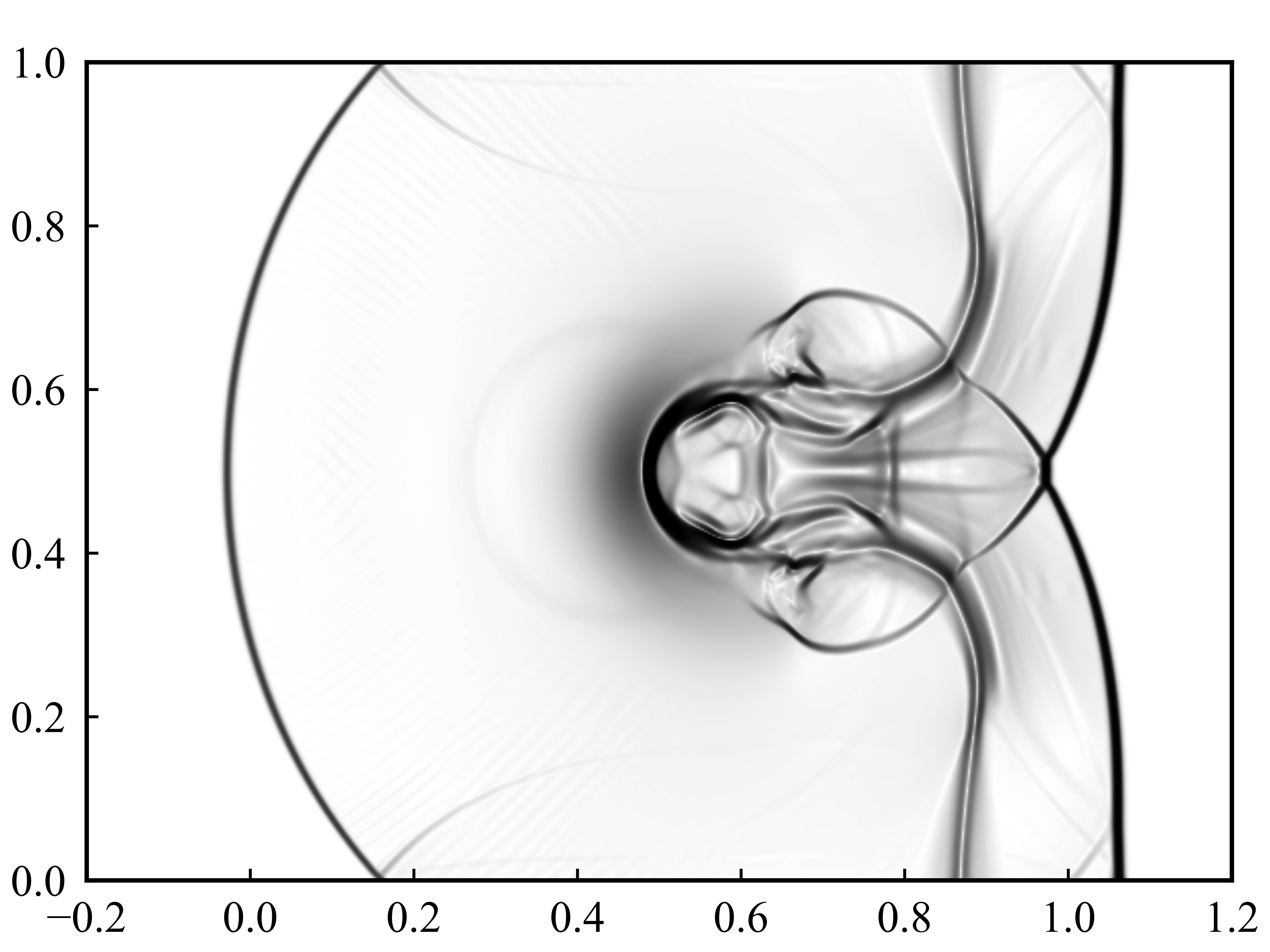}
		\end{subfigure}

		\caption{Schlieren images of the rest-mass density logarithm (top-left), thermal pressure logarithm (top-right), Lorentz factor (bottom-left), and magnitude of magnetic field (bottom-right) for \Cref{Ex:ShockCloud}.
		}
		\label{fig:Ex-ShockCloud}
	\end{figure} 
\end{expl}

\begin{expl}[Blast Problems]\label{Ex:Blast}
RMHD blast wave problems are widely used to test the robustness of numerical schemes, as nonphysical solutions can easily be produced in numerical simulations. Our setup is similar to those described in \cite{WuTangM3AS,WuShu2020NumMath}, except for that the IP-EOS \eqref{EOS:IPEOS} is used. The computational domain is $\Omega=[-6,6]^2$. Initially, $\Omega$ is filled with stationary fluid and divided into three regions: the explosion region ($r < 0.8$), the low-density and low-pressure region ($r > 1$), and the middle region ($0.8 \leq r \leq 1$), where $r=\sqrt{x^2+y^2}$ represents the distance from the center of the domain. The initial conditions are specified as follows:
	\begin{equation*}
		 \bm{B} = \left(B_a, 0, 0\right), \qquad \bm{v}={\bf 0},\qquad 
		(\rho,p) = 
		\begin{cases}
			(10^{-2}, 1), & r < 0.8, \\
			(\rho_a, p_a), & 0.8 \leq r \leq 1.0, \\
			(10^{-4}, 5\times10^{-4}), & \text{otherwise}.
		\end{cases}
	\end{equation*}
	where $B_a$ is constant, and $\rho_a$ and $p_a$ are obtained by linear interpolation.  Specifically, $\rho_a = 10^{-4} + \rho_b \frac{1-r}{0.2} $ and $p_a = 5\times10^{-4} + p_b \frac{1-r}{0.2}$ with  $\rho_b=10^{-2}-10^{-4}$ and $p_b = 1-5\times10^{-4}$.
 In our simulations, we consider three different values of $B_a$, namely, $0.1$, $0.5$, and $2000$. These correspond to a moderate magnetic field, a relatively strong magnetic field, and an extremely strong magnetic field, respectively. The domain $\Omega$ is partitioned into $400 \times 400$ uniform cells, and outflow boundary conditions are applied to all boundaries. \Cref{fig:Ex-Blast} displays the rest-mass density logarithm, thermal pressure, and magnitude of the magnetic field at $t=4$ obtained by the proposed BP locally DF CDG method. The complicated wave patterns are well captured, and our results agree with those reported in \cite{WuTangM3AS,WuShu2020NumMath} by the non-central DG schemes. \Cref{fig:BL_Error} illustrates the temporal evolution of the global relative divergence error $\varepsilon_{\rm div}$ for three different values of $B_a$. Notably, for all three cases, our method keeps the global relative divergence errors at quite small levels. Again, it is observed that the BP limiter is essential for ensuring the BP condition \eqref{CDG2D:QuaCond}. Without it, the CDG code would produce nonphysical numerical solutions, causing the simulation to fail.

	\begin{figure}[!htb]
		\centering
		\begin{subfigure}{0.32\textwidth}
			\includegraphics[width=\textwidth]{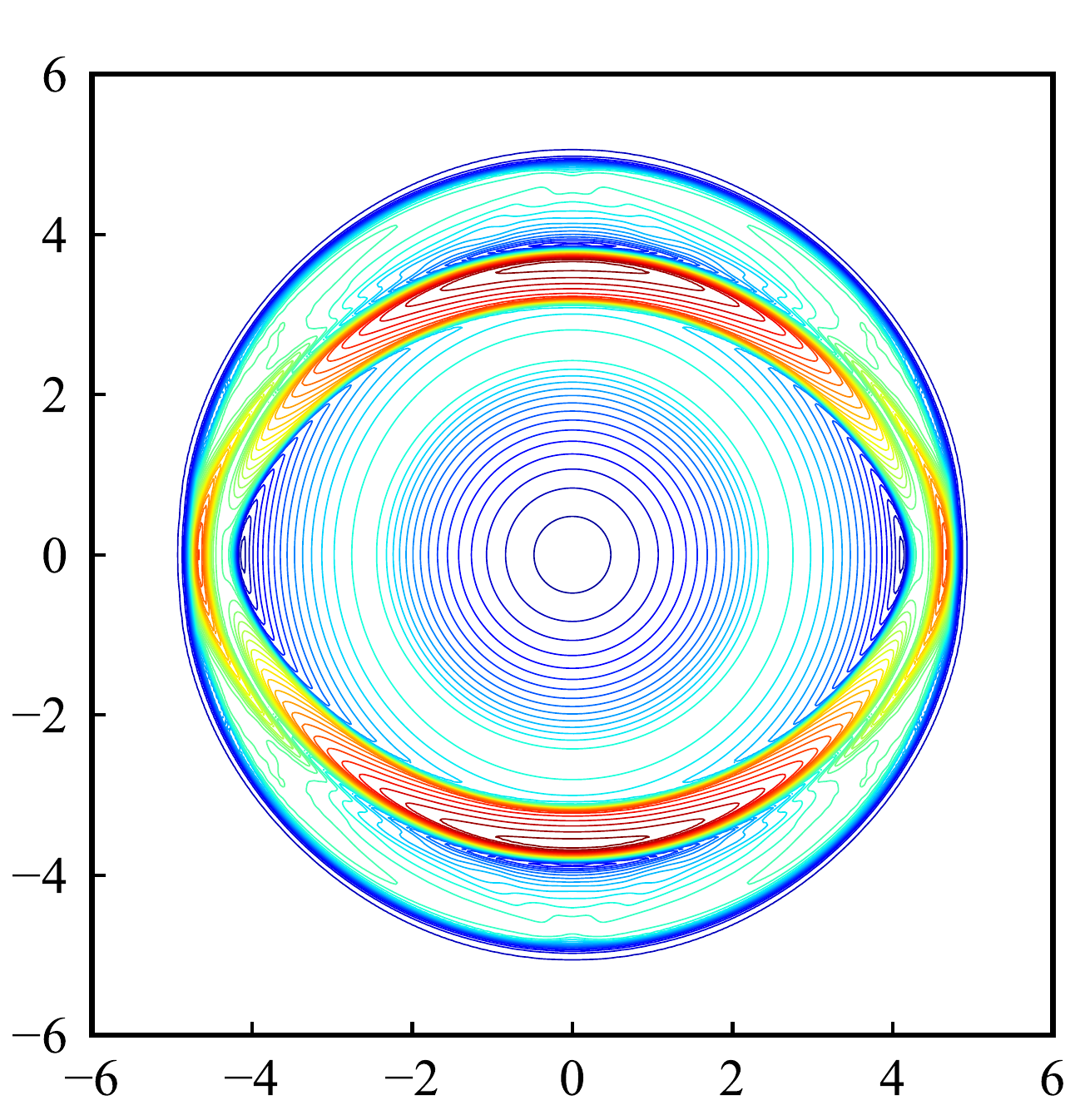}
		\end{subfigure}
		\hfill
		\begin{subfigure}{0.32\textwidth}
			\includegraphics[width=\textwidth]{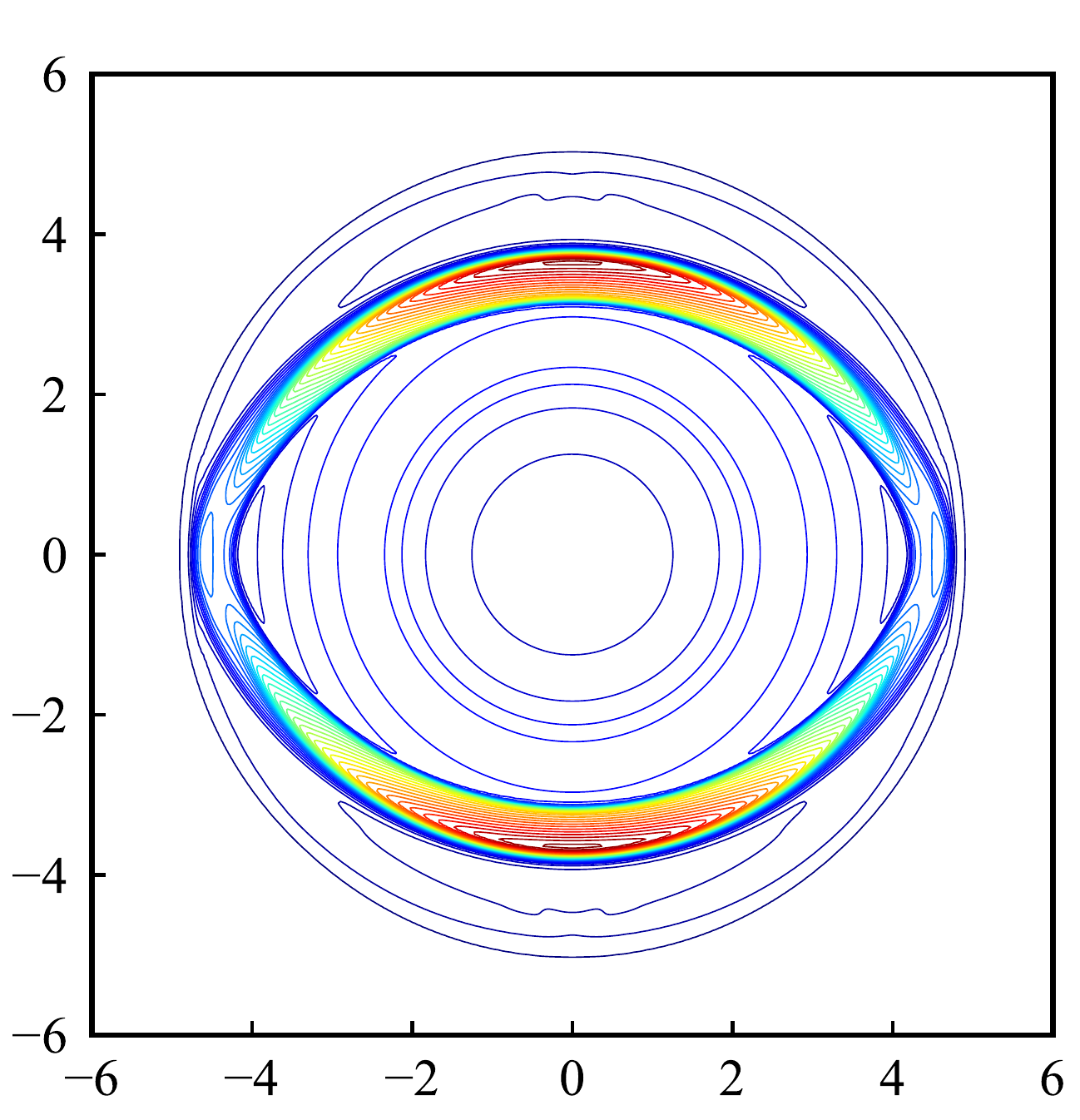}
		\end{subfigure}
		\hfill
		\begin{subfigure}{0.32\textwidth}
			\includegraphics[width=\textwidth]{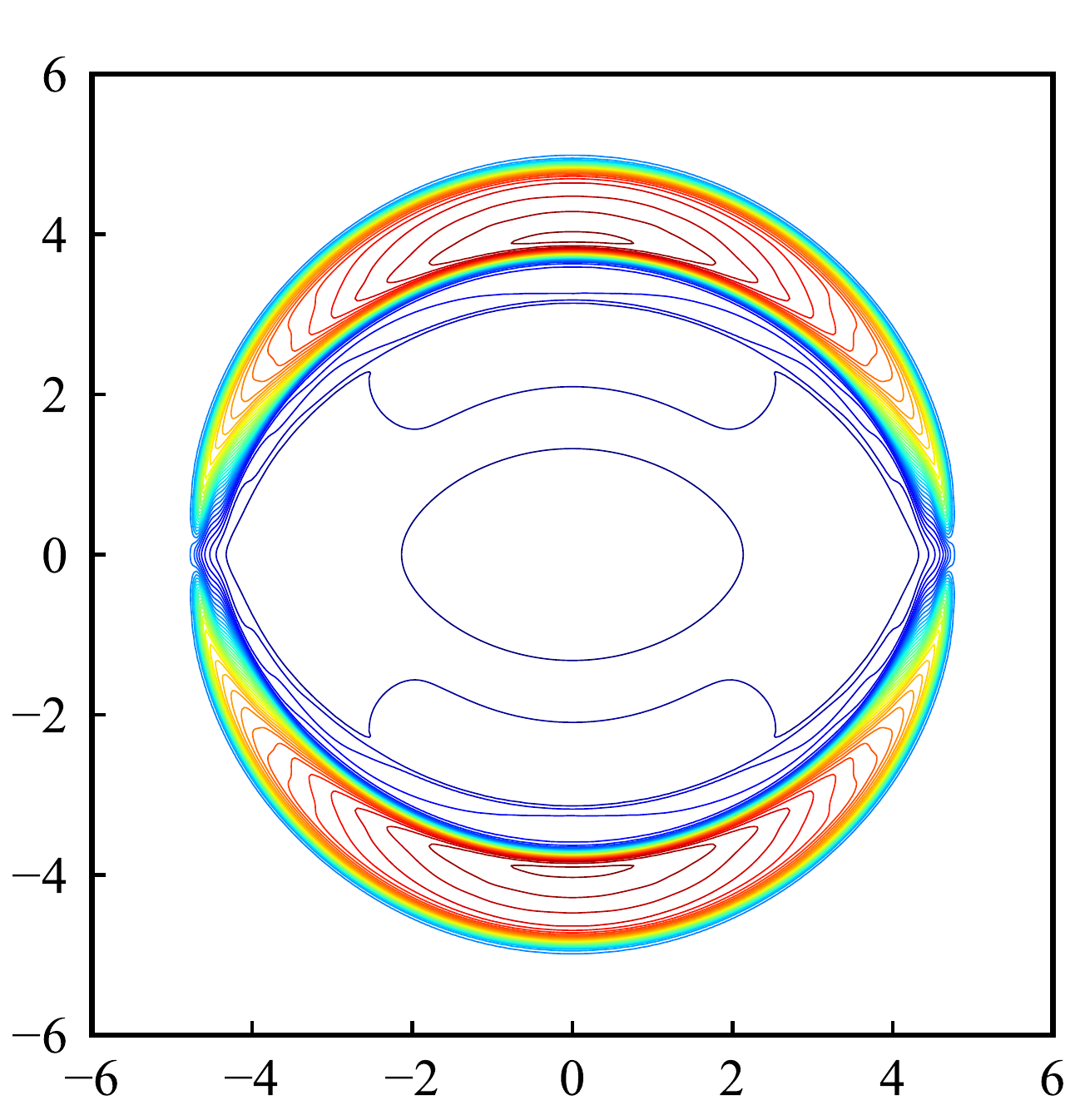}
		\end{subfigure}
		
		\begin{subfigure}{0.32\textwidth}
			\includegraphics[width=\textwidth]{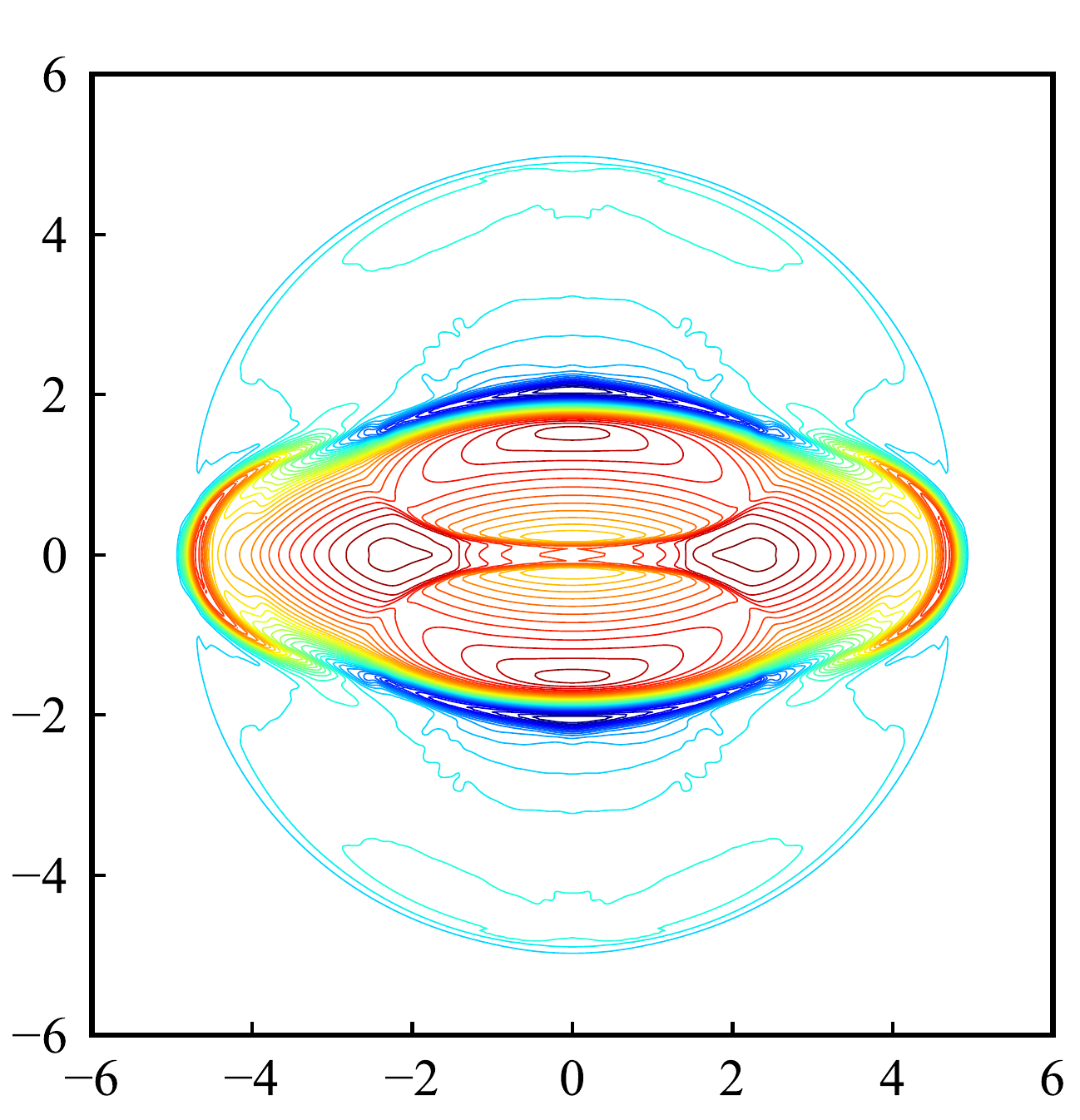}
		\end{subfigure}
		\hfill
		\begin{subfigure}{0.32\textwidth}
			\includegraphics[width=\textwidth]{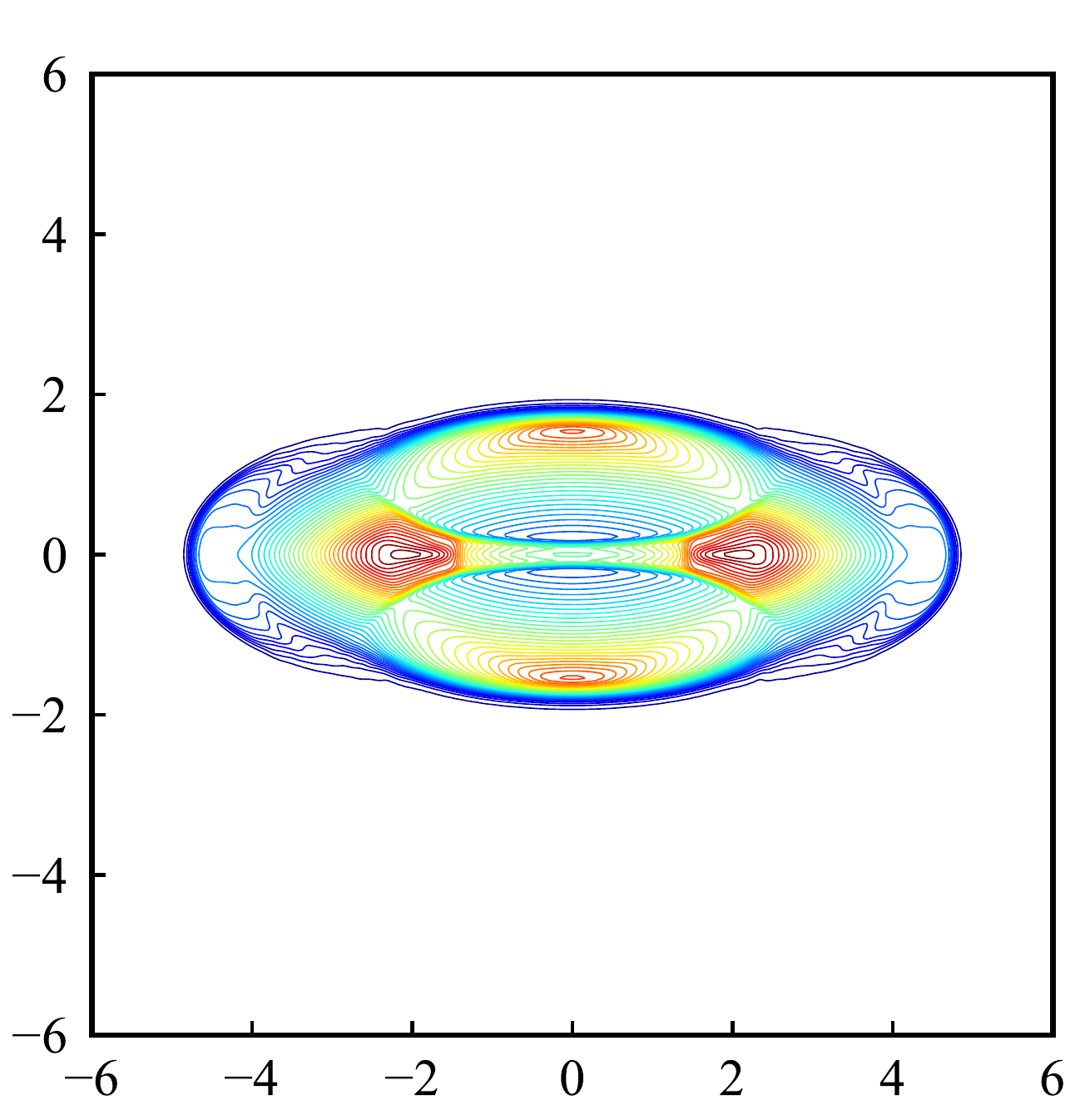}
		\end{subfigure}
		\hfill
		\begin{subfigure}{0.32\textwidth}
			\includegraphics[width=\textwidth]{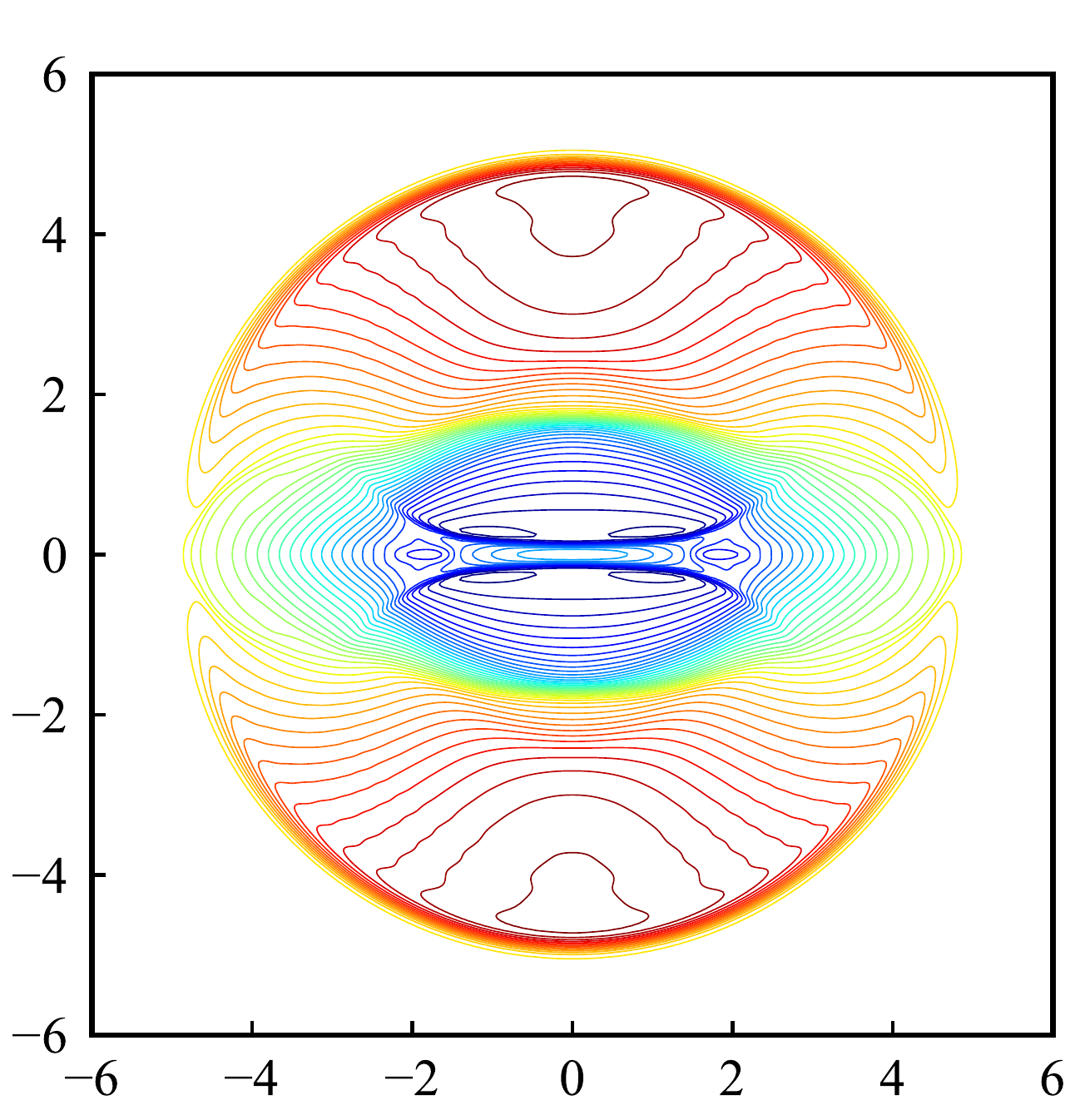}
		\end{subfigure}
		
		\begin{subfigure}{0.32\textwidth}
			\includegraphics[width=\textwidth]{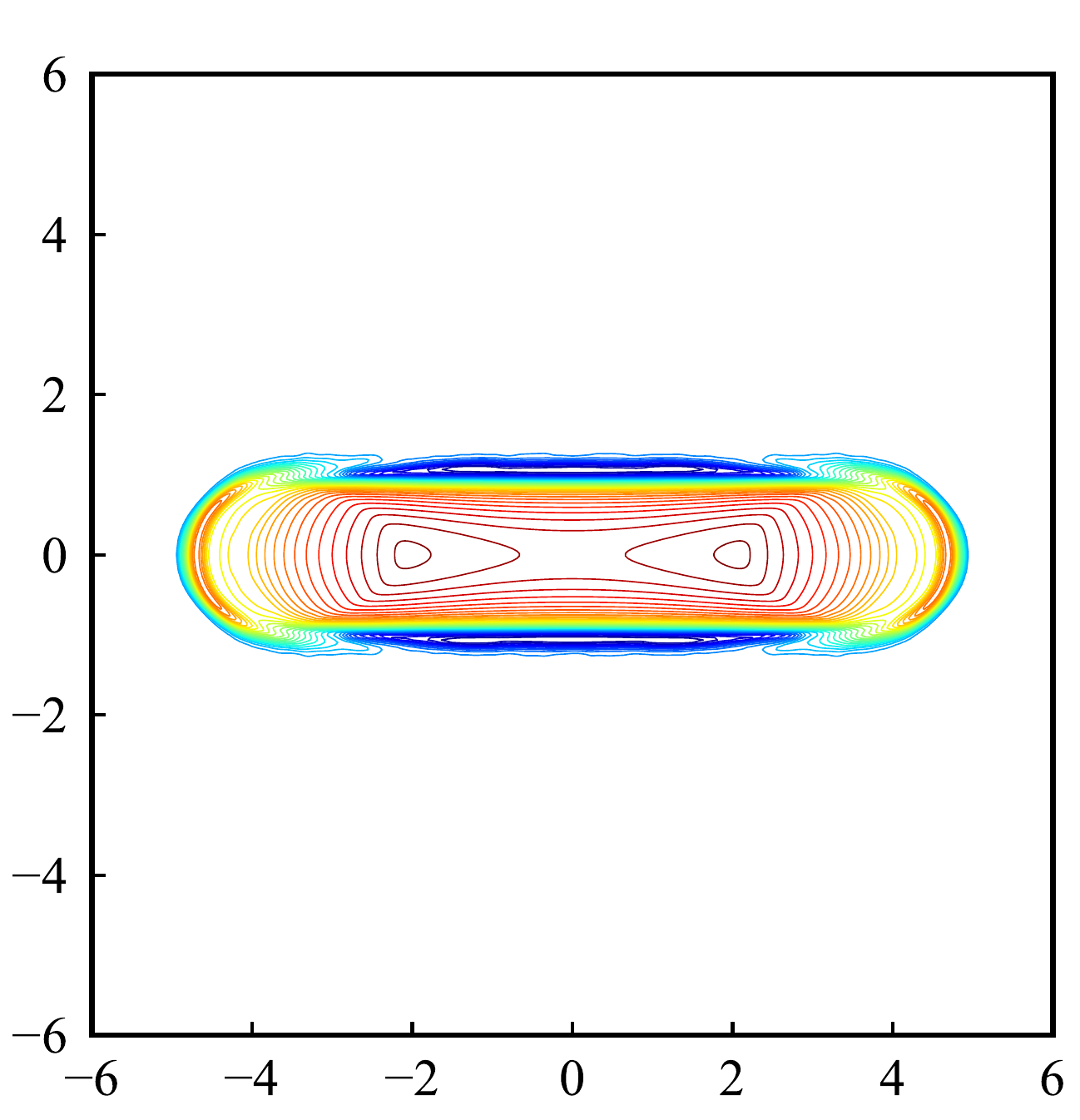}
		\end{subfigure}
		\hfill
		\begin{subfigure}{0.32\textwidth}
			\includegraphics[width=\textwidth]{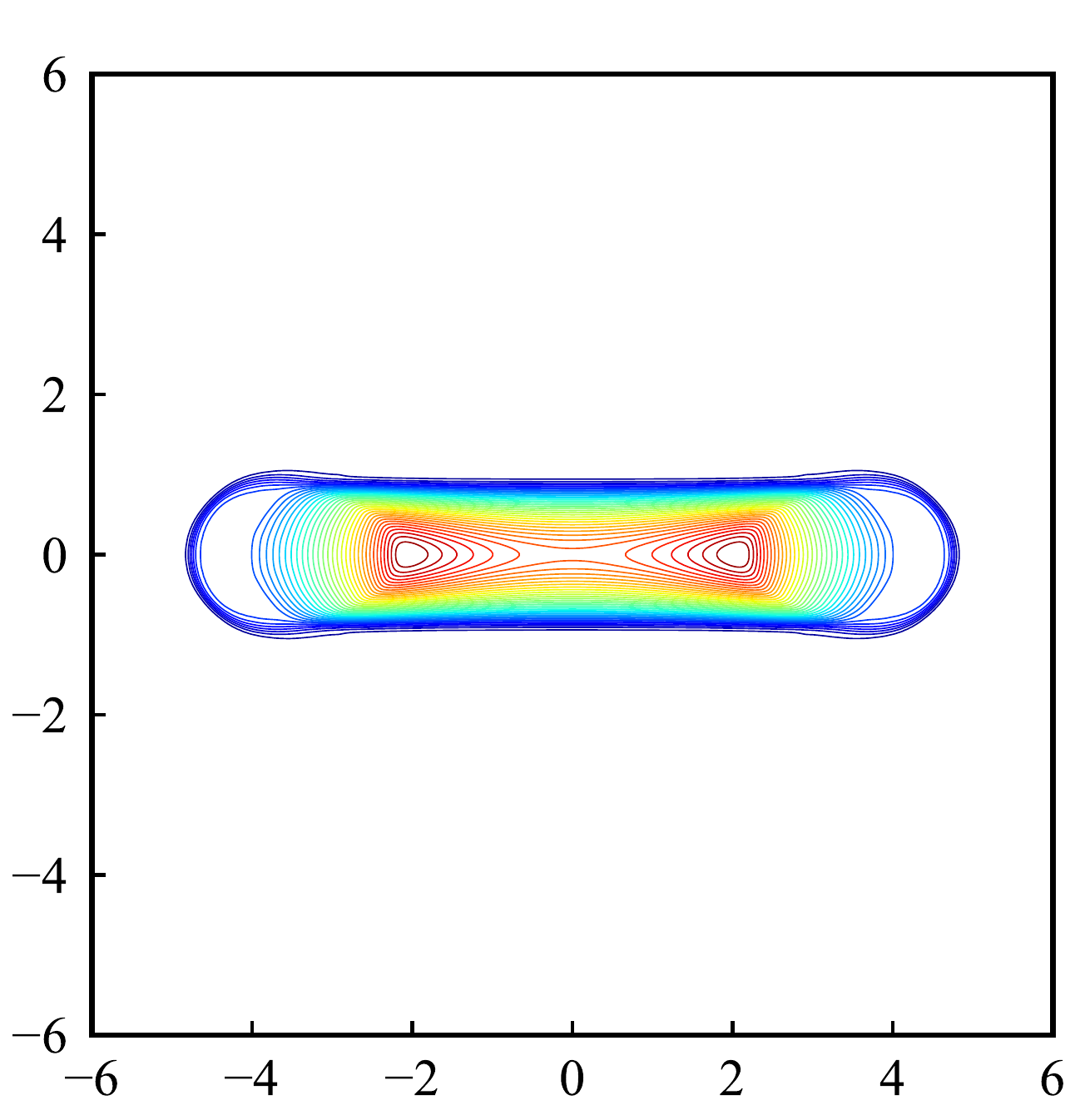}
		\end{subfigure}
		\hfill
		\begin{subfigure}{0.32\textwidth}
			\includegraphics[width=\textwidth]{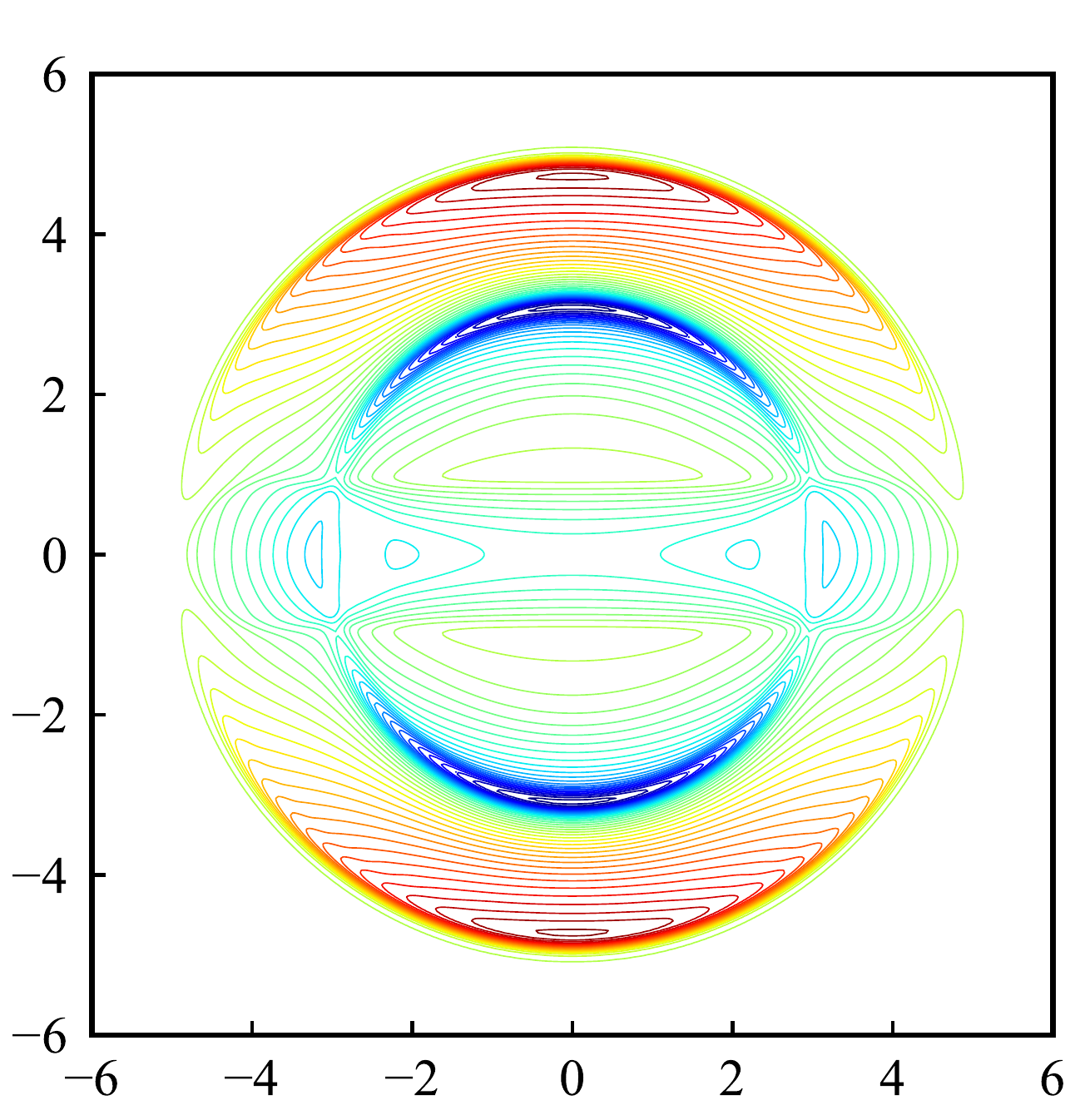}
		\end{subfigure}
		
		\caption{Contour plots of the rest-mass density logarithm (left), thermal pressure (middle), and magnitude of magnetic field (right) for \Cref{Ex:Blast} with the IP-EOS \eqref{EOS:IPEOS}. 
			Top: $B_a = 0.1$; 
			middle: $B_a = 0.5$; 
			bottom: $B_a=2000$.
		}
		\label{fig:Ex-Blast}
	\end{figure}

\begin{figure}[!htb]
		\centering
		\begin{subfigure}{0.32\textwidth}
			\includegraphics[width=\textwidth]{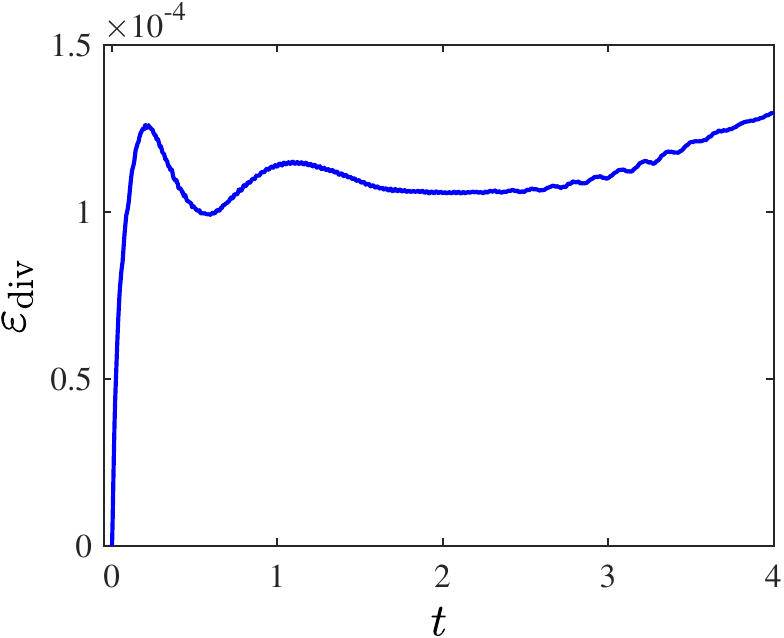}
		\end{subfigure}
		\hfill
		\begin{subfigure}{0.32\textwidth}
			\includegraphics[width=\textwidth]{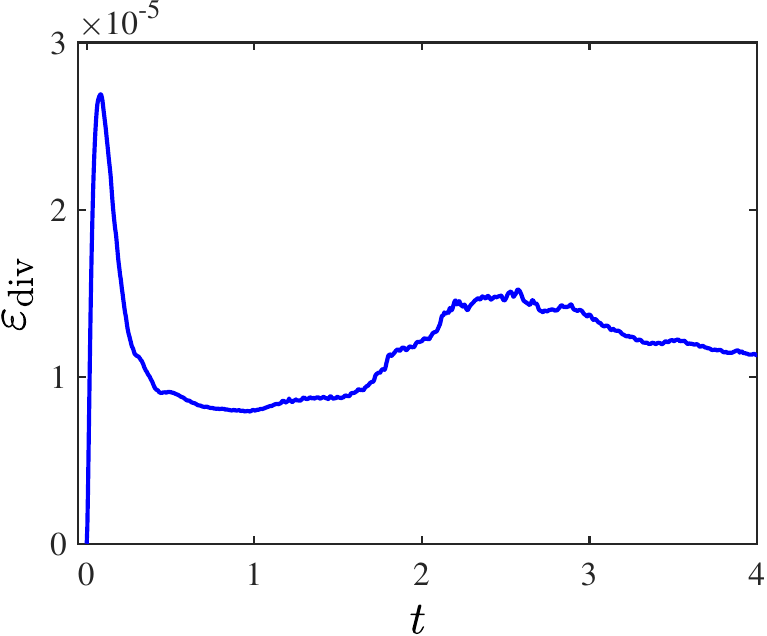}
		\end{subfigure}
		\hfill
		\begin{subfigure}{0.32\textwidth}
			\includegraphics[width=\textwidth]{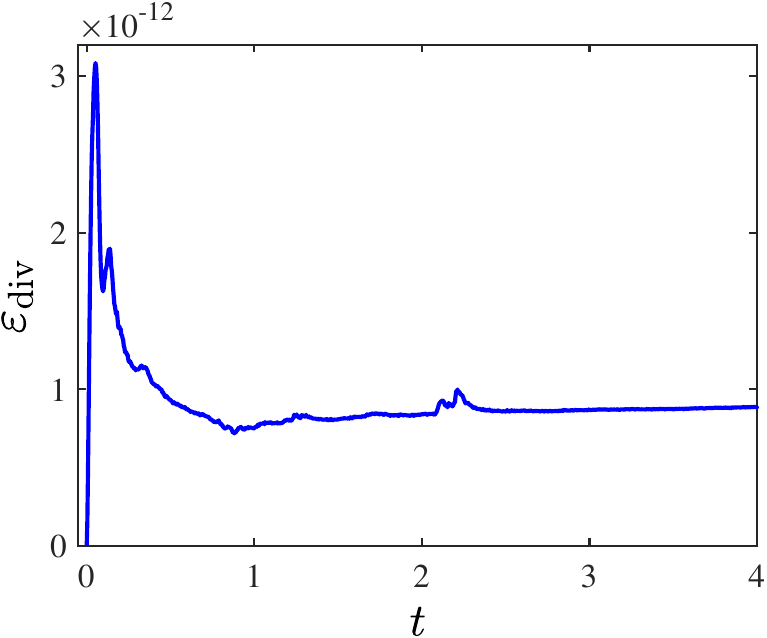}
		\end{subfigure}
		
		\caption{Time evolution of the global relative divergence error $\varepsilon_{\rm div}$ for \Cref{Ex:Blast} with $B_a=0.1$ (left), $B_a=0.5$ (middle), and $B_a=2000$ (right).
		}
		\label{fig:BL_Error}
	\end{figure} 
\end{expl}

\begin{expl}[Astrophysical Jets]\label{Ex:Jet}

As our final test problem, we consider two high-speed RMHD jet flows. This problem was originally proposed and investigated in \cite{WuShu2020NumMath} with the ideal EOS. The domain $\Omega=[-12,12]\times [0,25]$ is initially filled with stationary ambient plasma having 
$(\rho, \bm{B}, p) = (1, 0, \sqrt{2000 P_a}, 0, P_a)$, where $P_a= 2.35362407217\times 10^{-5}$. A jet with a Mach number of $M_b=50$ is injected into the domain $\Omega$ from the bottom boundary in the $y$-direction, initially located at $x\in [-0.5,0.5]$ and $y=0$. Due to symmetry, we simulate only half of the domain $\Omega$, specifically $[0,12]\times [0,25]$, which is divided into $240\times 500$ uniform cells. The reflective boundary condition is applied to $x=0$, while outflow conditions are imposed on all other boundaries, except for the region $x\in [0,0.5]$ on the bottom boundary where the inflow conditions of
$(\rho, \bm{v}, \bm{B}, p) = (0.1, 0, 0.99, 0, 0, \sqrt{2000 P_a}, 0, P_a)$
are used. This problem is highly challenging due to the large Lorentz factor $W\approx 7.09$, the high relativistic Mach number $M_r \approx 354.37$, and the low plasma-beta of $0.001$. The CDG code would quickly break down due to nonphysical numerical solutions if we turn off the BP limiter, remove the discretized part of the symmetrization source terms from our schemes, or use the conventional (non-DF) CDG discretization. This confirms the importance of these three ingredients in our schemes for ensuring the BP property.

%
%
	
	\begin{figure}[!htb]
		\centering	
		\begin{subfigure}{0.32\textwidth}
			\includegraphics[width=\textwidth]{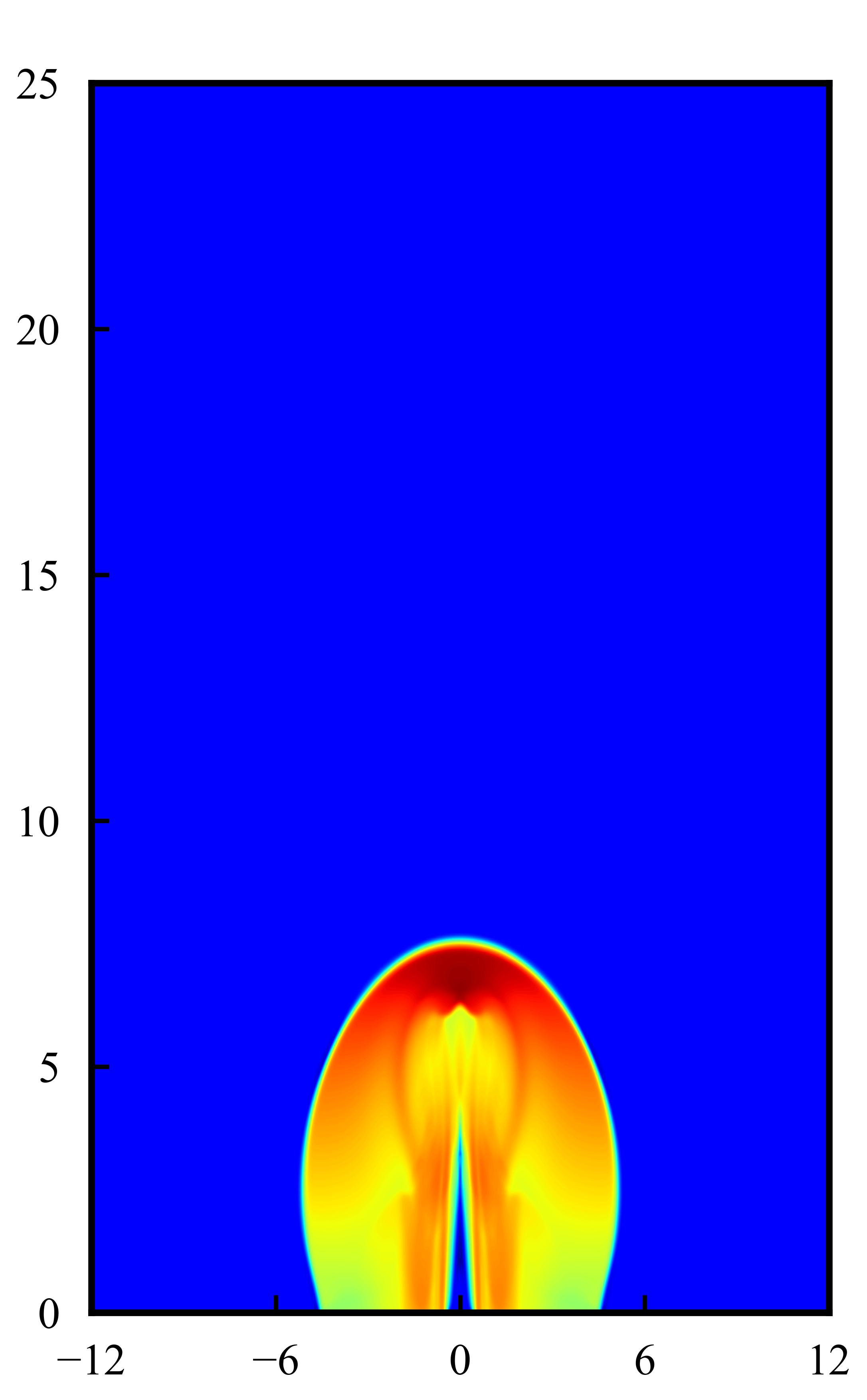}
		\end{subfigure}
		\hfill
		\begin{subfigure}{0.32\textwidth}
			\includegraphics[width=\textwidth]{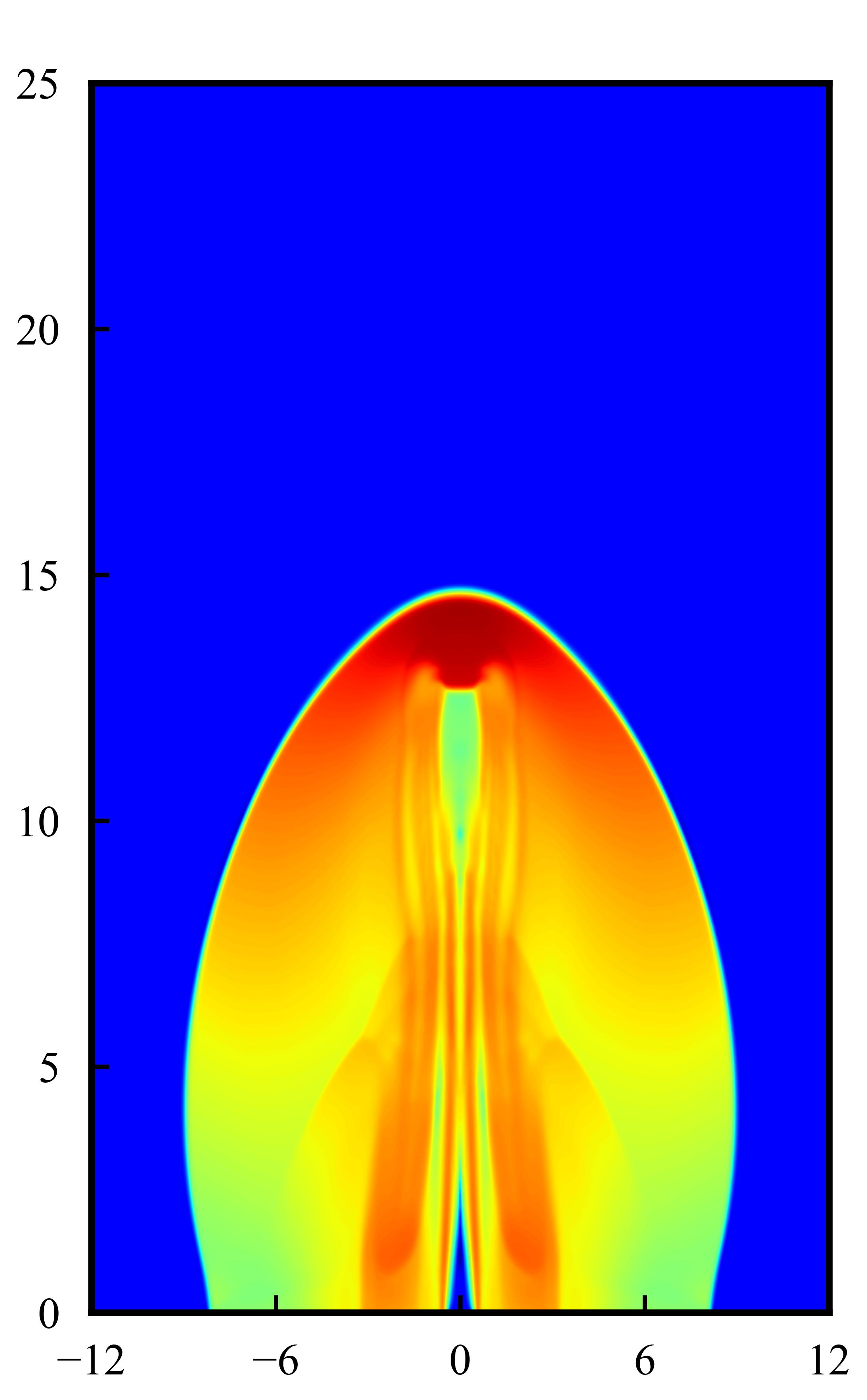}
		\end{subfigure}
		\hfill
		\begin{subfigure}{0.32\textwidth}
			\includegraphics[width=\textwidth]{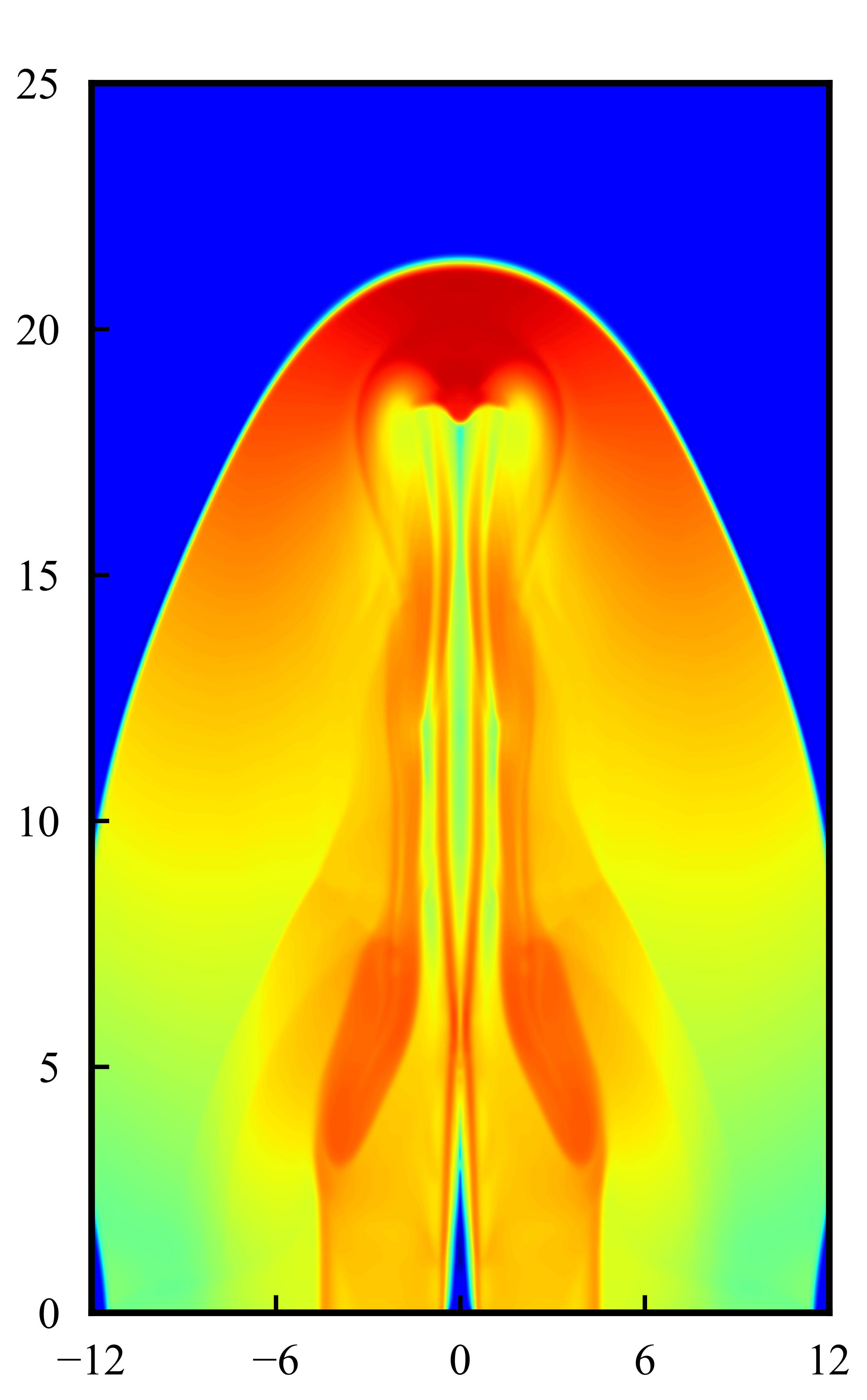}
		\end{subfigure}
		
		\begin{subfigure}{0.32\textwidth}
			\includegraphics[width=\textwidth]{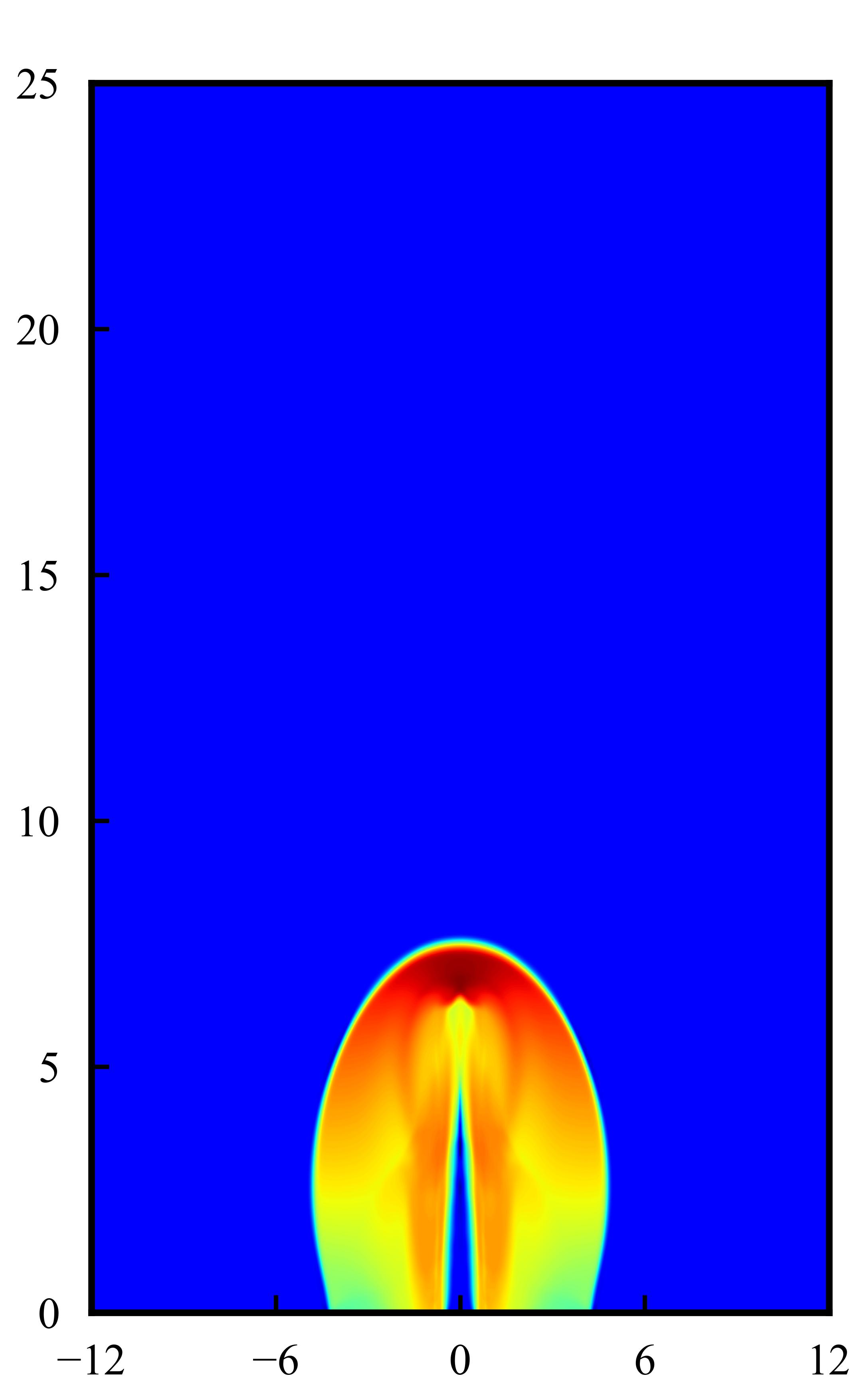}
		\end{subfigure}
		\hfill
		\begin{subfigure}{0.32\textwidth}
			\includegraphics[width=\textwidth]{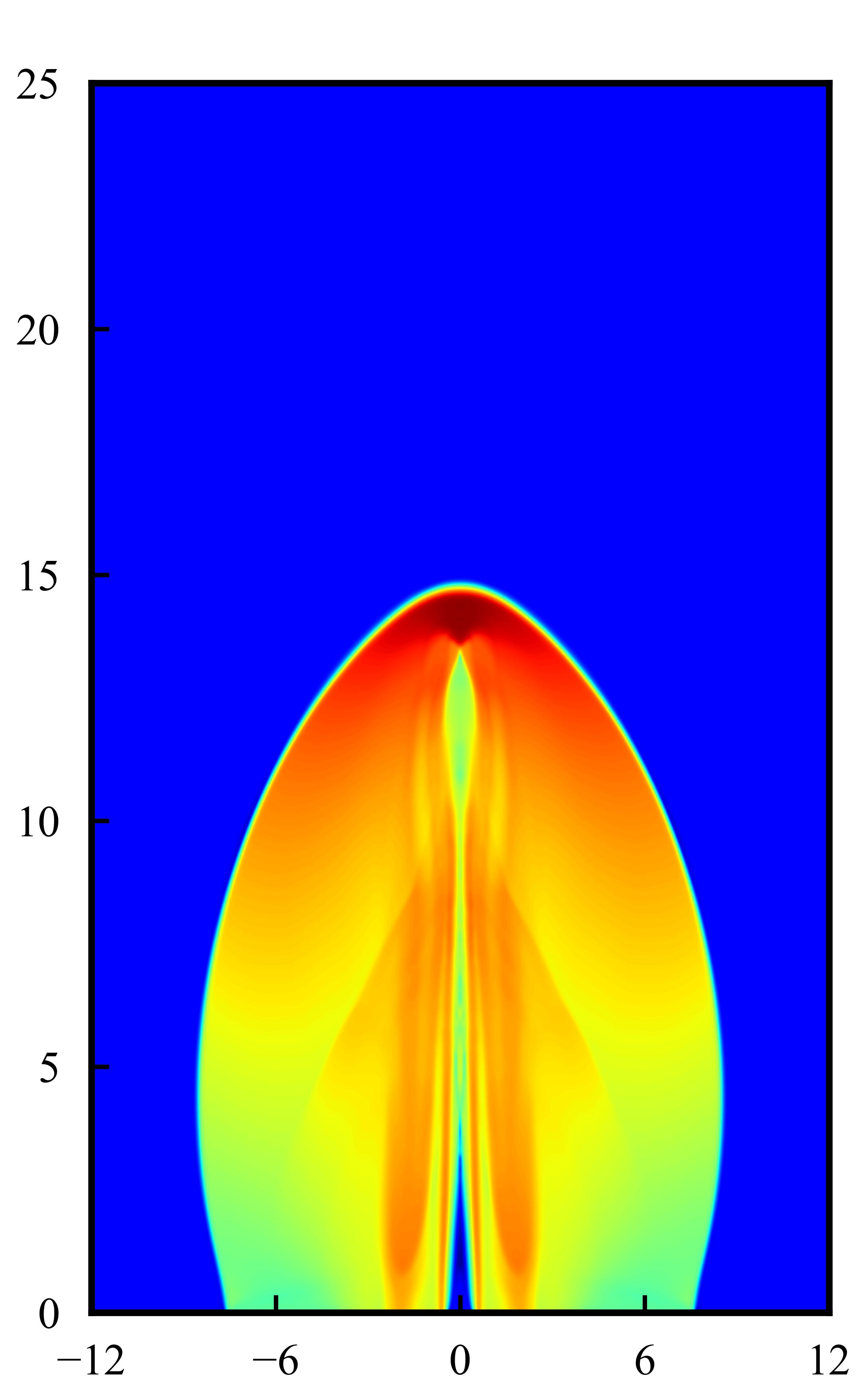}
		\end{subfigure}
		\hfill
		\begin{subfigure}{0.32\textwidth}
			\includegraphics[width=\textwidth]{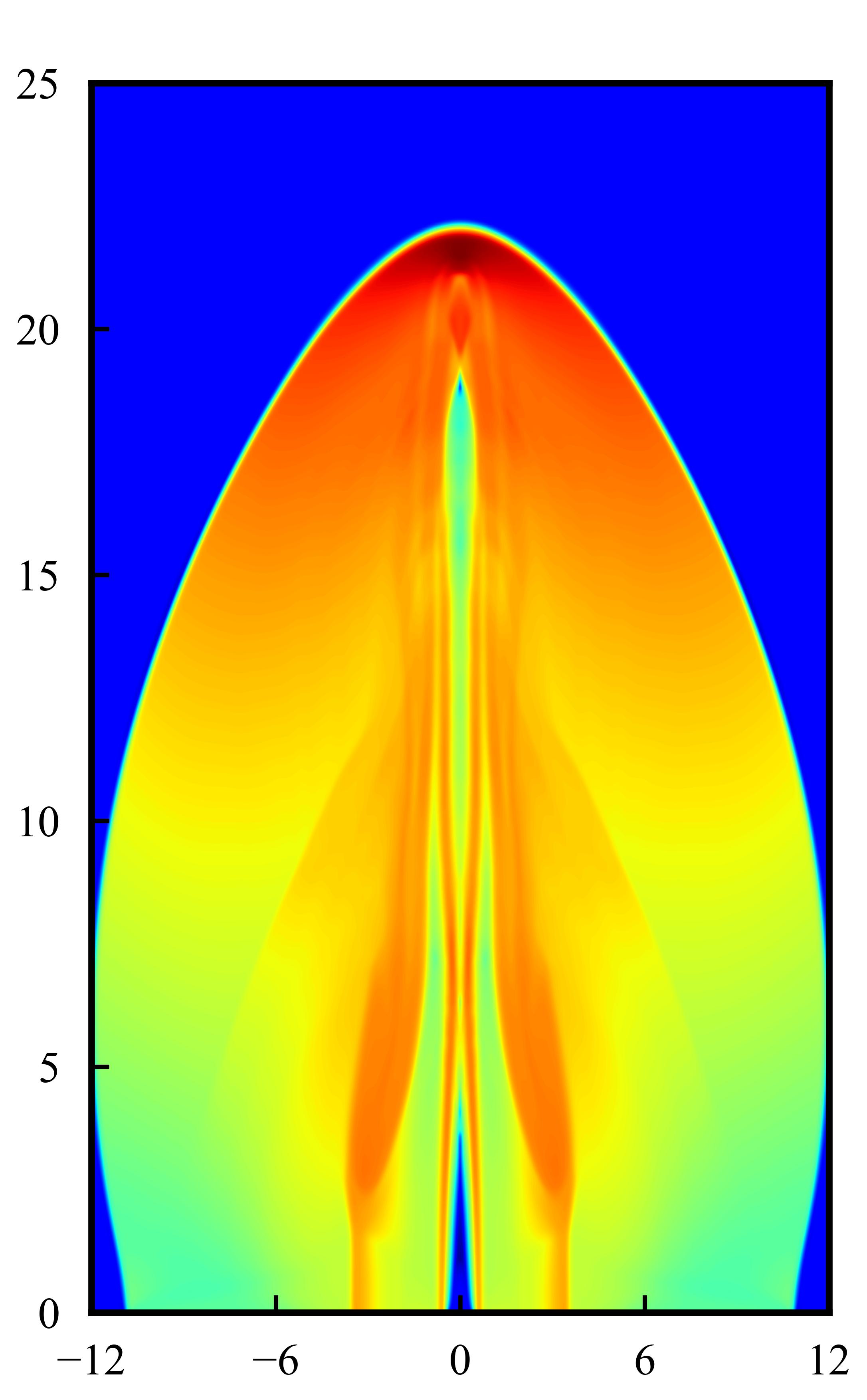}
		\end{subfigure}
		
		\caption{Pressure logarithm at $t=10, 20$, and $30$ (from left to right) for \Cref{Ex:Jet}. Top:  ideal EOS \eqref{EOS:IDEOS}; bottom: RC-EOS \eqref{EOS:RCEOS}.
		}
		\label{fig:Ex-Jet-2}
	\end{figure} 
	
	\begin{figure}[!htb]
		\centering			
		\begin{subfigure}{0.32\textwidth}
			\includegraphics[width=\textwidth]{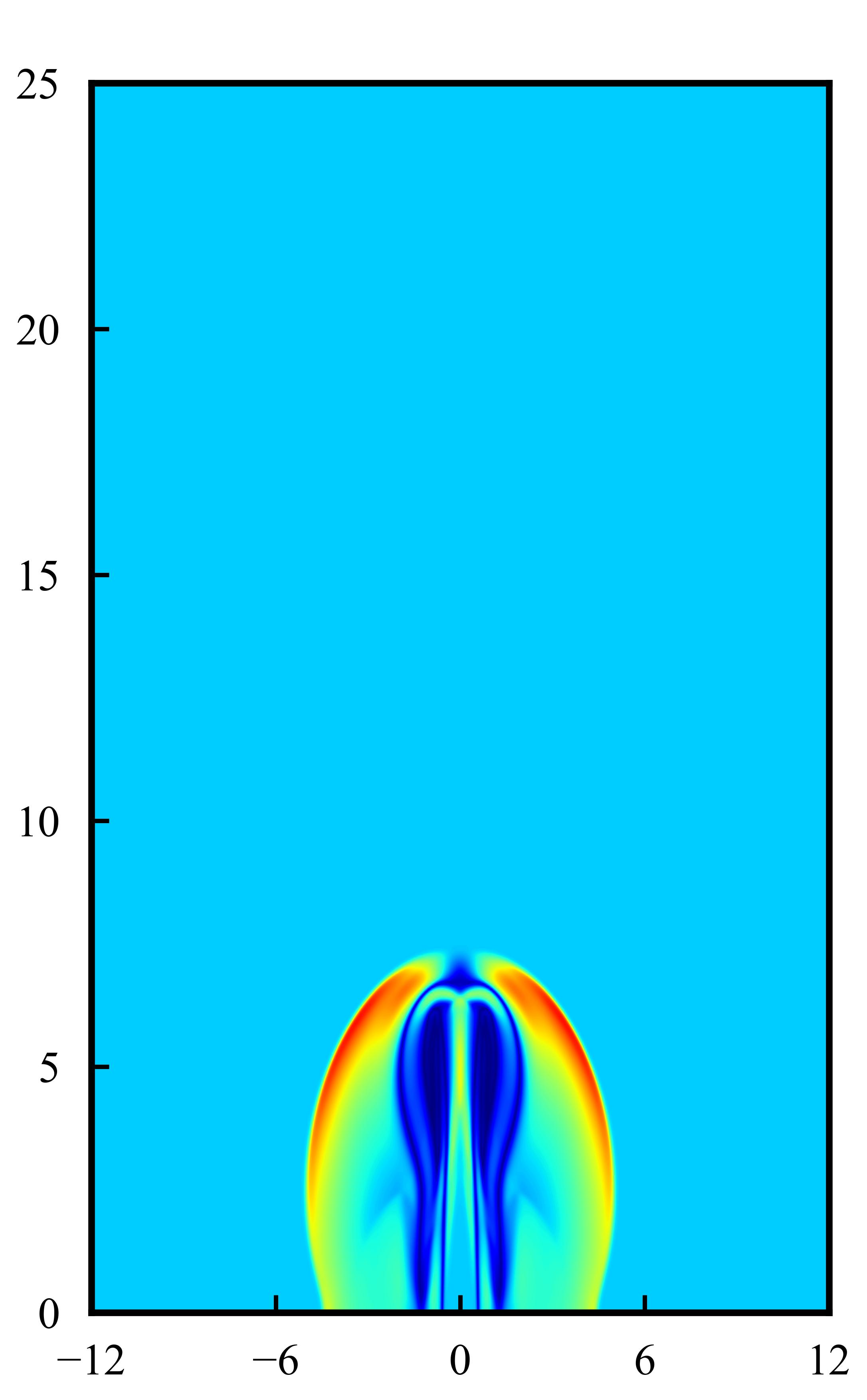}
		\end{subfigure}
		\hfill
		\begin{subfigure}{0.32\textwidth}
			\includegraphics[width=\textwidth]{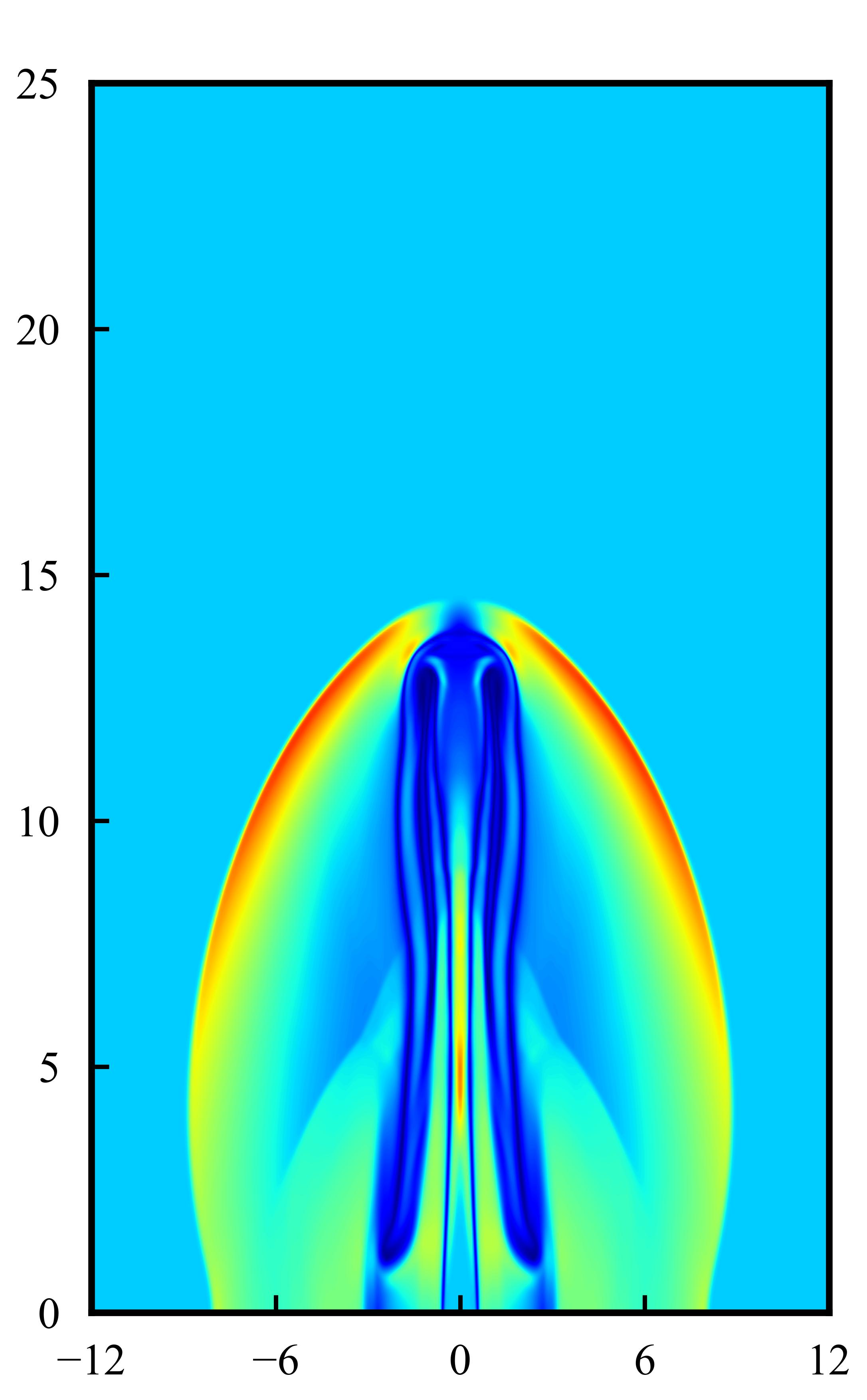}
		\end{subfigure}
		\hfill
		\begin{subfigure}{0.32\textwidth}
			\includegraphics[width=\textwidth]{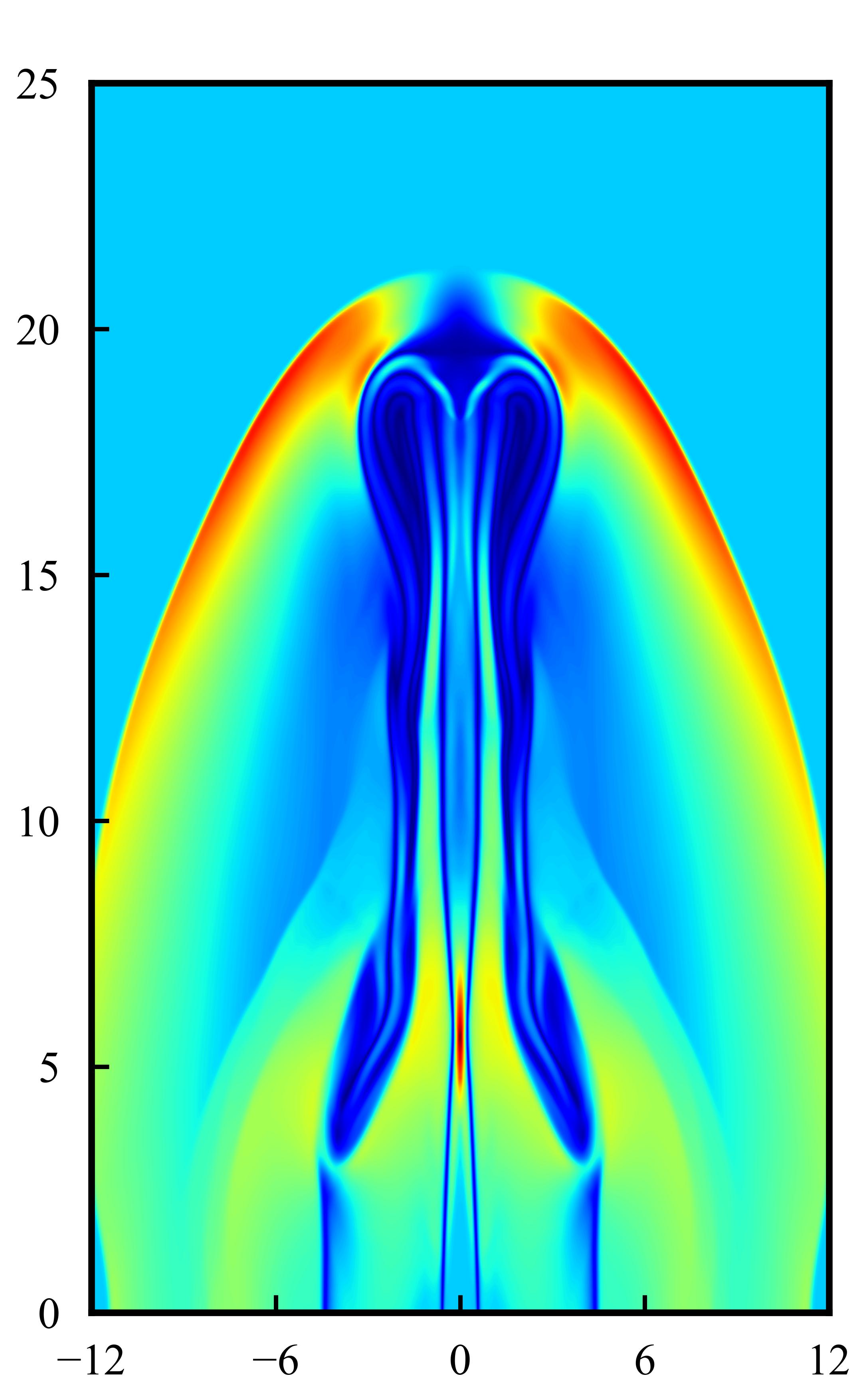}
		\end{subfigure}
		
		\begin{subfigure}{0.32\textwidth}
			\includegraphics[width=\textwidth]{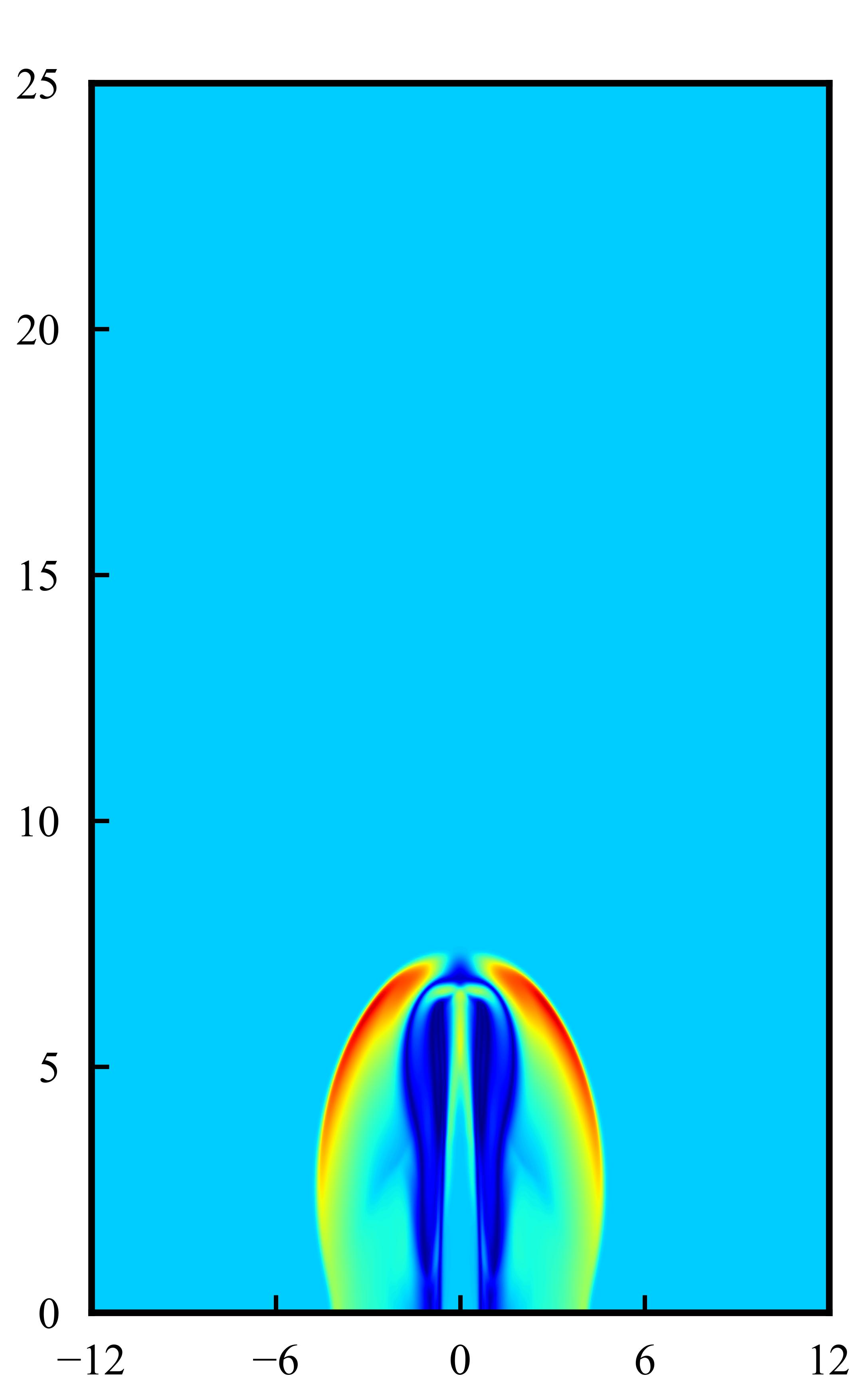}
		\end{subfigure}
		\hfill
		\begin{subfigure}{0.32\textwidth}
			\includegraphics[width=\textwidth]{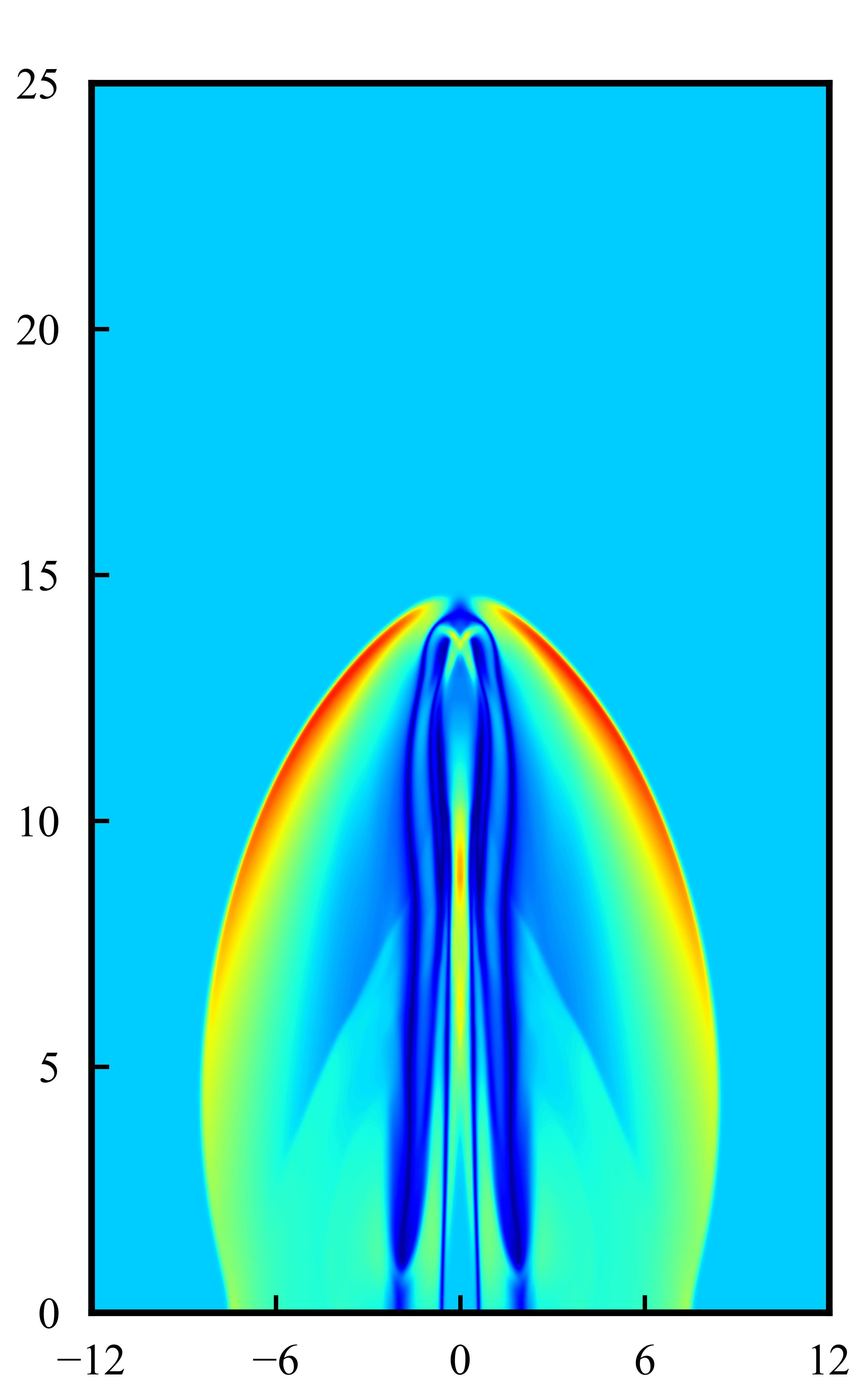}
		\end{subfigure}
		\hfill
		\begin{subfigure}{0.32\textwidth}
			\includegraphics[width=\textwidth]{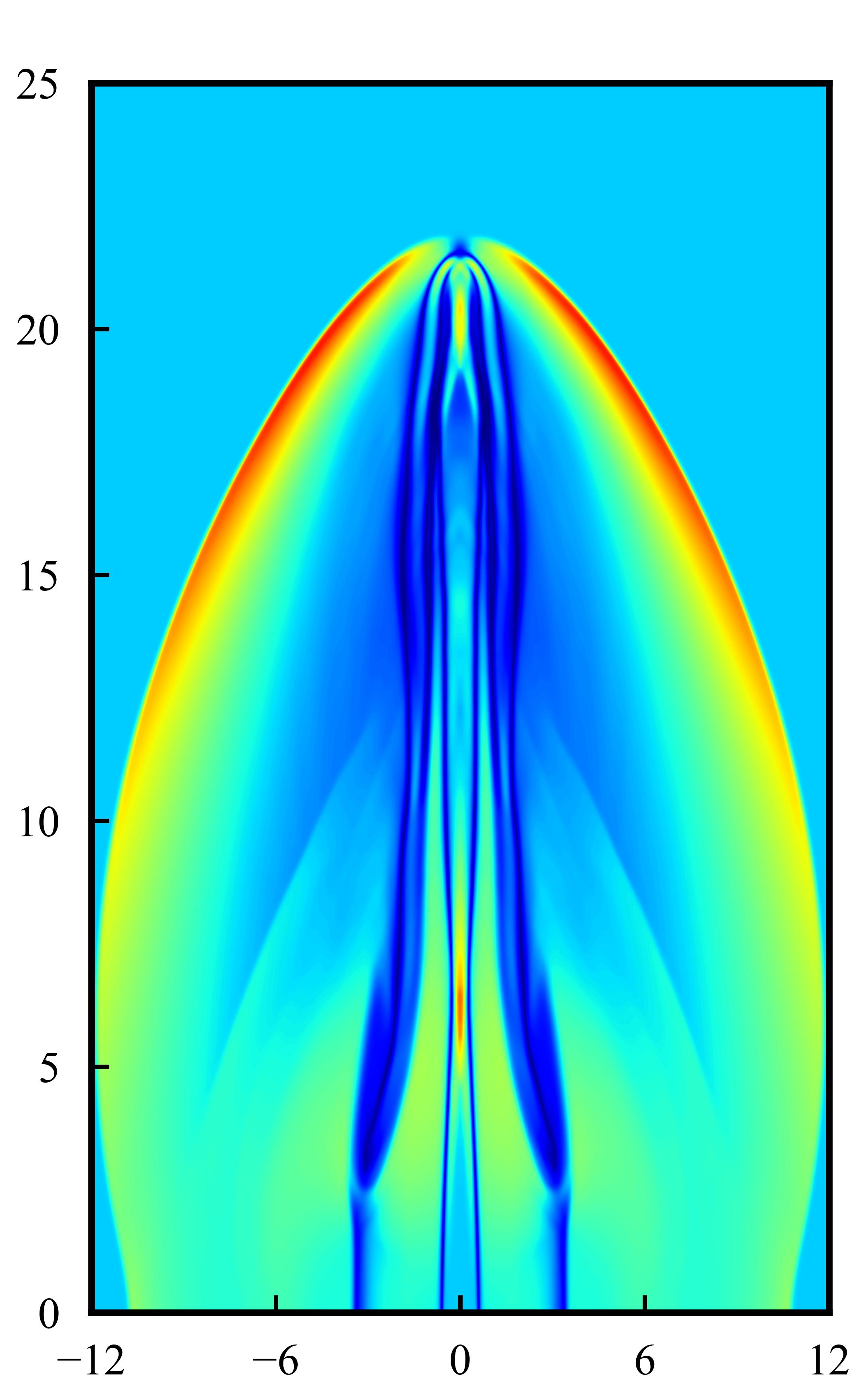}
		\end{subfigure}
		
		\caption{Magnitude of magnetic field at $t=10, 20$, and $30$ (from left to right) for \Cref{Ex:Jet}. Top: ideal EOS \eqref{EOS:IDEOS}; bottom: RC-EOS \eqref{EOS:RCEOS}.
		}
		\label{fig:Ex-Jet-3}
	\end{figure} 
	
	\begin{figure}[!htb]
		\centering
		\begin{subfigure}{0.45\textwidth}
			\includegraphics[width=\textwidth]{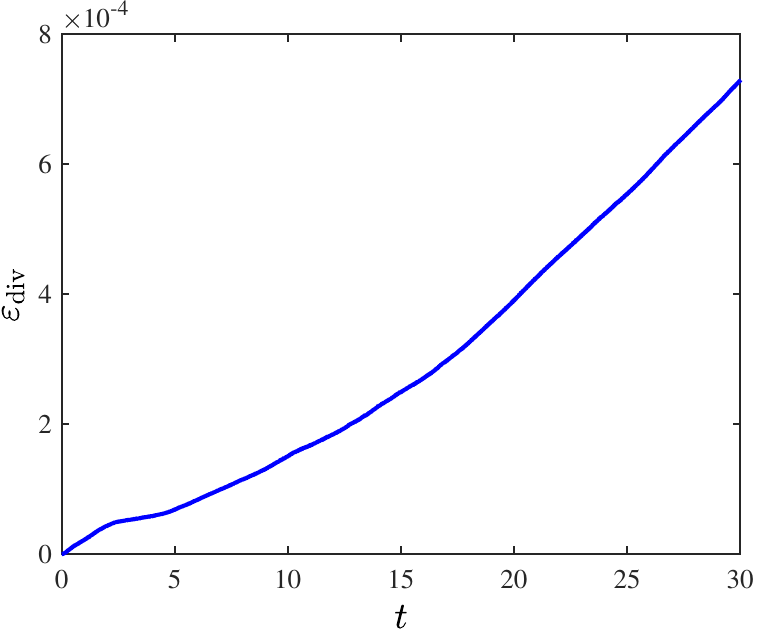}
		\end{subfigure}
		\qquad
		\begin{subfigure}{0.45\textwidth}
			\includegraphics[width=\textwidth]{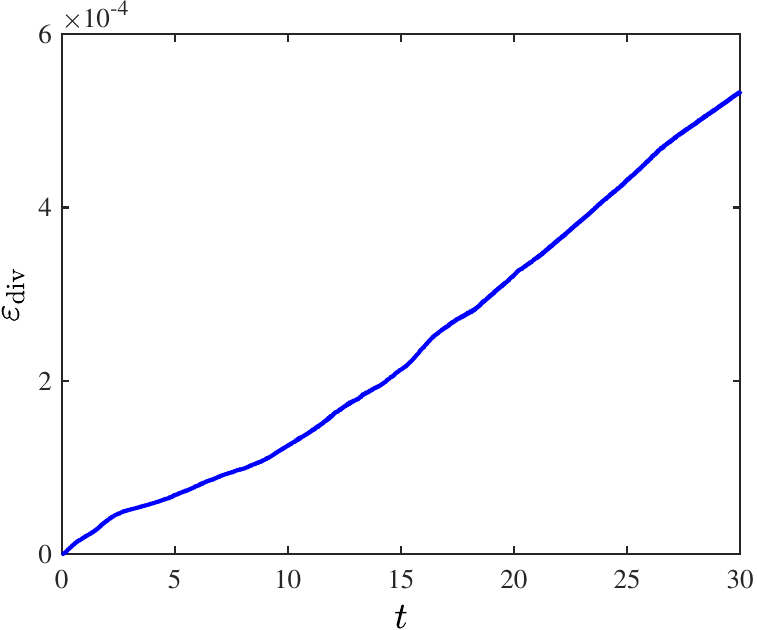}
		\end{subfigure}
		
		\caption{Time evolution of the divergence error $\varepsilon_{\rm div}$ for \Cref{Ex:Jet}. Left: ideal EOS \eqref{EOS:IDEOS}; right: RC-EOS \eqref{EOS:RCEOS}.
		}
		\label{fig:Jet_Error}
	\end{figure} 
\end{expl}

We consider two different EOSs, namely, the ideal EOS \eqref{EOS:IDEOS} and the RC-EOS \eqref{EOS:RCEOS}. \Cref{fig:Ex-Jet-2,fig:Ex-Jet-3} display the pressure logarithm $p$ and the magnitude of the magnetic field $|\bm{B}|$, at $t=10$, 20, and $30$, obtained using our BP locally DF CDG method. Our numerical results for the ideal EOS \eqref{EOS:IDEOS} are comparable to those computed in \cite{WuShu2020NumMath}. Furthermore, the results for the RC-EOS \eqref{EOS:RCEOS} also show high resolution for the flow structures. In \Cref{fig:Jet_Error}, we also present the time evolution of the global relative divergence errors for the ideal EOS \eqref{EOS:IDEOS} and the RC-EOS \eqref{EOS:RCEOS}, respectively. In both cases, the divergence errors remain at a low level of $\mathcal{O}(10^{-4})$. The proposed BP locally DF CDG method demonstrates high robustness and stability for this challenging problem.

\section{Conclusions}\label{Conclusions}

In numerical simulations of relativistic magnetohydrodynamics (RMHD), 
it is crucial yet challenging to preserve the physical bounds of fluid velocity, density, and pressure, while also maintaining the divergence-free (DF) constraint of the magnetic field. The difficulties are primarily due to the strong nonlinearity of the RMHD system and the intricate influence of divergence errors on the bound-preserving (BP) property. 
In this paper, we have designed robust, uniformly high-order, central discontinuous Galerkin (CDG) schemes that are provably BP, locally DF, and applicable to a general equation of state (EOS) for RMHD. For 1D RMHD, the standard CDG method is exactly DF, and we have proven its BP property under a condition achievable through the BP limiter. 
For multidimensional RMHD, we have shown that the BP property of the standard CDG method is closely tied to a discrete DF condition on overlapping meshes. This presents a significant challenge, as this condition is globally coupled across all cells, making it incompatible with the standard local scaling BP limiter.
 To address this issue, we have devised novel CDG schemes based on suitable locally DF discretization of a modified RMHD system---the relativistic counterpart to Godunov's symmetrizable system for non-relativistic MHD---with the inclusion of additional source terms. Our careful discretization of these source terms precisely eliminates the impact of divergence errors at cell interfaces on the BP property. Consequently, the BP property of our new CDG schemes is influenced solely by a local discrete DF condition, ensured by our locally DF CDG discretization.
 We have conducted comprehensive and rigorous analyses of the BP property for our CDG schemes, based on technical estimates within the geometric quasilinearization (GQL) framework. Our analyses have led to the establishment of a theoretical link between BP and discrete DF properties on overlapping meshes.
Additionally, we have introduced a new 2D cell average decomposition on overlapping meshes, which requires fewer internal nodes and results in a milder theoretical BP CFL condition, thereby enhancing the efficiency of the 2D BP CDG method. 
The remarkable robustness and effectiveness of the proposed schemes have been demonstrated through various benchmark RMHD tests with different EOSs, including challenging ultra-relativistic blasts and jets in strongly magnetized scenarios.



\bibliographystyle{model1-num-names}%
\bibliography{references_article,references_supp2}

\end{document}